\documentclass[a4paper,11pt,reqno]{amsart}

\usepackage{amsmath}
\usepackage{amssymb}
\usepackage{times}
\usepackage{macros_for_trees}
\usepackage{bbm}
\usepackage{parskip}
\usepackage{mathtools}
\usepackage{tikz}
\usepackage{longtable}
\usepackage{booktabs}
\allowdisplaybreaks

\usepackage{color}

\usepackage[colorlinks=true, pdfstartview=FitV, linkcolor=blue, citecolor=blue, urlcolor=blue,pagebackref=false]{hyperref}

\let\emptyset \undefined
\newsymbol\emptyset    203F

\theoremstyle{plain}
\newtheorem{theorem}{Theorem}[section]
\newtheorem{corollary}[theorem]{Corollary}
\newtheorem{lemma}[theorem]{Lemma}
\newtheorem{proposition}[theorem]{Proposition}

\newtheorem{definition}[theorem]{Definition}
\newtheorem{assumption}[theorem]{Assumption}

\theoremstyle{remark}
\newtheorem{remark}[theorem]{Remark}
\newtheorem{example}[theorem]{Example}

\numberwithin{equation}{section}

\numberwithin{table}{section}

\newcommand{\R}{\mathbb{R}}

\newcommand{\X}{\mathbf{X}}
\newcommand{\Y}{\mathbb{X}}
\newcommand{\x}{\mathbf{X}}
\newcommand{\Z}{\mathbb{Z}}

\newcommand{\Ff}{\mathcal{F}}

\newcommand{\Ii}{\mathcal{I}}
\newcommand{\Tt}{\mathcal{T}}
\newcommand{\Ww}{\mathcal{W}}
\newcommand{\Nn}{\mathcal{N}}

\newcommand{\prodtree}{\mathcal{Q}}
\newcommand{\tNn}{\widetilde{\Nn}}
\newcommand{\Pp}{\mathrm{Poly}}

\newcommand{\numone}{m_{\one}}
\newcommand{\numpolyi}{m_{\mathbf{x}_{i}}}
\newcommand{\numpoly}{m_{\mathbf{x}}}
\newcommand{\numnoise}{m_{\scriptscriptstyle \Xi}}
\newcommand{\num}{m}
\newcommand{\numedge}{m_{e}}

\newcommand{\Tr}{\Tt_{r}}

\newcommand{\Tl}{\Tt_{l}}

\newcommand{\Tlm}{\Tt_{l,-}}
\newcommand{\Ttp}{\Tt_{+}}

\newcommand{\Tp}{\Tt_{+}}
\newcommand{\Trec}{\mathcal{T}^{\rec}}

\newcommand{\wTl}{\widehat{\Tl}}
\newcommand{\wTr}{\widehat{\Tr}}

\def\Vec{\mathrm{Vec}}
\def\Alg{\mathrm{Alg}}

\def\id{\mathrm{Id}}

\def\rec{\mathrm{cen}}

\def\locprod{\mathbb{X}}
\def\path{\mathbb{X}}

\renewcommand{\leq}{\leqslant}
\renewcommand{\geq}{\geqslant}
\newcommand{\les}{\lesssim}
\newcommand{\one}{\mathbf{1}}
\newcommand{\hol}{H\"older }

\newcommand{\mNn}{\mathring{\Nn}}
\newcommand{\mWw}{\mathring{\Ww}}

\newcommand{\heat}{(\partial_t-\Delta)}

\begin{document}

\title[Bounds for $\Phi^4$]
{A priori bounds for the $\Phi^4$ equation in the full sub-critical regime}

\author{Ajay Chandra}
\address{Ajay Chandra, Imperial College London
}
\email{a.chandra@imperial.ac.uk}

\author{Augustin Moinat}
\address{Augustin Moinat, Imperial College London
}
\email{a.moinat@imperial.ac.uk}

\author{Hendrik Weber}
\address{Hendrik Weber, University of Bath
}
\email{h.weber@bath.ac.uk}

\thanks{AC
gratefully acknowledges financial support from the Leverhulme Trust via an Early Career
Fellowship, ECF-2017-226.
HW is supported by the Royal Society through the University Research Fellowship UF140187.
We also thank Andris Gerasimovics for feedback on an earlier draft of this article.
}



 \begin{abstract}
We derive a priori bounds for the $\Phi^4$ equation in the full sub-critical regime using Hairer's theory of regularity structures.
The equation is formally given by
 \begin{equation}\label{phi4abstract}
\heat\phi=-\phi^3 + \infty \phi +\xi, \tag{$\star$}
\end{equation}
where the term $+\infty \phi$ represents infinite terms that have to be removed in a 
renormalisation 
procedure. We emulate fractional dimensions $d<4$ by adjusting the regularity of the noise term $\xi$,  
choosing $\xi \in C^{-3+\delta}$.  Our main result states that if $\phi$ satisfies this equation on a space-time 
cylinder $D= (0,1) \times \{ |x| \leq 1 \}$, then away from the boundary $\partial D$ the solution  $\phi$ can be bounded 
in terms of a finite number of explicit polynomial expressions in $\xi$.  The bound holds uniformly over all 
possible choices of boundary data for $\phi$ and thus relies crucially  on  the super-linear damping effect 
of the non-linear term $-\phi^3$. 

A key part of our analysis consists of an appropriate re-formulation of the theory of regularity structures 
in the specific context of \eqref{phi4abstract}, which allows to couple the small scale control one obtains from this 
theory  with a suitable large scale  argument. 

Along the way we make several new observations and simplifications: we reduce the number of objects required 
with respect to Hairer's work. Instead of a model $(\Pi_x)_x$ and the family of translation operators $(\Gamma_{x,y})_{x,y}$ 
we work with just a single object $(\path_{x, y})$ which acts on itself for translations, very much in the spirit 
of Gubinelli's theory of branched rough paths.  
Furthermore, we show that in the specific context of \eqref{phi4abstract}  the hierarchy of continuity conditions 
which constitute Hairer's definition of a  \emph{modelled distribution} can be reduced to the single continuity condition
on the ``coefficient on the constant level''. 
 \end{abstract}


\maketitle

\date\today
\setcounter{tocdepth}{1}

\tableofcontents

\section{Introduction}
\label{s:Intro}

The theory of regularity structures was introduced in Hairer's groundbreaking work \cite{hairer2014theory} 
and has since been developed into an impressive machinery  \cite{MR3935036,2016arXiv161208138C,2017arXiv171110239B} that systematically yields existence 
and uniqueness results for a whole range of singular stochastic partial differential equations from mathematical physics.
Examples include the KPZ equation \cite{MartinKPZ,MartinPeter}, the multiplicative stochastic heat equation \cite{WongZakai}, as well as
  reversible Markovian dynamics for the Euclidean $\Phi^4$ theory in three dimensions \cite{hairer2014theory}, in ``fractional dimension $d<4$'' \cite{2017arXiv171110239B},
 for the Sine-Gordon model \cite{SineGordon,2018arXiv180802594C}, for 
 the Brownian loop measure measure on a manifold \cite{String} and for the $d = 3$ Yang-Mills theory 
 \cite{CCHSPrep}.

 A serious limitation of this theory so far is that these existence and uniqueness results only hold for a short time, 
 and this existence time typically  depends on the specific realisation of the random noise term in the equation. 
 Most applications are 
 furthermore limited to a compact spatial domain such as a torus.  
 The reason for this limitation is that the whole machinery is set up as the solution theory for a mild formulation in terms of a  fixed-point problem, 
 and that specific features of the non-linearity, such as damping 
 effects or conserved quantities,  are not taken into account.  
 With this method, global-in-time solutions can only  be obtained in special situations, e.g.  if all non-linear terms are globally Lipschitz  \cite{Cyril} 
 or if  extra information on an invariant measure is available \cite{DPD,MR3785597}.  

This article is part of a programme to derive a priori bounds within the regularity structures framework 
in order to go beyond short time existence and compact spatial domains. 
We focus on the $\Phi^4$ dynamics which are formally given by the stochastic reaction diffusion equation
\begin{equation}\label{phi4}
\heat\phi=-\phi^3+\xi,
\end{equation}
where $\xi$ is a Gaussian space-time white noise over $\R \times \R^{d}$. 
A priori bounds for this equation have recently been derived by several groups for the two dimensional case $d=2$  \cite{mourrat2017global,tsatsoulis2018}
and the more difficult case $d=3$ \cite{Mourrat3D,AlbeverioKusuoka,GH,GubinelliHofmanova2,2018arXiv181105764M}.
In this article we obtain bounds throughout the entire sub-critical regime, formally dealing with all ``fractional dimensions'' up to (but excluding) the critical 
dimension $d=4$. 
Here we follow the convention of \cite{2017arXiv171110239B}
to emulate fractional dimensions $d<4$  by adjusting the regularity assumption on $\xi$, and assuming that it can only be controlled in a distributional 
parabolic Besov-H\"older space of regularity $-3+\delta$ for an arbitrarily small $\delta >0$.
Connecting back to the $\Phi^4$ dynamics driven by space-time white noise, $\delta = 0-$ mimics the scaling of the equation with $d=4$ and $\delta = 1/2-$ gives us back equation with $d=3$.

Our analysis is based the method developed in the $d=3$ context in  \cite{2018arXiv181105764M} where it was shown that if $\phi$ solves 
 \eqref{phi4}, on a parabolic cylinder, say on   

 \begin{equation}\label{parabolic_cylinder}
 D=(0,1)\times \{|x|<1\},
 \end{equation}
where $|x| = \max\{ |x_1| , \ldots, |x_d| \} $ denotes the supremum norm on $\R^d$, then it can be bounded on any smaller cylinder $D_R = (R^2,1) \times \{|x|<1-R\}$ only in terms of the distance $R$ and the realisation of 
$\xi$ when restricted to a small neighbourhood of $D$. This bound holds uniformly over all possible choices for  $\phi$
  on the parabolic boundary of $D$, thus leveraging on the full strength of the non-linear damping term $-\phi^3$.
This makes the  estimate  extremely useful when studying the large scale behaviour of solutions, because given a realisation of the noise, any local function of the solution (e.g. a localised norm or testing against a compactly supported test-function)
  can be controlled in a\textit{ completely deterministic way }by objects that depend on the noise realisation on a compact set, without taking the behaviour of solution elsewhere into account.
 
Our main result is the exact analogue valid throughout the entire sub-critical regime.
\begin{theorem}[Theorem~\ref{th:main theorem} below]
\label{thm:main-simplified}
Let $\delta>0$ and let  $\xi$ be of regularity $-3+\delta$.
 Let $\{ \locprod_{\bullet} \tau  \colon \tau \in \Ww, \Nn \} $ be a local product lift of $\xi$. 
Let $\phi$ solve 
\begin{equation}\label{eq: renormalised equation in intro}
\heat\phi=-\phi^{\circ_{\locprod}3}+\xi, \qquad \text{ on } D
\end{equation}
where $\phi^{\circ_{\locprod}3}$ refers to the renormalised cube sub-ordinate to $\locprod$. 

Then $v := \phi - \sum_{\tau \in \Ww} \locprod_{\bullet}  \Ii(\tau)$ satisfies 
\[
 \|v\|_{D_R}\leq C\max\Big\{\frac1R,[\locprod;\tau]^\frac1{\delta \numnoise(\tau)}, \tau \in \Nn\cup\Ww\Big\},
\]
uniform in the choice of the local product, where $\| \bullet \|_{D_R}$ denotes  the supremum norm on $D_R$
\end{theorem}
Here the ``local product'' denotes a finite number of functions/distributions $\locprod_{\bullet} \tau$, each of which is constructed as a polynomial of degree $\numnoise(\tau)$,
see Section~\ref{s:PE}. 
Local products  correspond to \emph{models} \cite[Definition 2.17]{hairer2014theory} in the theory of regularity structures, but we use them slightly 
differently and hence prefer a different name and notation. 
The functions / distributions  $\locprod_{\bullet} \tau$ are indexed by \emph{two} sets $\Ww$ and $\Nn$. 
Here $\Ww$ contains the most irregular terms so that after their subtraction the remainder $v$ can be bounded in a positive regularity norm.
The semi-norms $[\locprod;\tau]$ are defined in \eqref{equ: order of order seminorm} and they  correspond to the order bounds on models  \cite[Equation (2.15)]{hairer2014theory}. 
The renormalised cube sub-ordinate to a local product is defined in Definition \ref{def: renormalised product}. 
This notion corresponds exactly to the reconstruction with respect to a model / local product $\locprod_{\bullet}$ 
of the abstract cube in \cite{hairer2014theory}. 
When analysing an equation within the theory of regularity structures,  one proceeds in two steps: in a \emph{probabilistic step} a finite number 
of terms in a perturbative approximation of the solution are constructed
- these terms are referred to as the \emph{model} already mentioned above.
 The terms in this expansion are just as irregular as $\phi$ itself, and their construction a priori poses the same problem to define 
 non-linear operations. However, they are given by an explicit polynomial expression of the Gaussian noise $\xi$ and they can thus be analysed 
 using stochastic moment calculations. It turns out that in many situations
 the necessary non-linear operations on the model can be defined despite the low regularity due to stochastic cancellations.
However, this construction does require renormalisation with infinite counterterms.

  In the second \emph{analytic step}  the remainder of the perturbative expansion is bounded.
 The key criterion for this procedure to work is a scaling condition,  which is called \emph{sub-criticality} 
 in \cite{hairer2014theory},  and which corresponds  to \emph{super-renormalisability} in Quantum Field Theory. 
This condition states, roughly speaking, that on small scales the 
 non-linearity is dominated by the interplay of noise and linear operator. As mentioned above,  in the context of \eqref{phi4} 
 this condition is satisfied precisely for $\xi \in C^{-3+\delta}$ if $\delta>0$.
Sub-criticality ensures that only finitely many terms in the expansion are needed to yield a remainder that is small 
enough to close the argument.

It is important to note that while subcriticality ensures that the number of terms needed in the model is finite, 
this number can still be extremely large and typically diverges as one approaches the threshold of criticality. A substantial part of 
  \cite{MR3935036,2016arXiv161208138C,2017arXiv171110239B} is thus dedicated to a systematic treatment of the 
algebraic relations between all of these terms and their interaction, as well as the effect of renormalising the model on the original equation. 
The local-in-time well posedness theory for \eqref{phi4} for all sub-critical  $\xi \in C^{-3+\delta}$, which was developed in  \cite{2017arXiv171110239B},
was one of the first applications of the complete  algebraic machinery.

The three dimensional analysis  in \cite{2018arXiv181105764M}  was  the first work that used  regularity structures to derive a priori bounds.
All  of the previous works mentioned above  \cite{Mourrat3D,AlbeverioKusuoka,GH,GubinelliHofmanova2} were set in an alternative technical framework, the theory of paracontrolled distributions  developed in \cite{Gubi}.  
These two theories are closely related: both  theories were developed to understand the small scale behaviour of solutions to singular SPDEs,
  and both separate the probabilistic  construction of finitely many terms in a perturbative expansion from the deterministic 
analysis of a remainder. 
Furthermore, many technical arguments in the theory of regularity structures have a close correspondent in the paracontrolled distribution framework.
However,  up to now paracontrolled distributions have only been used to deal with equations 
 with a moderate number of terms in the expansion (e.g. \eqref{phi4} for $d \leq 3$ \cite{CatellierChouk} or the KPZ equation \cite{KPZrelaoded}).
 Despite efforts by several groups  (see e.g. \cite{BailleulBernicot,BailleulHoshino}) this method has not yet been extended to allow for expansions 
 of arbitrary order.  Thus  for some of the most interesting models  mentioned above, e.g. the Sine-Gordon model for $\beta^{2}$ just below $8\pi$,
 the reversible dynamics for the Brownian loop measure on a manifold, the three-dimensional Yang-Mills theory, or the $\Phi^4$  model close to 
 critical dimension considered here, even a short time existence and uniqueness theory is currently out of reach of 
 the theory of paracontrolled distributions.
The analysis in \cite{2018arXiv181105764M} was based on the idea that the large and small scale behaviour of 
 solutions to singular SPDEs should be controlled by completely different arguments: for large scales the irregularity of $\xi$ is essentially irrelevant and 
 bounds follow from the strong damping effect of the non-linearity $-\phi^3$.   The small scale behaviour is controlled using  the smoothing properties of the heat operator.
 This philosophy was implemented by working with a suitably regularised equation which could be treated with a maximum principle and by bounding the error 
 due to the regularisation using regularity structures. 

However, this analysis did not make use of the full strength of the regularity structure machinery. 
In fact, the  three-dimensional $\Phi^4$ equation is by now
considered as  one of the easiest examples  of a singular SPDE, because the model only contains a moderate  number of terms, 
only  five different non-trivial products  need to be defined using stochastic arguments and only  two different divergencies must  be renormalised. 
The interplay of these procedures is not too complex and no advanced algebraic machinery is needed to deal with it. 
Instead, in \cite{2018arXiv181105764M}
the few algebraic relations were simply treated explicitly ``by hand".  
The main contribution of  the present article is thus to implement a similar argument when the number of terms in the model is unbounded, thus combining the analytic ideas from 
\cite{2018arXiv181105764M} with the algebraic techniques  \cite{MR3935036,2017arXiv171110239B}. 
For this it turns out to be most convenient to 
re-develop the necessary elements of the theory of regularity structures in the specific context of \eqref{phi4}, leading to bounds that are tailor-made as input for the large-scale analysis.

Along the way, we encounter various serious simplifications and new observations which are interesting in their own right:

\begin{itemize}
\item

 As already hinted at in Theorem~\ref{thm:main-simplified} we make systematic use of the
 ``generalised Da Prato-Debussche trick" \cite{DPD,2017arXiv171110239B}. 
 This means that instead of working with $\phi$ directly we remove the most irregular terms of the expansion
 leading to a function valued remainder. This was already done in   \cite{2017arXiv171110239B} but only in order to avoid a technical problem concerning the initial conditions. 
 For us the remainder $v$ is the more natural object, observing that for all values of $\delta>0$ it solves an equation of the form
  \begin{align}
\label{e:intro-remainder} 
\heat v&=-v^3 + \ldots
\end{align}
where $\ldots$ represents a large number of terms (the number diverges as $\delta \downarrow 0$) which involve renormalised products of either $1$, $v$ or $v^2$ with various irregular ``stochastic 
terms''.
For each $\delta>0$,  $v$ takes values in a positive regularity H\"older norm (i.e. it is a function) and so an un-renormalised damping term $-v^3$ appears on the right hand side. 
Of course,  the H\"older regularity of $v$ is not enough to control many of the products appearing in  $\ldots$, and a local expansion of $v$ is required to control 
these terms.
However, we are able to show that for each fixed value of $\delta$ all of these terms are ultimately of lower order relative to $\heat v$ and $v^3$.
 \item
 One of the key ideas in the theory of regularity structures is \emph{positive renormalisation} and the notion of \emph{order}. 
 Most of the analysis works with a re-centered version of the functions / distributions from the model, which depends on a
 \emph{base-point} as well as the running argument - these objects are denoted by the $\Pi_x$.  
 A good description of their behaviour under a change of base-point is 
 key to the analysis, and in Hairer's framework this is accomplished by working with a family of translation operators $\Gamma_{x,y}$.

There is a close relationship between these $\Pi_{x}$ and $\Gamma_{x,y}$ maps and some generic identities relating them were found in \cite{Bruned2018}.
Our observation is that - at least in the context of Equation~\eqref{e:intro-remainder} - most of the matrix entries for $\Gamma_{x,y}$ 
coincide with entries for $\Pi_x$ evaluated at $y$. 
Therefore we can work with just a single object $\locprod_{\bullet}$ (corresponding to $\mathbf{\Pi}$ in \cite{hairer2014theory})
 and its re-centered version $\path_{\bullet,\bullet}$ that acts on itself for translation. 
 
With this choice our framework is highly reminiscent of Gubinelli's work on
branched rough paths \cite{gubinelli2010ramification}, the only real difference being the introduction of some (linear) polynomials, first order derivatives, and the flexibility to allow for non-canonical products. 
\item Inspired by  \cite{Otto2018,otto2018parabolic} and just as in \cite{2018arXiv181105764M} we derive Schauder estimates
using Safonov's kernel free method (popularised in \cite{krylov1996lectures}), thus working 
directly with the PDE rather than transforming into an integral equation.
This is more convenient for our analysis, because these bounds give more flexibility e.g. when localising functions by restricting 
them to certain sets.
\item 
As in \cite{hairer2014theory} we use the model / local product to build a local approximation of  $v$ around any base-point $x$. 
This takes  the form
\[
v(y) \approx \sum_{\tau \in \Nn} \Upsilon_x(\tau) \path_{y,x} \Ii(\tau),
\]
with a well-controlled error as $y$ approaches $x$.
In order to use this local expansion to control non-linearities two key analytic ingredients are needed: the first is the order bound discussed above, and 
the second is a suitable continuity condition on the coefficients $\Upsilon_x(\tau)$. In \cite{hairer2014theory} these conditions are encoded in a family of 
model-dependent semi-norms, which make up the core of the definition of a \emph{modelled distribution} \cite[Definition 3.1]{hairer2014theory}. 
It turns out however,  that the coefficients $\Upsilon_x(\tau)$ that appear in the expansion of the solution $v$ are far from generic: up to signs and 
combinatoric factors they can only be either $1$, $v(x)$, $v(x)^2$, or $v_{\X}(x)$ (a generalised derivative of $v$). 
Furthermore, there is a simple criterion 
(Lemma~\ref{Upsilon-Lemma}) to see which of these is associated to a given tree $\tau$. 
This fact was already observed in \cite{2017arXiv171110239B} and was called \emph{coherence} there. 
Here we observe that the various semi-norms in the definition of a modelled distribution are in fact all truncations of the single continuity 
condition on the first coefficient $\Upsilon(\one) = v$. This observation is key for our analysis, as this particular semi-norm is precisely the output of 
our Schauder Lemma.
\item Our deterministic theory more cleanly separates the issues of positive and negative renormalisation in the context of \eqref{phi4}. Indeed, we can derive a priori bounds under extremely general assumptions on the specific 
choice of the local product $\locprod$ which seems quite a bit larger and simpler than the space of models given in  \cite{MR3935036}.
The key information contained in $\locprod$ is how certain a priori unbounded products should 
be interpreted. 
Our definition of a local product allows for these interpretations to be completely arbitrary! 
We can then always define 
the re-centered version of $\locprod$ (or path) and the only assumption where the various functions interact is in the assumption that these re-centered 
products satisfy the correct order bound. 

We do however include a Section~\ref{sec:useful class of local products} in which we introduce a specific class of local products for which the renormalised product $\phi^{\circ_{\locprod} 3}$ appearing in \eqref{eq: renormalised equation in intro} is still a local polynomial in $\phi$ and its spatial derivatives.  
Our approach in this section is to apply a recursive negative renormalisation that commutes with positive renormalisation, similar to \cite{Bruned2018}. 
Finally, the class of local products described in Section~\ref{sec:useful class of local products} also contains local products that correspond to the BPHZ renormalised model \cite{MR3935036,2016arXiv161208138C}.
\end{itemize}

\subsection{Conventions}\label{ss:conventions}

Throughout we will work with functions / distributions defined on (subsets of) $\R \times \R^{d}$ for an arbitrary $d \geq 1$.
We measure regularity  in H\"older-type norms that reflect the parabolic scaling of the heat operator. For example, we set
\begin{equation}\label{eq:parab_metric}
d((t,x),(\bar{t}, \bar{x}))=\max\Big\{\sqrt{|t-\bar{t}|},|x-\bar{x}|\Big\},
\end{equation}
and for  $\alpha \in (0,1)$, we define the (local) \hol semi-norm $[\bullet]_\alpha$  accordingly as
\begin{equation}\label{e:def-hol1}
[u]_\alpha:=\sup_{d(z,\bar{z})<1}\frac{|u(z)-u(\bar{z})|}{d(z,\bar{z})^\alpha}.
\end{equation}
Distributional norms, i.e. H\"older type norms for negative regularity $\alpha<0$ play an important role throughout. These norms are defined in terms of 
the behaviour under convolution with rescaled versions of a  suitable compactly supported kernel $\Psi$. For example, for $\alpha<0$ we set
\begin{equation}\label{shauder ou1}
[\xi]_{\alpha}=\sup_{L\leq 1} \Big\| (\xi)_L \Big\| L^{-\alpha},
\end{equation}
where $\|  \bullet \| $ refers to the supremum norm on $\R \times \R^{d}$ and the operator $(\bullet)_L$ denotes convolution with a compactly supported smooth kernel $\Psi_L(x)=L^{-d-2}\Psi  \Big( \frac{x_0}{L^2}, \frac{\bar{x}}{L} \Big) $, where $x=(x_0,\bar{x})$. Just as in \cite{2018arXiv181105764M}
we work with a specific choice of $\Psi$, but this is only relevant in the proof of the Reconstruction 
Theorem, Lemma~\ref{Reconstruction}. These topics are discussed in detail in Appendix~\ref{ss:RL}.

In the case of space-time white noise, the quantity in \eqref{shauder ou1} is almost surely not finite, but our analysis only depends on the noise locally: a space-time cut-off can be introduced.
Throughout the paper we also make the \emph{qualitative} assumption that $\xi$ and all other  
functions are smooth. This corresponds to introducing a  regularisation of the noise term $\xi$ (e.g. by convolution 
with a regularising kernel at some small scale - in field theory this is called an ultra-violet cut-off). 
This is very convenient,  because it allows  to avoid unnecessary discussions about how certain objects have to be interpreted and 
in which sense partial differential equations hold. We stress however that our main result, Theorem~\ref{th:main theorem},  is a bound only in terms of those low-regularity norms 
(Definition~\ref{def:seminorm}) which can be controlled when the regularisation is removed in the renormalization procedure.
 Even though all functions involved are smooth, we will freely use the term "distribution" to refer to a smooth function that can only be bounded 
in a negative regularity norm.

\section{Overview}
\label{s:ov}

As stated in the introduction a large part of our analysis consists of a suitable re-formulation of elements of the theory of regularity structures. 
The key notions we require are \emph{local products}, the \emph{renormalized product} sub-ordinate to a local product, as well as the relevant norms 
that permit us to bound these renormalized products. We start our exposition with an overview over these notions and how they are interconnected. 
The exposition in this section is meant to be intuitive and rather ``bottom up''. The actual analysis begins in the subsequent Section \ref{s:PE}.

\subsection{Subcriticality:}
\label{ss:2-1}
 The starting point of our analysis is a simple scaling consideration: assume $\phi$ solves 
\begin{equation}\label{outl-1}
\heat \phi = - \phi^3 + \xi,
\end{equation}
for  $\xi \in C^{-3+\delta}$.
Schauder theory suggests that the solution $\phi$ is not better than $C^{-1 + \delta}$.
In this low regularity class no  bounds on $\phi^3$ are available, but as we will see below the notion of product we will work with
has the property that  negative regularities add under multiplication. Therefore
we will obtain 
a control on (a renormalised version of) $\phi^3$  as a distribution  in $C^{-3+3\delta}$.
Despite this very low regularity, for $\delta>0$, the term $\phi^3$ is still more regular than the noise $\xi$. This observation 
 is the core of Hairer's notion of sub-criticality (see \cite[Assumption 8.3]{hairer2014theory}) and suggests that the small-scale behaviour of $\phi$ and $\phi^3$
 can ultimately be well understood by building a perturbative expansion based on the linearised equation.

\subsection{Trees:}
We follow Hairer's convention to index the terms in this expansion by a set of trees.  
This is not only a convenient notation that allows to organise which term corresponds 
to which operation, but also allows for an efficient  organisation of the relations between these terms. 
We furthermore follow the convention to view trees as abstract symbols which form the basis of a  finite-dimensional vector space. 
The trees are built from a generator symbol $\Xi$ (which represents the noise $\xi$ and graphically are the leaves of the tree) followed by applying the operator $\Ii(\cdot)$ (which represents to solving the heat equation and graphically corresponds to the edges of the tree) and taking products of trees (which represents to some choice of point-wise product and graphically corresponds to joining two trees at their root). 
To carry out the localisation procedure, discussed  in Section~\ref{ss:2-4} below, along with $\Xi$, additional generators $\{\one,\X_{1},\dots,\X_{d}\}$ are used in our construction of trees. 

We associate concrete 
meaning to trees via an operator $\locprod_{\bullet}$ which we call a ``local product'', see Definition \ref{def:locprod}.
Even though this may seem somewhat bulky initially, 
it  turns out to be extremely convenient as the concrete definition of  $\locprod_{\bullet}$ on the same tree may change during the renormalisation procedure and because, 
 the local product also appears in a \emph{centered form} denoted by $\path_{\bullet, \bullet}$, see Section~\ref{ss:Overview_PositiveRen} below.

\subsection{Subtracting the most irregular terms:}\label{subsec:subtracting_most_irreg}

 The first step of our analysis consists of subtracting a finite number of terms from $\phi$ to obtain a remainder $v$ which is regular enough to be bounded  
 in a positive H\"older norm.  The regularity analysis in Section~\ref{ss:2-1} suggests that the 
 regularity of $\phi$ can be improved by  removing $\xi$ from the right hand side of  \eqref{outl-1}. We introduce the first graph,  $\Ii(\Xi)$ 
 or graphically  $\<1_black>$, and impose that $\locprod_{\bullet}$ acts on this symbol yielding a function that satisfies  
\begin{equation}\label{outl-2}
\heat  \locprod_{\bullet}  \<1_black>=  \xi.
\end{equation}
We  set $\tilde{v}:= \phi-  \locprod_{\bullet} \<1_black> $ so that $\tilde{v}$ solves 
\begin{equation}\label{outl-3}
\heat \tilde{v}=-\phi^3  = - \big(  \tilde{v}^3 + 3  \tilde{v}^2  \locprod_{\bullet}  \<1_black>+ 3  \tilde{v}  \big( \locprod_{\bullet}  \<1_black> \big)^2 +  \big( \locprod_{\bullet}  \<1_black> \big)^3  \big).
\end{equation}

Of course the problem of controlling the cube of a distribution of regularity $-1+\delta$ has not disappeared, but instead of $\phi^3$ one now has to 
control $( \locprod_{\bullet}  \<1_black> )^3 $ and $( \locprod_{\bullet}  \<1_black> )^2 $.
At this point one has to make use of the fact that $\locprod_{\bullet}  \<1_black>$ is known much more explicitly than the solution $\phi$, and can thus be 
analysed using explicit covariance calculations. We do not discuss these calculations here, but rather view these products as part of the given data: we introduce 
two additional symbols $ \Ii(\Xi) \Ii(\Xi) \Ii(\Xi)$ or graphically  $ \<3_black>$, and  similarly $ \Ii(\Xi) \Ii(\Xi) $ or $ \<2_black>$ and assume that 
$ \locprod $ acts on these additional symbols yielding distributions which are controlled in $C^{-3+3\delta}$ and $C^{-2+2\delta}$. We stress that 
only the control on these norms enters the proof of our a priori bound, and no relation to $\locprod_{\bullet}  \<1_black>$ needs to be imposed (see however Section~\ref{sec:useful class of local products} below).
Instead of \eqref{outl-3} we thus consider
\begin{equation}\label{outl-4}
\heat \tilde{v}=- \big(  \tilde{v}^3 + 3  \tilde{v}^2  \locprod_{\bullet}  \<1_black>+ 3  \tilde{v}   \locprod_{\bullet}  \<2_black> +  \locprod_{\bullet}  \<3_black> \big)  .
\end{equation}
Note that the most irregular term on the right hand side is  $ \locprod_{\bullet}  \<3_black> \in C^{-3+3\delta} $ so that we can expect $\tilde{v} \in C^{-1+3\delta}$
i.e. we have gained   $2 \delta$ differentiability with respect to  $\phi$.
For $\delta >\frac13$ (which corresponds to dimensions ``$d< 3 \frac13$'')  $\tilde{v}$ is thus controlled in a positive order H\"older norm. For smaller $\delta$ we proceed to subtract an additional term to again remove the 
most irregular term from the right hand side as above. We define a new symbol $\Ii ( \Ii(\Xi) \Ii(\Xi) \Ii(\Xi))$ or graphically $\<K*3_black>$, postulate that $\locprod_{\bullet} $ acts on this symbol 
yielding a distribution which solves 
\begin{equation}\label{outl-5}
\heat  \locprod_{\bullet}  \<K*3_black>= \locprod_{\bullet} \<3_black>,
\end{equation}
and define a new remainder $\tilde{\tilde{v}} :=  \tilde{v} + \locprod_{\bullet}  \<K*3_black> = \phi  -  \locprod_{\bullet} \<1_black>   + \locprod_{\bullet}  \<K*3_black> $ which takes values in $C^{-1+5 \delta}$. 
In general, for any $\delta>0$  we denote by $\Ww$ the set of trees of \emph{order}  $< -2$ (for these trees, order is the same as the regularity of the local product on this tree. Below, in Section~\ref{ss:Overview_PositiveRen} we will encounter additional trees for which these notions differ)
 and define 
\[
v := \phi - \sum_{w \in \Ww} (-1)^{\frac{\num(w)-1}{2}} \locprod_{\bullet} \Ii (\tau),
\]
where $\num(w)$ denotes the number of ``leaves'' of the tree $w$. Then $v$ takes values in a H\"older space of positive regularity. The remainder equation then turns into  
\begin{align}\label{outl-6}
\heat v &=-v^3\\
&-3\sum_{w\in\Ww} (-1)^{\frac{\num(w)-1}{2}}  v^2  \path_\bullet \Ii(w)\notag\\
&-3\sum_{w_1,w_2\in\Ww} (-1)^{\frac{\num(w_1) + \num(w_2)-2}{2}} v   \path_\bullet ( \Ii(w_1)  \Ii(w_2)) \notag\\
&-\sum_{\substack{w_1,w_2,w_3\in\Ww  \\ \Ii(w_1) \Ii (w_2) \Ii(w_3) \notin \Ww }}(-1)^{\frac{\num(w_1) + \num(w_2) + \num(w_3)-3}{2}} \path_\bullet (\Ii(w_1)  \Ii(w_2) \Ii(w_3)) .\notag
\end{align}
%
We stress that the structure of this equation is always the same in the sense that $ \heat v =-v^3$ is perturbed by a large number of irregular terms (the number actually diverges as $\delta \to 0$). 
Bounding these irregular terms forces us to introduce additional trees as we will see below, 
 but ultimately we will  show that all of these terms are of lower order with respect to $ \heat v =-v^3$.
\subsection{Iterated freezing of coefficients}
\label{ss:2-4}
We now discuss the remainder equation \eqref{outl-6} in more detail, writing it as  
\begin{align}
\heat v &=-v^3 - 3 v^2  \path_\bullet \<1_black> - 3 v \path_\bullet \<2_black>  -
\Upsilon(\tau_{0})     \path_\bullet \tau_{0} - \ldots,
\label{outl-7}
\end{align}
where we are isolating the most irregular terms in each of the three sums appearing on the right hand side of \eqref{outl-6}.
The most irregular term in the sum on the second line of \eqref{outl-6} is $-3v^2  \path_\bullet \<1_black>$ and the most irregular 
term in the third line is $-3 v \path_\bullet \<2_black> $. 
For the last line, the precise form of the most irregular term depends on $\delta$ and there could be multiple terms of the 
same low regularity. Here we just keep track of one of them, simply denote it by $ \path_\bullet \tau_{0}$ and also leave the combinatorial prefactor $\Upsilon (\tau_{0})$ 
implicit.  We remark that $ \path_\bullet \tau_{0}$ is always a distribution of regularity $C^{-2+\kappa}$ for some $\kappa \in (0, 2 \delta)$.  
To simplify the 
exposition we disregard all of the (many) additional terms hidden in the ellipses $\ldots$ for the moment.  

We recall the standard multiplicative inequality 
\[
\| f g \|_{C^{-\beta}} \lesssim \| f \|_{C^\alpha}  \|g   \|_{C^{-\beta}} 
\]
for $\alpha, \beta >0$ which holds if and only if $\alpha - \beta >0$. In view of the regularity 
  $\path_\bullet \<2_black>  \in C^{-2 +\delta}$ we would thus require $v \in C^\gamma$ for $\gamma > 2-2\delta$
  in order to have a classical interpretation of the product $v \path_\bullet \<2_black>$ on the right hand side of \eqref{outl-7}. 
   Unfortunately, $v$ is much more irregular: by Schauder theory   we can only expect $v$ to be of class $C^\kappa$.

The solution to overcome this difficulty presented in  \cite{hairer2014theory} amounts to an ``iterated freezing of coefficient'' procedure to obtain a good local description of $v$ around a fixed 
base-point: we  fix a space-time point $x$ and rewrite the third  (and most important) term on the right hand side of \eqref{outl-7} as 
\begin{equation}\label{outl-7A}
v \locprod_{\bullet}\<2_black> = v(x)\locprod_{\bullet} \<2_black> + (v- v(x))\locprod_{\bullet} \<2_black>
\end{equation}
and use this to rewrite the equation \eqref{outl-7} as
\begin{align}
\notag
&\heat  (v   + 3 v(x)  \path_{\bullet} \<K*2_black>     +  \Upsilon(\tau_{0})     \path_\bullet \Ii( \tau_{0})    ) \\
& \qquad \qquad =-v^3 - 3 v^2  \path_\bullet \<1_black> - 3 (v- v(x))\locprod_{\bullet} \<2_black>  - \ldots
\label{outl-8}
\end{align}
where we have introduced new symbols $\<K*2_black> $ and $ \Ii( \tau_{0}) $ and postulated that $\locprod$ acts on these symbols 
to yield a solution of the inhomogeneous heat equation with right hand sides $\locprod_{\bullet} \<2_black>$ and $\locprod_{\bullet} \tau_{0}$.
The worst term on the right hand side is now $\locprod_{\bullet} \<2_black>  $ so that the left hand side can at best be of regularity $2\delta$. 
However, near the base-point we can use the smallness of the pre-factor $|v(\bullet) - v(x)| \lesssim [v]_\kappa d(\bullet, x)^\kappa$ to get the better estimate
\begin{align}
\notag
| U(y,x)| &:=
\Big|  v(y)  -\big(  v(x)  - 3 v(x) \path_{y,x} \<K*2_black>       -  \Upsilon(\tau_{0})     \path_{y,x} \Ii( \tau_{0})  \big) \Big| \\
\label{outl-8b}
&   \lesssim d(y,x)^{2\delta + \kappa},
\end{align}
 where have used the  short-hand notation
\begin{align}
\notag
\path_{y,x} \<K*2_black> &:=   \path_{y} \<K*2_black> - \path_{x} \<K*2_black> \\
\label{outl-8A}
\path_{y,x} \Ii( \tau_{0}) & := \path_y \Ii( \tau_{0}) -  \path_{x} \Ii( \tau_{0}) .
\end{align}

This bound in turn can now be used to get yet a better approximation in \eqref{outl-7A}: we write
\begin{align}
\notag
&(v(y)- v(x))\locprod_{y} \<2_black> \\
&=  \big(U(y,x)   - 3 v(x)     \path_{y,x} \<K*2_black>    - \Upsilon(\tau_{0})     \path_{y,x} \Ii( \tau_{0})  \big)  \locprod_{y} \<2_black> .
\label{outl-9}
\end{align}
At this point two additional non-classical products appear in the second and third term on the right hand side,  and as before they are treated as part of the assumed data:
we introduce two additional symbols $  \<2K*2_black>  $ and $\Ii(\tau_{0})\Ii(\Xi) \Ii(\Xi) $ and assume that $\locprod$ acts on these symbols yielding  distributions 
which we interpret as playing the roles of the products  $ \path_{y} \<K*2_black>    \locprod_{y} \<2_black> $ and $ \path_y \Ii( \tau_{0})\locprod_{y} \<2_black>$.
Similarly, we introduce the base-point dependent versions as 
\begin{align}
\notag
 \path_{y,x} \<2K*2_black> &:= \locprod_{y} \<2K*2_black>  - \locprod_{x} \<K*2_black>  \locprod_{y} \<2_black> \\
 \label{outl-9Z}
 \path_{y,x} \Ii(\tau_{0})\Ii(\Xi) \Ii(\Xi)  & :=   \path_{y} \Ii(\tau_{0}) \Ii(\Xi) \Ii(\Xi)    -  \locprod_{x} \Ii(\tau_{0})   \locprod_{y} \<2_black> ,
\end{align}
so that \eqref{outl-9} becomes re-interpreted as
\begin{align}
\notag
&(v(y)- v(x))\locprod_{y} \<2_black> \\
&= U(y,x)  \locprod_{y} \<2_black> 
 - 3 v(x)  \path_{y,x} \<2K*2_black>   
  - \Upsilon(\tau_{0})\path_{y,x} \Ii(\tau_{0})\Ii(\Xi) \Ii(\Xi) .
\label{outl-9A}
\end{align}
The last two terms on the right hand side can now again be moved to the left hand side of the equation suggesting that near $x$ we can improve the approximation \eqref{outl-8} of  $v(y)$ by considering  
\begin{align}
\tilde{U}(y, x) &:= U(y,x) +  3 v(x)     \path_{y,x} \<K*2K*2_black>   
+  \Upsilon(\tau_{0}) \path_{y,x} \Ii( \Ii( \tau_{0}) \Ii(\Xi) \Ii(\Xi))
	  \label{outl-10}
\end{align}
where
\begin{align}
\notag
  \path_{y,x} \<K*2K*2_black> & :=   \path_{y} \<K*2K*2_black>  - \path_{x} \<K*2K*2_black>  - \path_{x} \<K*2_black>  \big( \locprod_{y} \<K*2_black>  -    \locprod_{x} \<K*2_black> \big)  \\
  \notag
  \path_{y,x} 	\path_y\Ii( \Ii( \tau_{0}) \Ii(\Xi) \Ii(\Xi))  &:= 
    	\path_y\Ii( \Ii( \tau_{0}) \Ii(\Xi) \Ii(\Xi))  -  \path_{x}\Ii( \Ii( \tau_{0}) \Ii(\Xi) \Ii(\Xi))   \\  
	\label{outl-10A}
	&	\qquad \qquad -  \path_{x} \Ii( \tau_{0})  \big(  \locprod_{y} \<K*2_black>  -  \locprod_{x} \<K*2_black>  \big) .
\end{align}
with the improved estimate $|\tilde{U}(y,x)| \lesssim d(y,x)^{4\delta +\kappa}$, thus gaining another  $2\delta$ with respect to $U(y,x)$. 

The whole procedure  can now be iterated:  in each step an improved approximation of $v$ is plugged into the product $v \locprod_{\bullet}\<2_black>$ which in turn yields an 
even better local approximation of $v$ near $x$. 
At some point additional terms have to be added:
\begin{itemize}
\item In order to get a local description of order $>1$,  ``generalized derivatives'' $v_{\X_i}$  of $v$ appears, i.e. a term $\sum_{i=1}^d v_{\X_i}(x)(y_i - x_i)$ has to be included.
\item The term $-3 v^2 \locprod \<1_black>$ on the right hand side of the remainder equation \eqref{outl-7} has regularity $-1+\delta$, so once one wishes to push the expansion 
of $v$ to a level $>1+\delta$, one also has to ``freeze the coefficient'' $v^2$, i.e. write 
\begin{align*}
v^2 \locprod  \<1_black> &= v^2(x) \locprod  \<1_black> + (v^2- v^2(x)) \locprod  \<1_black> \\
&= v^2(x) \locprod  \<1_black> +  2 v \big(      - 3 v(x) \path_{\bullet,x} \<K*2_black>       -  \Upsilon(\tau_{0})     \path_{\bullet,x} \Ii( \tau_{0})   \big)  \locprod  \<1_black> + \ldots
\end{align*}
leading to additional terms on the left hand side.

\item Of course, the various terms which were hidden in   $\ldots$ in \eqref{outl-7} above have to be treated in a similar way leading to (many) additional terms in the local description of $v$.
\end{itemize} 
Ultimately, we iterate this scheme until we have a local description of order $\gamma> 2-2\delta$, corresponding to the regularity required classically to define $v \locprod_{\bullet} \<2_black>$. 
\subsection{Renormalised products}
\label{ss:2-8}
The previous discussion thus suggests that we have a Taylor-like approximation of $v$ near the base-point $x$ 
\begin{align}\label{outl-11}
v(y) \approx v(x) + \sum_{i=1}^d v_{\X_i}(x) \cdot (y_i-x_i) +       \sum_{\tau \in\mNn  } \Upsilon_{x}(\tau)  \path_{y,x}\Ii(\tau)
\end{align}
for coefficients $\Upsilon_{x}$ and with an error that is controlled by $\lesssim d(x,y)^{\gamma}$. 
Here $\mNn$ denotes the set of trees appearing in the recursive construction described above.
We unify our notation by also writing the first two terms with ``trees''
and set 
\begin{align*}
\path_{y,x} \Ii(\one) &= 1  &    \path_{y,x} \Ii(\X_i) &= y_i  - x_i\\
\Upsilon_{x}(\one) &= v(x)    &  \Upsilon_{x}(\X_i ) &= v_{\X_i} (x) ,
\end{align*}
thus permitting to rewrite \eqref{outl-11} as
\begin{align}\label{outl-12}
v(y) \approx       \sum_{\tau \in  \Nn} \Upsilon_{x}(\tau)  \path_{y,x}\Ii(\tau),
\end{align}
where $\Nn =\mNn\cup \{\one,\X_{1},\dots,\X_{d}\}$.
 
Of course, up to now our reasoning was purely formal, because it relied on all of the ad hoc products of singular distributions that were simply postulated along the way. We now turn this 
formal reasoning into a \emph{definition} of the  products sub-ordinate to the choices in the local product $\locprod$. More precisely, we \emph{define} renormalized products such as
\begin{align}
\notag
v  \circ_{\path} \path_\bullet \<2_black> (x)  &:=   \sum_{\tau \in \Nn } \Upsilon_{x}(\tau)  \path_{x,x}  (\Ii(\tau)\<2_black> ), \\
\label{outl-13}
v \circ_{\locprod} v   \circ_{\locprod} \path_\bullet \<1_black> (x)  &:=   \sum_{\tau_1, \tau_2 \in \Nn }    \Upsilon_{x}(\tau_1)  \Upsilon_{x}(\tau_2)   \path_{x,x}  (\Ii(\tau_1)\Ii(\tau_1) \<1_black> ).
\end{align}
Our main a priori bound in Theorem~\ref{th:main theorem} holds for the remainder equation interpreted in this sense, under very general assumptions on the local product $\locprod$. However, under 
these very general assumptions it is not clear (and in general not true) that the renormalized products are in any simple relationship to the \emph{usual} products. In Section \ref{sec:useful class of local products}
we discuss a class of local products for which the renormalized products can be re-expressed as explicit local functionals of \emph{usual} products. 
In particular, for those local products we always have
\[
\phi^{\circ_{\locprod} 3}(y)    = \phi^3(y) - a \phi^2(y) - b \phi(y) - c - \sum_{i=1}^d d_i \partial_i  \phi(y),
\]
for real parameters $a,b,c, d_i$.  This class of local products contains the examples that can actually be treated using probabilistic arguments.

\subsection{Positive renormalisation and order}
\label{ss:Overview_PositiveRen}

One of the key insights of the theory of regularity structures is that the renormalized products defined above can be controlled quantitatively in a process called renormalization,
and the most important ingredient for that process are the definitions of  suitable notions of regularity / continuity for the local products $\locprod$ and the coefficients $\Upsilon$.
We start with the local products.

The base-point dependant or centered versions of the local product, $\path_{y,x}$ that appear naturally in the expansions above (e.g. in \eqref{outl-8A}, \eqref{outl-9Z}, \eqref{outl-10A})
are in fact much more than a notational convenience.  The key observation is that their behaviour as the running argument $y$ approaches the base-point $x$ is well controlled in  the 
 so-called \emph{order bound}.  For $ \path_{y,x} \<K*2_black> $  defined in \eqref{outl-8A} we have
\begin{equation}
\label{outl-11A}
 |\path_{y,x} \<K*2_black>| =  | \path_{y} \<K*2_black> - \path_{x} \<K*2_black> | \lesssim d(y,x)^{2\delta},
\end{equation}
which amounts to the H\"older regularity of $\locprod_{\bullet}  \<K*2_black>$. The order bounds become more interesting in more complex examples: for  $\path_{y,x} \<K*2K*2_black>$
defined in  \eqref{outl-10A} we have
\begin{align}
\label{outl-21}
 |  \path_{y,x} \<K*2K*2_black> | & :=  
  \big| \path_{y} \<K*2K*2_black>  - \path_{x} \<K*2K*2_black>  - \path_{x} \<K*2_black>  \big( \locprod_{y} \<K*2_black>  -    \locprod_{x} \<K*2_black> \big) \big| \lesssim d(y,x)^{4\delta}.
\end{align}
The remarkable observation here is that the function $\path_{y} \<K*2K*2_black>  $ is itself only of regularity $2\delta$, so that this estimate expresses that the second term 
$ - \path_{x} \<K*2_black>  \big( \locprod_{y} \<K*2_black>  -    \locprod_{x} \<K*2_black> \big)$ exactly compensates the roughest small scale fluctuations. The exponent $4\delta$ is 
defined as the \emph{order} of the tree $ \<K*2K*2_black>$ simply denoted by $| \<K*2K*2_black>|$. Analogously, for the tree  $\path_{y,x} \<2K*2_black> $ defined in \eqref{outl-9Z} we have 
the order $| \<2K*2_black>| = -2+ 4\delta$ exceding the \emph{regularity} of the distribution $\path_{y} \<2K*2_black>$ which is only $-2+2\delta$, the same as the regularity of $\path_{y} \<2_black>$. As these quantities are distributions the \emph{order} bound now 
has to be interpreted by testing against the rescaled kernel $\Psi_T$  
\begin{align}
\label{outl-101}
\Big| \int \Psi_T(y-x) \path_{y,x}\<2K*2_black> \enskip  dy\Big| \lesssim T^{-2+4\delta}.
\end{align}
This notion of order of trees has the crucial property that it behaves additive under multiplication - just like the regularity of distributions discussed above. 
This property is what guarantees that for sub-critical equations the number of trees with order below any fixed threshold is always finite.

\subsection{Change of base-point}
As sketched in the discussion above, the base-point dependent centered local products  $\path_{\bullet, \bullet}$ are defined recursively from the un-centered ones. 
For what follows, a good algebraic framework to describe the centering operation and the behaviour under the change of base-point is required. 
It turns out that both operations can be formulated conveniently using a combinatorial  operation called the \emph{coproduct} $\Delta$ (note that this $\Delta$ has nothing to do with the Laplace operator, it will always be clear from the context which object we refer to). This coproduct associates to each tree a finite sum of couples $(\tau^{(1)}, \tau^{(2)})$
where $\tau^{(1)}$ is a tree and $\tau^{(2)}$ is a finite list of trees. Equivalently, the coproduct can be seen as a linear map 
\[
\Delta:\Ttp \rightarrow \Vec(\Ttp) \otimes \Alg(\Trec),
\]
where $\Ttp$ and $\Trec$ are sets of tree that we will define later (see Section~\ref{ss:order} for the former and Section~\ref{subsec:coproduct} for the latter) and $\Vec$ and $\Alg$ denote the vector space and the free non-commutative unital algebra generated by a set, respectively.
This coproduct is defined recursively reflecting exactly the recursive positive renormalization described above in Section~\ref{ss:2-4}. For example
\begin{align*}
\Delta  \<K*2_black> &:= \<K*2_black>  \otimes \Ii(\one)  -  \Ii(\one) \otimes  \<K*2_black> \\
\Delta  \<K*2K*2_black> & :=   \<K*2K*2_black> \otimes  \Ii(\one) + \Ii(\one) \otimes  \<K*2K*2_black>  +  \<K*2_black> \otimes \<K*2_black>   
\end{align*}
so that for example the first definitions of \eqref{outl-8A} and \eqref{outl-10} turn into
\begin{align*}
\path_{y,x} \<K*2_black> &:=   (\locprod_{y}  \otimes \path_{x}^\rec) \Delta  \<K*2_black>  \\
\path_{y,x}   \<K*2K*2_black> &:=   (\locprod_{y}  \otimes \path_{x}^\rec) \Delta    \<K*2K*2_black>, 
\end{align*} 
i.e. the different terms in the coproduct correspond to the different terms appearing in the positive renormalization, and for each pair $\tau^{1} \otimes \tau^{2}$, 
the first tree $\tau^{1}$ corresponds to the ``running variable $y$'' and $\tau^{2}$ to the value of the base-point. The coefficients $\path_{x}^\rec$ are also defined 
recursively to match this definition e.g 
\begin{align*}
 \path_{x}^\rec  \Ii(\one) &= 1 &   \path_{x}^\rec  \<K*2_black> &= - \path_{x} \<K*2_black> &    \path_{x}^\rec   \<K*2K*2_black> &=-  \path_{x}  \<K*2K*2_black> +  \path_{x}  \<K*2_black>   \path_{x} \<K*2_black>.
\end{align*} 
This way of codifying the relation between the centered and un-centered local products is useful, e.g. when analysing the effect of the renormalization procedure (Section~\ref{sec:useful class of local products}) but even more importantly 
they give an efficient way to describe how $\path_{y,x} $ behave under change of base-point. It turns out that we obtain the remarkable formula for all $\tau\in\Tt$
\[
\path_{y,z}(\tau)  = (\path_{y,\bar{z}}   \otimes \path_{\bar{z},z})  \Delta \tau ,
\]
i.e. the centered object $\path_{y,z}$ acts on itself as a translation operator!

\subsection{Continuity of coefficients}
\label{ss:2-8b}
With this algebraic formalism in hand, we are now ready to describe the correct continuity condition on the coefficients. This continuity condition is formulated in terms of the concrete realisation of the local product, in that 
an ``adjoint'' of the translation operator appears. In order to formulate it, we introduce another combinatorial notation  $C_+(\bar{\tau},\tau)$, which is defined recursively to ensure that
\[
\Delta \Ii(\tau)=\sum_{\bar{\tau}\in \Nn\cup \Ww }\Ii(\bar{\tau})\otimes C_+(\bar{\tau},\tau).
\]
We argue below that the correct family of semi-norms for the various coefficients $\Upsilon(\tau)$ is given by  
\begin{equation}\label{outl-18}
\sup_{x,y}  \frac{1}{d(x, y)^{\gamma-|\tau|} }  \Big|  \Upsilon_{x}(\tau)  -\sum_{ \substack{\bar{\tau} \in \Nn \\ |\bar{\tau} | < \gamma }}  \Upsilon_{y}(\bar{\tau}) \path_{y,x} C_+(\tau, \bar{\tau})  \Big|.
\end{equation}
The Reconstruction Theorem (see Lemma~\ref{Reconstruction} for our formulation)  implies that the renormalized products~\eqref{outl-13} can be controlled in terms of the semi-norms~\eqref{outl-18} 
and the order bounds (e.g. \eqref{outl-101}).
Reconstruction takes as input the whole family of semi-norms~\eqref{outl-18}, but it turns out that in our case, it suffices to deal with a single semi-norm on the coefficients: the coefficients $\Upsilon_{x}(\tau)$
that appear in the recursive freezing of coefficients described in Section~\ref{ss:2-4} are far from arbitrary. 
 It is very easy to see that (up to combinatorial 
coefficients and signs) the only possible coefficients we encounter are $v$, $v^2$, $v_\X$,  and $1$. It then turns out that all of the semi-norms~\eqref{outl-18} are in fact truncations of the single continuity condition on the 
coefficient $v$ itself. This semi-norm can then be easily seen to be  
\begin{equation}\label{outl-19}
\sup_{x,y}  \frac{1}{d(x, y)^{\gamma  }}  \Big|  v(x)  - \sum_{ \substack{\tau \in \Nn \\ | \tau | < \gamma }}  \Upsilon_{y}(\tau) \path_{y,x} \Ii(\tau)  \Big|,
\end{equation}
which measures precisely the quality of the approximation~\eqref{outl-12} at the starting point of this discussion.

\subsection{Outline of paper}
\label{ss:2-9}

A large part of this article is concerned with providing the details of the arguments sketched above in a streamlined ''top-down'' way: The set 
of trees, their order and local products are defined in Section~\ref{s:PE}, while Section~\ref{sec:coproduct} provides a systematic treatment of combinatorial 
properties of the coproduct. Positive renormalization of local products is discussed in Section~\ref{s:positive_renorm}, while Section~\ref{s:CP} contains 
the detailed discussion of the coefficients $\Upsilon$ sketched above in Section~\ref{ss:2-8}. The renormalized products in the spirit of \eqref{outl-13} are 
defined in  Section \ref{sec: actual products}. As already announced above Section~\ref{sec:useful class of local products} contains the discussion of a special class 
of local products, for which the renormalized product can be expressed in a simple form. 
The actual large-scale analysis only starts in Section~\ref{s:MR}, where the main result is announced. This section also contains a detailed outline of the strategy 
of proof.  The various technical Lemmas that constitute this proof can then be found in Section~\ref{s:PMT}. 
Finally, we provide two appendices in which some known results are collected: Appendix~\ref{ss:RL} discusses norms on spaces of distributions in the context of the 
reconstruction theorem. Appendix~\ref{ss:LSL} collects different variants of classical Schauder estimates.
%
%
%
\section{Tree expansion and local products}
\label{s:PE}
The objects we refer to as trees will be built from
\begin{itemize} 
\item a set of generators $\{\mathbf{1},\X_{1},\dots, \X_{d},\Xi\}$, which can be thought of as the set of possible types of leaf nodes of the tree
\item applications of an operator $\Ii$, which can be thought of as edges 
\item A tree product, which joins edges $\Ii$ at a common new node. 
\end{itemize}
As an example, we have  
\[\Xi = \bullet,\ 
\Ii[\Xi]^{2} = \<2b_black>,\ \Ii(\Xi)\Ii(\Ii(\Xi))\Ii(\Xi)=\<2K*1b_black>,\ 
\]
\[
\Ii(\Xi)\Ii(\Ii(\Xi)^2\Ii(\x_i))\Ii(\Xi)=\<2K*2Xi_black>,\text{ and } \Ii(\one)\Ii(\Ii(\Xi)^3)\Ii(\one)=\<K*3b11_black>.
\]
In particular, when drawing our trees pictorially we decorate the leaf nodes with a $\bullet$ for an instance of $\Xi$, $0$ for an instance of $\one$, and $j \in \{1,\dots,d\}$ for an instance of $\X_{j}$. 
Notice that we do not decorate internal (non-leaf) nodes and have the root node at the bottom. 

Our tree product is non-commutative which in terms of our pictures means that we distinguish between the ways a tree can be embedded in the plane. 
For example, the following trees are treated as distinct from the trees above:
\[
\Ii(\Xi)\Ii(\Ii(\Xi)\Ii(\x_i)\Ii(\Xi))\Ii(\Xi)=\<2K*2XiAS_black> \text{ and } \Ii(\Ii(\Xi)^3)\Ii(\one)\Ii(\one)=\<K*3bAS_black>\;.
\] 
\begin{remark}
We work with a non-commutative tree product only to simplify combinatorial arguments. Whenever we map trees over to concrete functions and or distributions this mapping will treat identically any two trees that coincide when one imposes commutativity of the tree product.
\end{remark}
We say a tree $\tau$ is planted if it is of the form $\tau = \Ii(\tilde{\tau})$ for some other tree $\tilde{\tau}$, some examples would be: 
\[
\<K*3b_black>,\; \<K*(1K*3K*z)b_black>,\text{ and }\<K*(1K*3K*3)b_black>.
\]
We take a moment to describe the intuition behind these trees.
The symbol $\Xi$ will represent the driving noise, we will often call nodes of type $\Xi$ noise leaves/nodes. 
Regarding the operator $\Ii$, when applied to trees different from  $\{\one,\X_{1},\dots,\X_{d}\}$, $\Ii$ will represent solving the heat equation, that is 
\[
\textnormal{``}(\partial_{t} - \Delta) \Ii(\tau) = \tau \textnormal{''}\;.
\]
However, we think of the trees as algebraic objects so such an equation is only given here as a mnemonic and will be made concrete when we associate functions to trees in Section~\ref{sec: local products}. 

The symbols $\{\one,\X_{1},\dots,\X_{d}\}$ themselves will not correspond to any analytic object, but the trees $\{\Ii(\one),\Ii(\X_{1}),\dots,\Ii(\X_{d})\}$ will play the role of the classical monomials, that is $\Ii(\one)$ corresponds to $1$ and $\Ii(\X_{j})$ corresponds to the monomial $z_{j}$. 

We define $\wTr$ to be the smallest set of trees containing  $\{\mathbf{1},\X_{1},\dots, \X_{d},\Xi\} \subset \wTr$ and such that for every $\tau_{1}, \tau_{2}, \tau_{3} \in \wTr$ one also has $\Ii(\tau_{1})\Ii(\tau_{2})\Ii(\tau_{3}) \in \wTr$. 
The trees in $\wTr \setminus \{\one, \X_{1},\dots,\X_{d}\}$ will be used to write expansions for the right hand side of \eqref{phi4}. 
We remark that non-leaf nodes in $\wTr$ have three offspring, for instance $\Ii(\Xi)^{2}\Ii(\one) \in \wTr$ but $\<2_black> = \Ii(\Xi)^{2} \not \in \wTr$. However, the three different permutations of  $\Ii(\Xi)^{2}\Ii(\one)$ will play the role of that $\<2_black>$ did in expressions like \eqref{outl-4}, and as an example of how this simplifies our combinatorics we remark that this allows us to forget about the  ``$3$'' that appears in \eqref{outl-4}. 

We also define a corresponding set of planted trees $\wTl
=
\{
\Ii(\tau): \tau \in \wTr\} $. 
The planted trees in $\wTl$ will be used to describe an expansion of the solution $\phi$ to \eqref{phi4}. 

At certain points of our argument the roles of the planted trees of $\wTl$ and the unplanted trees of $\wTr$ will be quite different. 
For this reason we will reserve the use of the greek letter $\tau$ (and $\bar{\tau}$, $\tilde{\tau}$, etc.) for elements of $\wTr$. 
If we want to refer to a tree that could belong to either $\wTl$ or $\wTr$ we will use the greek letter $\sigma$. 

\subsection{The order of a tree and truncation}
\label{ss:order}

We give a recursive definition of the order $|\cdot|$ on $\wTr \cup \wTl$ as follows. 
Given $\Ii(\tau) \in \wTl$ we set $|\Ii(\tau)| = |\tau| +2 $. 
Given $\tau \in \wTr$ we set
\[
|\tau|
:=
\begin{cases}
-2, & \tau=\one\;, \\
-1, & \tau=\X_{i}\;, i \in \{1,\dots,d\}\\
-3+\delta, & \tau=\Xi\;,\\
\sum_{i=1}^{3} |\Ii(\tau_{i})| = 
6+\sum_{i=1}^3|\tau_i|, & \tau=\Ii(\tau_1)\Ii(\tau_2)\Ii(\tau_3)\;.
\end{cases}
\]
The values of $-2$ and $-1$ for homogeneities of the trees $\one$ and $\X_{i}$ may seem a bit odd but this is just due to the convention that it is $\Ii(\one)$ and $\Ii(\X_{i})$ that actually play the role of the classical monomials and we want $|\Ii(\one)| = 0$ and $|\Ii(\X_{i})| = 1$.
We find that treating the classical monomials as planted trees makes our combinatorial arguments and various inductive proofs cleaner. 

We now restrict the set of trees we work with and organise them into various sets.   
We define the following subsets of $\wTr$:
\begin{equation*}\label{sets_of_trees_introduced}
\begin{split}
\Pp :=& \{\X_{1},\dots,\X_{d},\one\}\;,\\
\Ww :=& \{\tau\in\wTr:\ |\tau| < -2\}\;,\ \mWw := \Ww \setminus \{\Xi\},\\
\Nn :=& \{\tau\in\wTr \;,\; -2 \leqslant |\tau| \leqslant 0\},\ \mNn := \Nn \setminus \Pp\;.
\end{split}
\end{equation*}
As a mnemonic, 
$\mWw$ (resp $\mNn$), is the set of those trees in $\Ww$ (resp $\mNn$) which are themselves the tree product of three planted trees. 

\begin{assumption}\label{assump: orders away from integers}
For the rest of the paper, we treat $\delta > 0$ as fixed, and assume, without loss of generality for the purposes of our main theorem, that $\delta$ has been chosen so that $\{|\tau| : \tau \in \Ww \cup \mNn\}$ does not contain any integers.
\end{assumption}
 
It will be helpful throughout this article to have notation for counting the number of occurrences of a certain leaf type in a tree. 
We define the functions $\numone,\ \numpolyi,\  \numnoise: \wTr \rightarrow \Z_{\ge 0}$ which count, on any given tree, the number of occurrences of $\one$, $\X_{i}$ and $\Xi$ as leaves in the tree. 

We also set $\numpoly = \sum_{i=1}^{d} \numpolyi$ for the function that returns the total number of $\{\X_{1},\cdots,\X_{d}\}$ leaves and $\num =\numone+\numpoly+\numnoise$ which returns the total number of leaves of the given tree.
 
One can easily check that, for $\tau\in\wTr$,
\begin{equation}\label{formula_hom}
|\tau|=-3+\numnoise(\tau)\delta+\numone(\tau)+2\numpoly(\tau)\;.
\end{equation}

We now describe the roles of the various sets defined earlier. 
As mentioned earlier, the trees of $\Pp$ will not, by themselves, play a role in our expansions.  
The set $\Ww$ consists of those trees that appear in our expansion that have the lowest orders. 
When ``subtracting the most irregular terms'' as described in Section~\ref{subsec:subtracting_most_irreg} we will be subtracting the trees of $\Ii(\Ww)$ which are all of negative order themselves. 
In particular, the trees of $\Ww$ will appear in tree expansions for the right hand side of \eqref{phi4} but will \emph{not} appear by themselves on the right hand side of the remainder equation.   
On the other hand, the trees of $\mNn$ will appear on the right hand side of expansions of both \eqref{phi4} and the remainder equation. 
We do not include $|\tau| > 0$ in $\tau \in \mNn$ since we only need to expand the right hand side of the remainder equation up to order $0$.
 
Our remainder will then be described by an expansion in terms of trees of $\Nn$ where the trees in $\Ii(\Pp)$ will come with ``generalised derivatives''. 

We have the following straightforward lemma.

\begin{lemma}\label{lemma: set of trees is finite}
The sets $\Ww$ and $\Nn$ are both finite.	
\end{lemma}
\begin{proof}
From the formula \eqref{formula_hom}, one can see that $\tau \in \Ww$ if and only if $\numone(\tau)=\numpoly(\tau)=0$ and $\numnoise(\tau)<\delta^{-1}$. 
Similarly for $\tau\in \mNn$, one has 
\[
(\numone(\tau),\numpoly(\tau))\in \{(0,0),(1,0),(2,0),(0,1)\}\;,
\] 
and $\numnoise(\tau)<\delta^{-1}(3-\numone(\tau)-2\numpoly(\tau))$.
\end{proof}
\begin{remark}
Clearly Lemma~\ref{lemma: set of trees is finite} would be false for $\delta = 0$, that is when the equation is critical.	
\end{remark}
We also have the following lemma describing the trees in $\mWw$. 
\begin{lemma}\label{lem:w-trees_all_w}
For any $\tau \in \wTr \setminus \{\Xi\}$,  $|\tau| \ge -3 + 3\delta > |\Xi| = -3 + \delta$.
Moreover, for any $w \in \mWw$, one has
\begin{equation}\label{eq:w-tree_all_w}
w = \Ii(w_{1})\Ii(w_{2})\Ii(w_{3})
\end{equation}	
where $w_{1},w_{2},w_{3} \in \Ww$. 
\end{lemma}
\begin{proof}
The first statement about $|\tau|$ is a simple consequence of \eqref{formula_hom} and the constraint that $\numnoise(\tau) \ge 3$. 
For the second statement, we write $w = \Ii(\tau_{1})\Ii(\tau_{2})\Ii(\tau_{3})$ with $\tau_{1},\tau_{2},\tau_{3} \in \wTr$ .
Then we have, by bounding the orders of $\tau_{2}$ and $\tau_{3}$ from below, 
\[
|w| = 6 + |\tau_{1}| + |\tau_{2}| + |\tau_{3}| \ge |\tau_{1}| + 2 \delta 
\]
Then, the condition that $|w| \le -2$ forces $|\tau_{1}| < -2$ so we have $\tau_{1} \in \Ww$ and clearly the same argument applies for $\tau_{2},\tau_{3}$. 
\end{proof}
We also define \label{other_sets_of_trees}
\begin{equation} 
\Tr := \Ww \cup \mNn,\ 
\Tl :=  \Ii(\Tr) \cup \Ii(\Pp), \textnormal{ and }
\Tt := \Tr \cup \Tl\;.
\end{equation}
Above, and in what follows, given $A \subset \wTr$ we write $\Ii(A) = \{\Ii(\tau):\ \tau \in A\}$.  

Our various tree expansions will be linear combinations of trees in $\Tt$ and we define tree products of such linear combinations by using linearity.
However, here we implement a truncation convention that will be in place for the rest of the paper. 
Namely, given $\sigma_{1},\sigma_{2},\sigma_{3} \in \Tl$ we enforce that if $|\sigma_{1}| + |\sigma_{2}| + |\sigma_{3}| > 0$, we enforce that
\[
\sigma_{1}  \sigma_{2} \sigma_{3} := 0\;.
\] 
In particular, with these conventions and important identity for us will be
\begin{equation}\label{eq:cube}
\Big(
\sum_{ \tau \in \Nn \cup \Ww}
\Ii(\tau)
\Big)^{3}
=
\sum_{\tau \in \mNn \cup \mWw}
\tau\;.
\end{equation}

Above, on the left, the $(\bullet)^{3}$ indicates a three-fold tree product.
\subsection{Local products}\label{sec: local products}
In this section we begin to specify how trees are mapped into analytic expressions. 
Our starting point for this will be what we call a \emph{local product} and will be denoted by $\locprod$. 
Each local product $\locprod$  should be thought of as a (minimal) description of how products of planted trees should be interpreted at a concrete level.
 
We will view local products as being defined on a relatively small set of trees and then canonically extended to all of $\Tt$ (and in the sequel, to larger sets of trees that will appear).
\begin{definition}\label{def: small class of trees}
We define $\prodtree \subset \mNn \cup \mWw$ to consist of all trees $\tau = \Ii(\tau_{1}) \Ii(\tau_{2}) \Ii(\tau_{3}) \in \mNn \cup \mWw$ satisfying the following properties:
\begin{itemize}
\item $\tau_{1},\tau_{2},\tau_{3} \not \in \{\X_{1},\dots,\X_{d}\}$. 
\item At most one of the trees $\tau_{1},\tau_{2},\tau_{3}$ is equal to $\one$.  
\end{itemize}
\end{definition}
Note that $\mWw \subset \prodtree$. 
The set $\prodtree$ includes all ``non-trivial'' products of trees, namely those corresponding to classically ill-defined products of distributions.  
Our philosophy is that once a local product $\locprod$ is specified on the noise $\Xi$ and all these non-trivial products then we are able to define all other products that appear in our analysis. 

We impose the first of the two constraints stated above because multiplication by the tree $ \Ii(\X_{i})$ corresponds to multiplication of a distribution/function by  $z_{i}$ which always well-defined - this makes it natural to enforce that this product is not deformed. 
We impose the second of the two constraints above since a tree of the form $\Ii(\one)\Ii(\tau)\Ii(\one)$ (or some permutation thereof) doesn't really represent a new non-trivial product because the factors $\Ii(\one)$ corresponds to the the classical monomial $1$. 
\begin{definition}\label{def:locprod}
A \emph{local product} is a map $\locprod: \prodtree \cup \{\Xi\} \rightarrow C^{\infty}(\R \times \R^{d})$, which we write $\tau \mapsto \locprod_{\bullet}\tau$. 

We further enforce that if $\tau,\bar{\tau} \in \prodtree$ differ from each other only due to the non-commutativity of the tree-product then $\locprod_{\bullet} \tau = \locprod_{\bullet} \bar{\tau}$, that is $\locprod$ must be insensitive to the non-commutativity of the tree product.   
\end{definition} 

\subsection{Extension of local products}\label{subsec: ext of loc prods}
We now describe how any local product $\locprod$ is extended to $\Tt$, this procedure will involve induction in $\numedge(\sigma) + \numpoly(\sigma)$ where $\numedge(\sigma)$ is the number of edges of $\sigma$. 

We start by defining, for any function $f: \R \times \R^{d} \rightarrow \R$, $(\mathcal{L}^{-1}f)$ to be the unique bounded solution $u$ of 
\begin{equation}\label{eq: cut-off heat equation}
\heat u =\rho f.
\end{equation}
where $\rho$ is a smooth cutoff function with value $1$ in a neighbourhood of $D$ and vanishes outside of $\{z;\ d(z,0)<2\}$. 

We now describe how we extend $\locprod$ to $\mNn \setminus \prodtree$.
If $\tau = \Ii(\tau_{1}) \Ii(\tau_{2}) \Ii(\tau_{3}) \in \mNn \setminus \prodtree$ then precisely one of the following conditions holds
\begin{enumerate}
\item Exactly one of the $\tau_{1}$, $\tau_{2}$, $\tau_{3}$ belong to the set $\{\X_{i}\}_{i=1}^{d}$. 
\item Two of the factors $\tau_{1},\tau_{2},\tau_{3}$ are equal to $\one$.
\end{enumerate} 
In the first case above we can assume without loss of generality that $\tau_{1} = \X_{i}$, then we set
\[
\locprod_{z} \tau = 
z_{i} 
\locprod_{z}(\Ii(\one) \Ii(\tau_{2}) \Ii(\tau_{3}))\;.
\]
In the second case above we can assume without loss of generality that $\tau_{1} = \tau_{2} = \one$, then we set
\[
\locprod_{z} \tau = (\mathcal{L}^{-1} \locprod_{\bullet} \tau_{3})(z)\;.
\]
Next we extend any local product $\locprod$ to $\Tl$ by setting, for any $\Ii(\tau) \in \Tl$, 
\begin{equation}\label{eq: admissibility extension}
\locprod_{z}\Ii(\tau) 
=
\begin{cases}
z_{i} & \textnormal{ if }\tau = \X_{i}\\
1 & \textnormal{ if }\tau = \one\\
(\mathcal{L}^{-1} \locprod_{\bullet}\tau)(z) & \textnormal{ otherwise}\;.	
\end{cases}
\end{equation}

Finally, we extend by linearity to allow $\locprod$ to act on linear combinations of elements of $\Tt$. 
Adopting the language of rough path theory and regularity structures, given smooth noise $\xi:\R \times \R^{d} \rightarrow \R$ we say a local product $\locprod$ is a \emph{lift} of $\xi$ if $\locprod_{z}\Xi = \xi(z)$. 
Without additional constraints lifts are not unique.
\begin{definition}\label{def: multiplicative local product}
We say a local product $\locprod$ is \emph{multiplicative} if, for every 
\[
\Ii(\tau_{1})\Ii(\tau_{2})\Ii(\tau_{3}) \in \prodtree\;,\]
one has
\begin{equation}\label{eq: multiplicative local product}
\locprod_{z} \Ii(\tau_{1})\Ii(\tau_{2})\Ii(\tau_{3})
=
\locprod_{z}
\Ii(\tau_{1})
\locprod_{z}
\Ii(\tau_{2})
\locprod_{z}
\Ii(\tau_{3})\;,
\end{equation}	
where on the right hand side we are using the extension of $\locprod$ to planted trees.
\end{definition}
The following lemma is then straightforward to prove. 
\begin{lemma}
Given any smooth $\xi:\R \times \R^{d} \rightarrow \R$ there is a unique multiplicative lift of $\xi$ into a local product, up to the choice of cut-off function $\rho$.
\end{lemma}
Multiplicative local products will not play a special role in our analysis but we will use them at several points to compare our solution theory for \eqref{phi4} to the classical solution theory. 
%
%
%
%
%
%
%
%
%
\section{The coproduct}
%
%
\label{sec:coproduct}
As discussed in Section~\ref{ss:2-4},  local products $\locprod_{\bullet}$ enter in our analysis in a centered form which depends on the choice of a basepoint $x$.
The construction of these centered objects is given in Section~\ref{s:positive_renorm}. 
As a preliminary step, in this section we define a combinatorial  operation on trees, called the coproduct, which  plays a central role in  this construction.
\subsection{Derivative edges}\label{ss:DE}
For trees $\sigma$ with $|\sigma| \in (1,2)$ we will need the centering procedure to generate first order Taylor expansions in spatial directions.
In order to encode these derivatives at the level of our algebraic symbols we introduce a new set of edges $\Ii^{+}_{i}, i=1...d$ and also define the sets of trees\label{more_trees_defined} 
\begin{align*}
\tNn
=
\{ \tau \in \mNn: 
-1 < |\tau| < 0\}\;, \\
\Tp 
=
\Tr \cup \Tl \cup \big\{ \Ii^{+}_{i}(\tau): 1 \le i \le d,\ \tau \in \tNn  \cup \{\X_{i}\}
\big\}
\;.
\end{align*}
Given $1 \le i \le d$ and $\tau \in \tNn \cup \{\X_{i}\}$, we also call $\Ii_{i}^{+}(\tau)$ a planted trees. 
The order of these new planted trees introduced here is given by $|\Ii_i^+(\tau)|=|\tau|+1$. 
We also adopt the shorthand that, for $1 \le i \le d$, 
\begin{equation*}
\Ii^{+}_{i}(\X_{j}) = 0 \textnormal{ for }j\not = i
\textnormal{ and }
\Ii^{+}_{i}(\tau) = 0 \textnormal{ for all }
\tau \in \mNn \setminus \tNn\;.
\end{equation*}
We emphasise that these new edges will only ever appear as the bottom edge of a planted tree. Graphically we distinguish these edges by writing an index by them.
For example,
\[
\Ii^{+}_{i}(\X_{j})=\<Ki*Xj_black>=0\text{ if }i\neq j,\quad \Ii^+_j(\Ii(\Xi)\Ii(\Xi)\Ii(\X_i))=\<Kj*3bi_black>.
\]
At an analytic level, the role of $\Ii^{+}_{i}(\X_{j})$ is the same as that of $\Ii(\one)$ but distinguishing these symbols will be important - see Remark~\ref{rem:three_ways_to_write_one}. 
\subsection{Algebras and vector spaces of trees}\label{ss:AVST}
We now give some notation for describing the codomain of our coproduct $\Delta$. 
Given a set of trees $T$ we write $\Vec(T)$ for the  vector space (over $\R$) generated by $T$.

Given a set of planted trees $T$ we write $\Alg(T)$ for the unital non-commutative algebra (again over $\R$) generated by $T$. 
We will distinguish between the tree product introduced in Section~\ref{s:PE} and the product that makes $\Alg(T)$ an algebra, calling the the latter product the ``forest product''.
While both the tree product and forest product are non-commutative, the roles they play are quite different - see Remark~\ref{rem: two products} 
 
 We will write $\cdot$ to denote the forest product when using algebraic variables for trees, that is given $\sigma, \tilde{\sigma} \in T$, we write $\sigma \cdot \tilde{\sigma}$ for the forest product of $\sigma$ and $\tilde{\sigma}$. 
 As a real vector space $\Alg(T)$ is spanned by products $\sigma_{1} \cdot \sigma_{2} \cdots \sigma_{n}$ with $\sigma_{1},\dots,\sigma_{n} \in T$. 
 We call such a product $\sigma_{1} \cdot \sigma_{2} \cdots \sigma_{n}$ a ``forest''.  
 The unit for the forest product is given by the ``empty'' forest and is denoted by $1$. 

Graphically, we will represent forest products just by drawing the corresponding planted trees side by side, for instance writing
\[
\<K*3bi_black>
\<Kj*3bi_black>
\<K*(1K*3K*3)b_black>\<K*3bi_black>\;.
\]
%
%
\subsection{Coproduct}
\label{subsec:coproduct}
%
%
%
We define another set of trees
\begin{equation}\label{eq:trec}
\Trec
:=
\Ii(\Nn) \cup \{ \Ii_{i}^{+}(\tau): \tau \in \tNn \cup \{\X_{i}\}, 1 \le i \le d\}\;.
\end{equation}
Our coproduct will be a map
\[
\Delta: \Ttp \rightarrow \Vec(\Ttp) \otimes \Alg(\Trec)\;.
\]
Our definition of $\Delta$ will be recursive.
The base cases of this recursive definition are the trees in $w\in \Ww$ and planted trees $\Ii(w),\; w\in\Ww$, and the elementary trees $\Ii(\one)$, $\Ii(\X_i)$ and $\Ii^+_i(\x_i)$: 
\begin{align}\label{eq: bascases deltaplus}
&\Delta \Ii(\one) 
=
\Ii(\one) \otimes \Ii(\one),\\
&\Delta \Ii(\X_{i})
=
\Ii(\one) \otimes \Ii(\X_{i}) 
+
\Ii(\X_{i}) \otimes \Ii_{i}^{+}(\X_{i}),\notag\\
&\Delta \Ii^{+}_{i}(\X_{i}) = \Ii^{+}_{i}(\X_{i}) \otimes \Ii^{+}_{i}(\X_{i}),\notag\\
&\Delta w= w \otimes 1,\ 
\Delta \Ii(w) = \Ii(w) \otimes 1,\ w\in \Ww.\notag
\end{align}
Note that in the last line, the $1$ appearing is the unit element in the algebra and should not be mistaken for $\one \in \Pp$.

The recursive part of our definition is then given by%
\begin{align}\label{eqs: recursive deltaplus} 
\Delta&\Ii(\tau)=\Ii(\one)\otimes\Ii(\tau)+
 \Ii(\X_{i})\otimes \Ii^{+}_{i}(\tau)\notag
+(\Ii \otimes \id)\Delta \tau,\ \tau\in  \mNn,\\
\Delta& \Ii^{+}_{i}(\tau)
=
\Ii^{+}_{i}(\X_{i}) \otimes \Ii^{+}_{i}(\tau) + (\Ii^{+}_{i} \otimes \id) \Delta \tau, \tau\in \tNn,\\
\Delta&(\Ii(\tau_1)\Ii(\tau_2)\Ii(\tau_3)) =\Delta(\Ii(\tau_1))\Delta(\Ii(\tau_2))\Delta(\Ii(\tau_2)),\ \Ii(\tau_1)\Ii(\tau_2)\Ii(\tau_3) \in \mNn,\notag
\end{align}
where, on the right hand side of the last line above we are referring to the natural product $[\Vec(\Tl) \otimes \Alg(\Trec)]^{\otimes 3} \rightarrow \Vec(\Tr) \otimes \Alg(\Trec)$. 
This product is just given by setting
\begin{equation}\label{eq:product_on_tensors}
 \otimes_{i=1}^{3} (\Ii(\tau_{i}) \otimes a_{i})
\mapsto  \Ii(\tau_{1})\Ii(\tau_{2})\Ii(\tau_{3}) \otimes (a_{1}\cdot a_{2}\cdot a_{3})  \;,
\end{equation}
and then extending by linearity.
We also make note of the fact that, in the first and second lines of \eqref{eqs: recursive deltaplus}, we are using our convention of Einstein summation - since $i$ does not appear on the left hand side then the $i$ on the right hand side is summed from $1$ to $d$.  

One can verify that it is indeed the case that $\Delta$ maps $\Tp$ into $\Vec(\Tp) \otimes \Alg(\Trec)$ by checking inductively, using  \eqref{eq: bascases deltaplus} for the bases cases
and  \eqref{eqs: recursive deltaplus} for the inductive step.
We also have the following lemma on how $\Delta$ acts on subsets of $\Tp$. 
\begin{lemma}\label{lemma:delta preserves Trec}
We have that
\begin{itemize}
\item $\Delta$ maps $\Tl$ into $\Vec(\Tl) \otimes \Alg(\Trec)$, 
\item $\Delta$ maps $\Tr$ into $\Vec(\Tr) \otimes \Alg(\Trec)$,
\item $\Delta$ maps $\mNn$ into $\Vec(\mNn) \otimes \Alg(\Trec)$, and 
\item $\Delta$ maps $\Trec$ into $\Vec(\Trec) \otimes \Alg(\Trec)$. 
\end{itemize}
\end{lemma}
\begin{proof}
The first two statements are immediate consequences of our definitions. 
We turn to proving the third statement, where we proceed by induction in the number of edges of $\tau \in \Nn$. 
The base case(s) where $\tau$ has three edges are easily verified by hand. For the inductive step, we write $\tau = \Ii(\tau_{1})\Ii(\tau_{2})\Ii(\tau_{3})$.
Now, if $\tau_{1},\tau_{2},\tau_{3} \in \Ww$ one can check that the last line of \eqref{eqs: recursive deltaplus} gives us that $\Delta \tau = \tau \otimes 1$ and we are done. 
On the other hand, if we have $\tau_{i} \in \Nn$ for some $i$ then $\Delta \Ii(\tau) \in \Vec(\Ii(\Nn)) \otimes \Alg(\Trec)$ and so we are done by combining the last line of \eqref{eqs: recursive deltaplus} with \eqref{eq:w-tree_all_w}. 

Finally, the fourth statement is immediate by inspection for planted trees $\Trec \setminus \Ii(\mNn)$ while for planted trees in $\Ii(\mNn)$ it follows from using the third statement for $\tau$ in the first line of \eqref{eqs: recursive deltaplus}. 
\end{proof}
We extend $\Delta$ to sums of trees of linearity, so that $\Delta: \Vec(\Tp) \rightarrow \Vec(\Tp) \otimes \Alg(\Trec)$. 
\begin{remark}\label{rem: two products}
The two products we have introduced on trees, the tree product and the forest product, play different roles in our framework:
the tree product represents a point-wise product of functions/distributions which may not be defined canonically. 
 Therefore we do \emph{not} enforce that local products act multiplicatively with respect to the tree product.
On the other hand, any map on trees that is applied to a forest is extended multiplicatively - we do not allow for any flexibility in how forest products are interpreted at a concrete level.
In particular, the trees in the forests that $\Delta$ produces in the right factor of its codomain are all of non-negative order and should be thought of as being associated to products of base-point dependent constants rather than a point-wise product of space-time functions/distributions. 
\end{remark}
\begin{remark}\label{rem:three_ways_to_write_one}
The coproduct   plays a central role in algebraically encoding the terms that appear in our centering procedure,
but the precise choices \eqref{eq: bascases deltaplus} and \eqref{eqs: recursive deltaplus} are also 
%
motivated by additional algebraic/combinatorial properties we will need from the coproduct, namely \eqref{eq:coassoc} and \eqref{prop_C.5} . 

The key content of \eqref{prop_C.5} is that  for any $\tau \in \Nn$ and any $\sigma \otimes \sigma_{1} \cdots \sigma_{n}$ appearing in the expansion of $\Delta \Ii(\tau)$, the precise number of $\X_{i}$ and $\one$ generators appearing in the forest $\sigma_{1} \cdots \sigma_{n}$ and in $\tau$ coincide. 
This in turn is needed for the crucial relation \eqref{eq: coherence relation for upsilon +}. 
For this reason we have to work with the different algebraic objects $\Ii(\one)$, $\Ii_{i}^{+}(\X_{i})$, and the empty forest $1$ - even though all these symbols are treated identically at an analytic level. 

For instance, one might be tempted to set $\Delta \Ii(\one) = \Ii(\one) \otimes 1$ but this would break \eqref{prop_C.5} since the number of $\one$'s in the empty forest $1$ is zero while the number in $\one$ is one. 
Similarly, one cannot include $\Ii(\X_{i}) \otimes \one$ or $\Ii(\X_{i}) \otimes \Ii(\one)$ in the expansion of $\Delta \Ii(\X_{i})$. 
Finally, we have to write the term $\Ii_{i}^{+}(\X_{i}) \otimes \Ii_{i}^{+}(\tau)$ in the second line of \eqref{eqs: recursive deltaplus} instead of say, $\Ii(\one) \otimes \Ii_{i}^{+}(\tau)$, because of our earlier choices and \eqref{eq:coassoc}. 
\end{remark}
\begin{example}
We show one pictorial example. 
The value of our parameter $\delta$ influences the definition of the set $\mNn$ and whether $\Ii^{+}_{i}(\tau)$ vanishes or not for $\tau \in \mNn$, therefore most non-trivial computations of $\Delta$ we would present are valid only for a certain range of the parameter $\delta$.

For the example we present below, we restrict to $\frac37>\delta>\frac13$ and therefore $1>|\<K*3b_black>|>0$ and $2>\Big|\<K*(1K*3K*3)b_black>\Big|=-1+7\delta>1$.
We then have, using Einstein's convention for the index $i\in \{1,...,d\}$ (when an index $i$ appears twice on one side of an equation, it means a summation over $i=1...d$)
\begin{align*}
\Delta\<K*(1K*3K*3)b_black>=\<K*one_black>\otimes\<K*(1K*3K*3)b_black>+\<K*Xi_black>\otimes \<Ki*(1K*3K*3)b_black>+\<K*(1K*zK*z)b_black>\otimes \<K*3b_black>\; \<K*3b_black>\\
+\<K*(1K*zK*3)b_black>\otimes \<K*3b_black>+\<K*(1K*3K*z)b_black>\otimes \<K*3b_black>
+\<K*(1K*3K*3)b_black>\otimes\<K*one_black>.
\end{align*}
We show now the example of an unplanted tree in the case $\delta<\frac13$ and therefore $|\<K*3b_black>|<0$. On the other hand, we always have $|\<K*3bi_black>|=1+2\delta>1$. With Einstein's convention for the index $j\in \{1,...,d\}$
\[
\Delta \<K*31b_black>=\<K*2Xb1z_black>\otimes\<K*3bi_black>+\<K*2Xb1j_black>\otimes\<Kj*3bi_black>+\<K*31b0_black>\otimes \<K*Xi_black>+\<K*31b_black>\otimes \<Ki*Xi_black>.
\]
\end{example} 

\begin{remark}
The last formula in \eqref{eqs: recursive deltaplus} for $\tau \in \mNn$ is also valid for $\tau\in\mWw$ where it is trivial. 
The first formula of \eqref{eqs: recursive deltaplus} can also be extended to $\tau \in \Nn$ if one adopts the convention that $\Delta \X_{i} = \Delta \one = 0$. 
We chose not to do this since the trees of $\Pp$ do not, by themselves,  play a role in our algebraic expansions and analysis except when they appear in a larger tree.
\end{remark}
We extend $\Delta$ to forests of planted trees by setting $\Delta 1 = 1 \otimes 1$ and, for any forest $\sigma_{1} \cdots \sigma_{n}$, $n \ge 1$, 
\[
\Delta ( \sigma_{1} \cdots \sigma_{n})
=
(\Delta \sigma_{1}) \cdots (\Delta \sigma_{n})\;,
\]
where on the right hand side we use the forest product to multiply all the factors components-wise. 
We extend to sums of forests of planted trees by additivity so that $\Delta: \Alg(\Tl) \rightarrow \Alg(\Tl) \otimes \Alg(\Trec)$ and, by Lemma~\ref{lemma:delta preserves Trec}, we also have that $\Delta$ maps  $\Alg(\Trec)$ into $\Alg(\Trec) \otimes \Alg(\Trec)$.  

While a single application of $\Delta$ will be  used for centering objects around a basepoint, we will see in Section~\ref{s:positive_renorm} that a double application of $\Delta$ 
will be used for describing the behaviour when changing this basepoint.  
This is our reason for also defining $\Delta$ on the planted trees of $\Trec \setminus \Tl$. 
It will be important below to know that the two ways of ``applying $\Delta$ twice'' agree, this is encoded in the following lemma. 
\begin{lemma}\label{lem:coassoc}
$\Delta$ satisfies a co-associativity property: for any $\sigma \in \Tt$ one has 
\begin{equation}\label{eq:coassoc}
(\Delta \otimes \id) \Delta \sigma 
=
(\id \otimes \Delta) \Delta \sigma \;,
\end{equation}	 
where both sides are seen as elements of $\Vec(\Tt) \otimes \Alg(\Trec) \otimes \Alg(\Trec)$. 
\end{lemma}
\begin{proof}
We argue by induction in the size of $\sigma$.
 The cases where $\sigma =  \Ii(\one)$, or $\Ii(\X_{i})$ are straightforward to check. 
Note that by multiplicativity of $\Delta$ with respect to the tree product it suffices to establish the inductive step for $\sigma = \Ii(\tau)$ for some $\tau \in \mNn \cup \mWw$. 
The case where $\tau \in \mWw$ is trivial so we walk through the verification of the identity when $\tau \in \mNn$.  
 On the left we have 

\begin{align*}
(\Delta \otimes \id) \Delta \Ii(\tau)
=&
(\Delta \otimes \id)
\Big[ 
\Ii(\one)\otimes\Ii(\tau)+
 \Ii(\X_{i})\otimes \Ii^{+}_{i}(\tau)
+(\Ii \otimes \id)\Delta \tau
\Big]\\
=&
\Ii(\one) \otimes \Ii(\one)
\otimes 
\Ii(\tau)
+
\Ii(\X_{i}) \otimes \Ii_{i}^{+}(\X_{i}) \otimes \Ii^{+}_{i}(\tau) \\
&+ 
\Ii(\one) \otimes \Ii(\X_{i}) \otimes  \Ii^{+}_{i}(\tau)
+
\Ii(\one) \otimes 
\big(
(\Ii \otimes  \id)  \Delta \tau
\big)\\
&+
\Ii(\X_{j})
\otimes
\big(
(\Ii_{j}^{+} \otimes \id )  \Delta \tau
\big)
+
(\Ii \otimes \id \otimes \id)
(\Delta \otimes \id) \Delta \tau\;.
\end{align*}
On the right we have 
\begin{align*}
( \id \otimes \Delta) \Delta \Ii(\tau)
=&
(\id \otimes \Delta )
\Big[ 
\Ii(\one)\otimes\Ii(\tau)+
 \Ii(\X_{i})\otimes \Ii^{+}_{i}(\tau)
+(\Ii \otimes \id)\Delta \tau
\Big]\\
=&
\Ii(\one) \otimes \Ii(\one)
\otimes 
\Ii(\tau)
+
\Ii(\one)
\otimes
\Ii(\X_{j})
\otimes
\Ii^{+}_{j}(\tau)
\\
&+
\Ii(\one) \otimes 
\big(
(\Ii \otimes \id)  \Delta \tau
\big)
+
\Ii(\X_{i}) \otimes \Ii_{i}^{+}(\X_{i}) \otimes 
\Ii^{+}_{i}(\tau)
\\ 
&+
\Ii(\X_{i}) \otimes
\big(
(\Ii_{i}^{+} \otimes \id) \Delta \tau
\big)
+
(\Ii \otimes \id \otimes \id)
(\id \otimes \Delta) \Delta \tau\;.
\end{align*}
All the terms in the expression for the left and right hand sides can be immediately matched except for the very last terms, but these are seen to be identical by using our induction hypothesis which is 
\[
(\id \otimes \Delta) \Delta \tau = 
(\Delta \otimes \id) \Delta \tau\;.
\]
\end{proof}

\subsection{Useful relations on trees}
\label{ss:relations}
We introduce two (reflexive and antisymmetric) relations on unplanted trees $\Tr$, which we denote $\leq$ and $\subset$.

\begin{definition}
Given $\bar{\tau}, \tau \in \Tr$ we have $\bar{\tau} \leq \tau$ if and only if one can obtain $\tau$ from $\bar{\tau}$ by replacing occurrences of the generators $\one$ in $\bar{\tau}$ with appropriately chosen trees $\tau_1,...\tau_{\numone(\bar{\tau})}\in \Tr$ and, for every $1 \le i \le d$, occurrences of $\X_{i}$ with trees $\tau_{1},\dots,\tau_{m_{\X_{i}(\bar{\tau})}} \in \Tr\setminus\{\one, \x_j,j\neq i\}$. 

\end{definition}
\begin{example}We give two pictorial examples
\[
\<K*3b11_black>\leqslant\<K*31b_black>,\qquad \<K*2Xb_black>\leqslant\<K*2Xb1_black>.
\] \end{example}
\begin{definition}
Given $\bar{\tau}, \tau \in \Tr$ we have $\bar{\tau} \subset \tau$ if and only if $\bar{\tau} = \tau$ or $\bar{\tau}$ appears in the inductive definition of $\tau$, that is the expression $\Ii(\bar{\tau})$ should appear at some point when one writes out the full algebraic expression for $\tau$.
\end{definition}
\begin{example}We give an example below. 
\[
\<3b_black>,\<3bi_black>\subset \<K*31b_black>
\]\end{example}
We also use the notation $<$ and $\subsetneq$ to refer to the non-reflexive (strict) relations corresponding to $\le$ and $\subset$.

One can get an intuition of how the coproduct works with the idea of cutting branches: on the left-hand side of $\Delta \tau$ we have trees $\bar{\tau}\leqslant\tau$,
and on the right-hand side, we have the trees $\Ii(\tilde{\tau})$ or $\Ii^+_i(\tilde{\tau})$ where $\tilde{\tau}\subset\tau$ has been cut from $\tau$ to obtain $\bar{\tau}$. We formalise this in in the following section.

\subsection{Another formula for $\Delta$} 
\label{ss: explicit coprod}
We write $\Ff^{\rec}$ for the collection of all finite, non-commutative words in $\Trec$, including the empty word. 
In particular, $\Ff^{\rec}$ is a vector space basis for $\Alg(\Ttp)$. 
We  define a map $C_+ : \Tr \times \Tr \rightarrow \Ff^{\rec}$ recursively. The recursion is given in the following table: 
%
%
%
%
\begin{table}[h!]
\centering
\begin{tabular}{|c|c c c c|}
\hline
$\tau\setminus\bar{\tau}$ & $\one$ & $\x_{i}$ & $\Xi$ & $\Ii(\bar{\tau}_1)\Ii(\bar{\tau}_2)\Ii(\bar{\tau}_3)$ \\ 
\hline  
$\one$ & $\Ii(\one)$ & $0$ & $0$ & $0$ \\ 
$\X_j$ & $\Ii(\X_j)$ & $\Ii^+_i(\X_j)$ & $0$ & $0$ \\ 
$\Xi$ & 0 & 0 & $1$ & $0$ \\ 
$\Ii(\tau_1)\Ii(\tau_2)\Ii(\tau_3)$ & $p_+\Ii(\tau)$ & $\Ii^{+}_{i}(\tau)$ & 0 & $C_+(\bar{\tau}_1,\tau_1) \cdot C_+(\bar{\tau}_2,\tau_2) \cdot C_+(\bar{\tau}_3,\tau_3)$ \\ 
\hline 
\end{tabular}
\vskip1ex
\caption{
This table gives a recursive definition of $C_+(\bar\tau, \tau)$. Possible values of $\tau$ are displayed in the first column, while possible values of $\bar{\tau}$ are shown in 
the first row. The corresponding values of $C_+(\bar\tau, \tau)$ are shown in the remaining fields. 
%
 }
\label{table:co-prod1} 
\end{table}

Here $p_+$ is the projection on trees of positive order. In particular, for $\delta<1$, one has $p_+\Ii(\Xi),\ \Ii^{+}_{i}(\Xi) =0$. 
Note also that $C_+(\Xi,\Xi)=1$, which is the unit element in the algebra of trees, not to be mistaken for $\one$.

\begin{example}
We give two pictorial examples 
\[
 C_+(\<K*3b11_black>,\<K*31b_black>)=\begin{cases}
 \<K*3b_black>\ \<K*3bi_black>& \text{ if }\delta>\frac13\\
 0&\text{ else.} 
 \end{cases}
\]
and
\[
C_+(\<K*2Xb1j_black>,\<K*31b_black>)=\<Kj*3bi_black>.
\]
We explain how this can be understood in the language of ``cuts''. There are three types of cutting procedures that can be applied to a tree $\sigma$. 
\begin{enumerate}
\item One cuts an $\Ii$-branch and takes the attached planted tree, leaving behind an $\Ii(\one)$.
\item One cuts an $\Ii$ branch and takes the attached planted tree, with its ``trunk'' becoming a derivative $\Ii^{+}_{k}$, and leaving behind a $\Ii(\X_{k})$. Note that this only occurs when one the tree $\tau \subset \sigma$ attached to this $\Ii$ branch belongs to $\tNn$. 
\item One cuts an $\Ii_{i}^{+}$ branch (which must be the trunk of $\sigma$) and takes the all of $\sigma$, leaving behind an $\Ii_{i}^{+}(\X_{i})$.
\end{enumerate}
If $\delta > \frac{1}{3}$, the tree $\<K*3b11_black>$ is obtained from $\<K*31b_black>$ by performing the first type of cut on the leftmost and rightmost $\Ii$ branches of $\<K*31b_black>$ connected to the root, leaving behind $\<K*3b11_black>$. 

In the second example, one performs the second type of cut on the rightmost $\Ii$ branch connected to the root of $\<K*31b_black>$, generating an $\Ii_{j}^{+}$ trunk on the planted tree taken and leaving behind an $\Ii(\X_{j})$ on $\<K*2Xb1j_black>$. 
\end{example}
Some immediate properties of these forests $C_{+}(\tau,\bar{\tau})$ are given in the following lemma. 
\begin{lemma}\label{lem:lemC}
Let $\tau, \bar{\tau} \in \Tr$. 
Then we have
\begin{align}
C_+(\bar{\tau},\tau)\neq 0\Rightarrow \bar{\tau} \leq \tau,\label{prop_C}\\
\Ii(\tau')\in C_+(\bar{\tau},\tau)\text{ or }\Ii^+_i(\tau')\in C_+(\bar{\tau},\tau)\Rightarrow \tau'\subset \tau\notag.
\end{align}
Under the assumption that $C_+(\bar{\tau},\tau)\neq 0$ we have
\begin{align}
\numnoise(\tau)=\numnoise(\bar{\tau})+\numnoise(C_+(\bar{\tau},{\tau}))\quad
\numone(\bar{\tau})+\numpoly(\bar{\tau})=\sharp(C_+(\bar{\tau},\tau)),\label{prop_C.5}\\ 
\numone(\tau)=\numone(C_+(\bar{\tau},\tau)),\quad \numpoly(\tau)=\numpoly(C_+(\bar{\tau},\tau))\notag,
\end{align} 
where $\sharp(C_+(\bar{\tau},\tau))$ is the number of trees in $C_+(\bar{\tau},\tau)$ (including multiplicity) and we extend the functions $\numone,\; \numpoly$ and $\numnoise$ to forests of the planted trees by summing over the individual planted trees in the forest.

\end{lemma}
The following lemma finally gives the expression of the coproduct $\Delta$ in terms of $C_+$. This expression is used throughout the paper without explicit reference to this lemma.
\begin{lemma}\label{lem:coproduct} 
For any $\tau\in \Nn \cup \Ww$, 
\begin{equation}\label{eq:eqeq}
\Delta \Ii(\tau)
=
\sum_{\bar{\tau}\in \Nn\cup \Ww }\Ii(\bar{\tau})
\otimes 
C_+(\bar{\tau},\tau).
\end{equation}

In particular, one has the formulae
\begin{equation}\label{eq:explicit_formula_for_I_tree}
\Delta \Ii(\tau)
=
\begin{cases}
\Ii(\tau) \otimes 1 & \textnormal{if }\tau \in \Ww\;,\\
\sum_{\bar{\tau}\in \Nn}\Ii(\bar{\tau})\otimes C_+(\bar{\tau},\tau)
& \textnormal{if }\tau \in \Nn\;.
\end{cases}
\end{equation}
Moreover, for any $\tau\in \Tr$,
\begin{equation}\label{eq:explicit_coproduct_of_product}
\Delta \tau=\sum_{\bar{\tau}\in \Nn\cup\Ww }\bar{\tau}\otimes C_+(\bar{\tau},\tau)
=
\tau \otimes 1 
+
\sum_{
\bar{\tau}\in \Nn
\atop
\bar{\tau} \not = \tau}
\bar{\tau}\otimes C_+(\bar{\tau},\tau)\;.
\end{equation}
\end{lemma} 

\begin{proof}
We prove \eqref{eq:explicit_formula_for_I_tree} by induction, with the base cases given by $\tau \in \Pp \cup \Ww$ which we check now. 
\[
\Delta\Ii(\one)=\Ii(\one)\otimes\Ii(\one)\text{ and }C_+(\bar{\tau},\one)=\Ii(\one)\delta_{\{\bar{\tau}=\one\}}.\]
Since $\Ii^+_j(\X_i)=0$ for $j\neq i$, we also have
$\Delta\Ii(\x_i)=\Ii(\one)\otimes\Ii(\x_i)+\Ii(\x_i)\otimes\Ii^+_i(\x_i)$ and $C_+(\bar{\tau},\x_i)=\Ii(\x_i)\delta_{\{\bar{\tau}=\one\}}+\Ii^+_i(\X_i)\delta_{\{\bar{\tau}=\x_i\}}.$

Finally, for
$\tau=w\in \Ww$, we have to show that the sum in the right-hand side of \eqref{eq:eqeq} contains only one term. Indeed, for any $w'\leqslant w$, we also have $w'\in \Ww$, hence $|\Ii(w')|<0$ and $C_+(\bar{\tau},w)=\delta_{\{\bar{\tau}=w\}}$.

We now prove the inductive step for $\tau=\Ii(\tau_1)\Ii(\tau_2)\Ii(\tau_3)\in \mNn$, with the induction hypothesis $\Delta\Ii(\tau_k)=\sum_{\bar{\tau}_k}\Ii(\bar{\tau}_k)\otimes C_+(\bar{\tau}_k,\tau_k),$ for $ k=1,2,3.$ 
We have that  $C_+(\one,\tau)=p_+\Ii(\tau)=\Ii(\tau)$ since $\tau\in\mNn$ and $C_+(\x_i,\tau)=\Ii^+_i(\tau)=\Ii^+_i(\tau)\delta_{\{\tau\in\tNn\}}$, and from the definition of $\Delta$, we have 
\begin{align*}
\Delta\Ii(\tau)=&\Ii(\one)\otimes\Ii(\tau)+\Ii(\x_i)\otimes \Ii^+_i(\tau)\\
&+\sum_{\bar{\tau}_1,\bar{\tau}_2,\bar{\tau}_3}\Ii(\Ii(\bar{\tau}_1)\Ii(\bar{\tau}_2)\Ii(\bar{\tau}_3))\otimes C_+(\bar{\tau}_1,\tau_1)C_+(\bar{\tau}_2,\tau_2)C_+(\bar{\tau}_3,\tau_3).
\end{align*}
Furthermore, for any $\bar{\tau}\leq\tau\in\mNn$, we have that either $\bar{\tau}=\tau$ or $\numone(\bar{\tau})+\numpoly(\bar{\tau})\geq 1$ therefore $|\bar{\tau}|\geq-2$. Hence the sum can be restricted to trees $\bar{\tau}=\Ii(\Ii(\bar{\tau}_1)\Ii(\bar{\tau}_2)\Ii(\bar{\tau}_3))\in\mNn$. This concludes the proof of \eqref{eq:eqeq}

Finally, we prove \eqref{eq:explicit_coproduct_of_product}. 
This is immediate if $\tau = \Ww$, otherwise one has $\tau = \Ii(\tau_{1})\Ii(\tau_{2})\Ii(\tau_{3}) \in \mNn$ and one obtains the desired result by combining \eqref{eq:explicit_formula_for_I_tree} and the multiplicativity of $\Delta$ with  respect to the tree product as described in the last line of \eqref{eqs: recursive deltaplus}.
\end{proof}

%
\section{From local products to paths}
\label{s:positive_renorm}
\subsection{Definition of paths and centerings}
For any choice of local product $\locprod$,
we will define two corresponding families of maps, a path $(\path_{z,x}:\Ttp \to\R;\ z,x\in\R \times \R^{d})$ and a centering  $(\path^\rec_z: \Trec \to \R;\ z\in\R \times \R^{d})$ where $\Trec$ is defined in \eqref{eq:trec}. 
 
Both the path and the centering are defined through an inductive procedure that intertwines these two families of maps.

One particular aim of our definitions will be to allow us to obtain the formula
\begin{equation}\label{eq: path on planted trees in terms of coproduct}
\path_{z,x} \Ii(\tau) 
=
(\locprod_{z} \otimes \locprod^{\rec}_{x})
\Delta \Ii(\tau) 
\textnormal{ for any } \tau \in \mNn\;,
\end{equation}
where we are extending $\locprod^{\rec}_{x}$ to act on forests of planted trees by multiplicativity. 

As we discussed in Section~\ref{s:ov}, we define
\begin{equation}\label{eq: path on I(Poly)}
\path_{z,x}\Ii(\one):=1 \textnormal{ and }\path_{z,x}\Ii(\x_i):=z_i-x_i,
\end{equation}
and 
\begin{equation}\label{eq: path on wild trees}
\path_{z,x} \tau := \locprod_{z} \tau,\quad 
\path_{z,x} \Ii(\tau) := \locprod_{z} \Ii(\tau) 
\textnormal{ for any }
\tau \in \Ww\;.
\end{equation}
With our definition \eqref{eq: path on wild trees} it is immediate that \eqref{eq: path on planted trees in terms of coproduct} holds for $\tau \in \Ww$.
For $\tau \in \mNn$ we define
\begin{equation}\label{eq: path on planted N trees}
\path_{z,x}\Ii(\tau)
:=
\mathcal{L}^{-1}(\path_{\bullet, x}\tau)(z)
-
\mathcal{L}^{-1}(\path_{\bullet, x}\tau)(x)
- \mathbbm{1}_{\tau \in \tNn} (z_{i} - x_{i}) \nu^{(i)}_{\tau}(x)\;,
\end{equation}
for which one must take as input the definition of $\path_{\bullet,x}\tau$ and 
\begin{equation}\label{eq: derivative term in path}
\nu_{\tau}(x) = 
(\nu^{(i)}_{\tau}(x))_{i=1}^{d} := 
\nabla( \mathcal{L}^{-1}(\locprod_{\bullet,x}\tau))(z)
|_{z=x}\;,
\end{equation} 
where $\nabla$ denotes the spatial gradient.
For the centering we will define 
\begin{equation}\label{eq: centering definitions}
\begin{split}
\locprod^{\rec}_{x} \Ii(\one)
:=& 
1,
\quad 
\locprod^{\rec}_{x} \Ii(\x_i):=-x_i,\quad 
\locprod^{\rec}_{x}\Ii^{+}_{i}(\x_{i}):= 1,\\
\locprod^{\rec}_{x}
\Ii(\tau)
:=&
-\mathcal{L}^{-1}(\path_{\bullet, x}\tau)(x)
+
{\mathbbm{1}}_{\tau\in\tNn}x_i \nu^{(i)}_{\tau}(x)
\textnormal{ for any }
\tau \in \mNn,\\
\locprod^{\rec}_{x}
\Ii_{i}^{+}(\tau)
:=&
-\nu_{\tau}^{(i)}(x)
\textnormal{ for any }
\tau \in \tNn\;.
\end{split}
\end{equation}
The formulae above are inductive, we remark that for $\tau \in \mNn$ one needs to be given $\locprod_{\bullet,y}\tau$ in order to define $\locprod^{\rec}_{y}\Ii(\tau)$ and, if $\tau \in \tNn$, that same input is needed to define $\locprod^{\rec}_{y}\Ii^{+}_{i}(\tau)$. 

Finally, to handle the tree products that appear in the remainder equation we define 
\begin{equation}\label{eq: path -  coproduct formula for product trees}
\locprod_{z,x} \tau
:=
(\locprod_{z} \otimes \locprod_{x}^{\rec})
\Delta \tau\;
\textnormal{ for all }\tau \in \mNn\;.
\end{equation}
Again, the formula above is an inductive definition - a sufficient condition for specifying the right hand side above is that we already know $\locprod_{\bullet,x}$ for every $\bar{\tau} \in \mNn$ for $\bar{\tau} \subsetneq \tau$.
\begin{remark}
Note that if \eqref{eq: path -  coproduct formula for product trees} is extended to $\tau \in \Ww$ it agrees with the definition given in \eqref{eq: path on wild trees}.
\end{remark}
\begin{lemma}\label{lem:pathCopro}
If one adopts the inductive set of definitions \eqref{eq: path on I(Poly)}, \eqref{eq: path on planted N trees}, \eqref{eq: path on wild trees}, and \eqref{eq: centering definitions}, to determine the path on $\Tt$ and the centering on $\Tt^{\rec}$ then 
\eqref{eq: path on planted trees in terms of coproduct} holds for every $\tau \in \Nn \cup \Ww $. 
\end{lemma}
\begin{proof}
The fact that \eqref{eq: path on planted trees in terms of coproduct} holds for every $\tau \in \Ww \cup \Ii(\Ww) \cup \Ii(\Pp)$ is immediate.  
Now suppose that $\tau\in\mNn$, we can then rewrite \eqref{eq: path on planted N trees} as 
\[
\path_{z,x}\Ii(\tau)=
\mathcal{L}^{-1}(\path_{\bullet, x}\tau)(z)-{\mathbbm{1}}_{\tau\in\tNn} z_{i}\nu^{(i)}_{\tau}(x)+\path^\rec_x\Ii(\tau).
\]
We also have
\begin{equation*}
\begin{split}
\mathcal{L}^{-1}(\path_{\bullet, x}\tau)(z)
=&
\mathcal{L}^{-1}[(\locprod_{\bullet}\otimes\path_x^\rec)
\Delta\tau](z)\\
=&
[(\mathcal{L}^{-1}\locprod_{\bullet}\cdot)(z)\otimes\path_x^\rec]
\Delta \tau\\
=&
(\locprod_{z}
\otimes\path_x^\rec)
(\Ii \otimes \id) 
\Delta \tau\;.
\end{split}
\end{equation*}
The desired claim follows upon observing that 
\begin{equation*}
\begin{split}
-{\mathbbm{1}}_{\tau\in\tNn}z_{i}\nu_{\tau}^{(i)}(x)
=&
(\locprod_z \otimes \locprod^\rec_x)(\Ii(\x_{i}) \otimes \Ii_{i}^{+}(\tau))\;,\\
\textnormal{and }
\path^\rec_x\Ii(\tau)
=&
(\locprod_z \otimes \locprod^{\rec}_{x})( \Ii(\one) \otimes \Ii(\tau))\;.
\end{split}
\end{equation*}
\end{proof}
At this point we have finished the inductive definition of the path on the trees of $\Tt$ and of the centering on the trees of $\Tt^{\rec}$. 
What is left is to define the path on the trees of $\{ \Ii^{+}_{i}(\tau): 1 \le i \le d,\ \tau \in \tNn  \cup \{\X_{i}\}\}$. 

In keeping with our convention of thinking of $\Ii^{+}_{i}(\X_{i})$ as acting like $\Ii(\one)$ for all analysis, we set
\[
\locprod_{z,x}\Ii^{+}_{i}(\X_i) := 1\;.
\] 
Our definition for the action of a path $\locprod_{\bullet,\bullet}$ on a tree $\Ii_{i}^{+}(\tau)$ for $\tau \in \tNn $ is  motivated by the fact that such trees are not really part of our tree expansions but instead only appear in order to encode change of base-point operations.

In particular, $\path_{u,x}\Ii_{i}^{+}(\tau)$ will play a role in how we relate centering at $u$ versus centering at $x$ and the identity we will be aiming for is Chen's relation \eqref{equ: def of Chen relation}. 

The key identity we would like to hold is that, for any $z,x \in \R \times \R^{d}$, 
\begin{equation}\label{eq: alt-alt-eq for gamma}
\locprod^{\rec}_{x} \Ii^{+}_{i}(\tau)
=
\big(
\locprod^{\rec}_{z} 
\otimes
\path_{z,x}
\big)
\Delta
\Ii^{+}_{i}(\tau)\;.
\end{equation}	

Note that in the above equation we are using our convention of extending $\path_{z,x}$ to forests of planted trees by multiplicativity. 

Expanding the action of $\Delta$ in \eqref{eq: alt-alt-eq for gamma} gives us an inductive procedure for defining $\path_{\bullet,\bullet}\Ii^{+}_{i}(\tau)$ for $\tau \in \tNn$.
Namely, we will define, for any $\tau \in \tNn$, 
\begin{equation}\label{eq: alt eq for gamma}
\begin{split}
\path_{z,x}\Ii^{+}_{i} (\tau)
:=&
\locprod^{\rec}_{x}\Ii^{+}_{i}(\tau)
-
(\locprod^{\rec}_{z}\circ \Ii^{+}_{i} \otimes \path_{z,x}) \Delta \tau\\
=&
-\nu_{\tau}^{(i)}(x)
+
\sum_{
\bar{\tau} \in \tNn }
\nu_{\bar{\tau}}^{(i)}(z)
\locprod_{z,x}
C_{+}(\bar{\tau},\tau)\;.
\end{split}
\end{equation}	
We then see that, in order to define $\path_{z,x} \Ii_{i}^{+}(\tau)$ it suffices to have defined $\path_{\bullet,x}(\tau)$, $\path_{\bullet,x}(\bar{\tau})$ for every $\bar{\tau} \in \mNn$ with $\bar{\tau} < \tau$, along with $\locprod_{z,x}\Ii(\tilde{\tau})$ and $\locprod_{z,x}\Ii_{i}^{+}(\tilde{\tau})$ for every $\tilde{\tau} \subsetneq \tau$. 
\begin{remark}
We take a moment to draw parallels between our definitions and those found in the theory of regularity structures. 
Those unfamiliar with the theory of regularity structures can skip this remark. 

In our context, the local product $\locprod_{z}$ plays the role of the ``un-recentered'' $\boldsymbol{\Pi}(\bullet)(z)$ map in the theory of regularity structures. 

The corresponding path $\path_{z,x}$ sometimes plays the role of the map $(\Pi_{x}\bullet)(z)$ and sometimes plays a role more analogous to $\gamma_{z,x}(\bullet)$ where $\gamma_{z,x}$ is as in \cite[Section~8.2]{hairer2014theory}, that is it is the character that defines $\Gamma_{z,x}$. 
\begin{itemize}
\item For $\sigma \in \Ww \cup \Ii(\Ww)$ the path $\path_{z,x}\sigma$ plays the role of $(\Pi_{x}\sigma)(z)$ or equivalently $(\boldsymbol{\Pi}\sigma)(z)$.
\item For $\tau \in \mNn$, 
\begin{itemize}
\item $\path_{z,x}\tau$ plays the role of $(\Pi_{x}\tau)(z)$. 
\item $\path_{z,x}\Ii(\tau)$ plays the role of $(\Pi_{x}\Ii(\tau))(z)$ and $\gamma_{z,x}(\Ii(\tau))$. In particular these two quantities are the same and in our context this means that the definition \eqref{eq: path on planted N trees} is actually compatible with the formula \eqref{eq: alt-alt-eq for gamma} - see \eqref{eq:strong_chen}. 
\end{itemize}
\item For $\tau \in \tNn$ and $1 \le i \le d$, $\path_{z,x}\Ii_{i}^{+}(\tau)$ plays the role of $\gamma_{z,x}(\Ii_{i}(\tau))$ which in general has a different value than $(\Pi_{x}\Ii_{i}(\tau))(z)$. This is why we cannot define $\path_{z,x}\Ii_{i}^{+}(\tau)$ with some formula that is analogous to \eqref{eq: path on planted N trees}. 
\end{itemize}
\end{remark}
\subsection{Properties of paths and centerings}
The first property we will investigate is Chen's relation. 
\begin{definition}
We say a local product satisfies Chen's relation on $\sigma \in \Tt$ if, for every $z,u,x \in \R \times \R^{d}$,
\begin{equation}\label{equ: def of Chen relation}
(\locprod_{z,u} \otimes \locprod_{u,x})
\Delta 
\sigma
=
\locprod_{z,x} \sigma \;.
\end{equation}	
\end{definition}
\begin{remark}
We use Chen's relation to study the change of base-point operation for tree expansions, and the sole role of $\Ii_{i}^{+}(\tau)$ for $\tau \in \tNn$ is to describe this procedure.  

Therefore we are not interested in Chen's relation \eqref{equ: def of Chen relation} for the case where $\sigma = \Ii_{i}^{+}(\tau)$ and instead $\Ii_{i}^{+}(\tau)$ plays the role of an intermediate object in the expansion of \eqref{equ: def of Chen relation}.    
\end{remark}
The following lemma is straightforward because of the trivial structure of the coproduct in those cases. 
\begin{lemma}\label{lem: Chen on wild is trivial}
Any local product automatically satisfies Chen's relation on every $\sigma \in \Ww \cup \Ii(\Ww) \cup \Ii(\Pp) \cup \{ \Ii_{i}^{+}(\X_{i}) \}_{i=1}^{d}$.
\end{lemma}
We also define semi-norms to capture our notion of order bounds, using the convolution with a approximation of unity denoted by $(\cdot)_L$ as introduced in equation \eqref{shauder ou1}.
\begin{definition}\label{def:seminorm}
Given a local product $\locprod$ and $\sigma \in \Tp$, we define 
\begin{equation}\label{equ: order of order seminorm}
[\locprod ; \sigma]
:=
\begin{dcases}
\sup_{
x \in \R \times \R^{d}
}
\sup_{L \in (0,1]}
\Big|
\big(
\path_{\bullet,x}\sigma
\big)_{L}(x)
\Big|
L^{-|\sigma|}&
\textnormal{ for }\sigma \in \Tr \cup \Ii(\Ww)\;, \\
\sup_{
z,x \in \R \times \R^{d}
\atop
d(z,x) \in (0,1)
}
|\path_{z,x}\sigma|
d(z,x)^{-|\sigma|}&
\textnormal{ for }\sigma \in \Tt^{\rec}  \;.
\end{dcases}
\end{equation}
We say $\locprod$ satisfies an order bound on $\sigma$ if $[\locprod ; \sigma] < \infty$. 
\end{definition}
\begin{remark}\label{rem: order bound automatic on wild trees}
We note that for any $\sigma \in \Ii(\Pp) \cup \{\Ii_{i}^{+}(\x_{i})\}_{i=1}^{d}$ we have the bound $[\locprod ; \sigma] \lesssim 1$ uniformly over local products $\locprod$. 

Since we are working in the smooth setting, it is also true that any local product $\locprod$ satisfies an order bound on $\tau \in \Tr$ (and $\Ii(\tau) \in \Ii(\Ww)$). 
However, it is not obvious and in general not true,  that these bounds remain finite, when one passes to the rough limit, where $\xi$ is genuinely only a $C^{-3+\delta}$ 
distribution. In the application to stochastic PDE, these bounds can be controlled in the limit, but this requires additional probabilistic arguments as well as a renormalization 
procedure.

\end{remark}
Since we have Lemma~\ref{lem: Chen on wild is trivial} and Remark~\ref{rem: order bound automatic on wild trees} our goal for this section is to verify that our definitions automatically guarantee that any local product satisfies 
\begin{itemize}
\item Chen's relation on any $\tau \in \Ii(\Nn) \cup \mNn$, and 
\item a quantitative order bound on any 
\[
\tau \in \Ii(\mNn)  \cup \{ \Ii_{i}^{+}(\tau): 1 \le i \le d,\ \tau \in \tNn\} \cup \Ii(\Ww)
\] 
in terms of order bounds on $\tau \in \mNn \cup \Tr$. 
\end{itemize}
We now turn to showing the desired statements about Chen's relation.
\subsubsection{Proving Chen's relation}
It is useful to introduce a stronger, partially factorized version of Chen's relation.
\begin{definition} 
Given $\Ii(\tau) \in \Ii(\Nn)$ we say a local product $\locprod$ satisfies the \textit{strong Chen's relation} on $\Ii(\tau)$ if, for every $x,y \in \R^{d}$, one has the identity
\begin{equation}\label{eq:strong_chen}
\begin{split}
\locprod^{\rec}_{y} \Ii(\tau)
=&
\big(
\locprod^{\rec}_{x} 
\otimes
\path_{x,y}
\big)\Delta
\Ii(\tau).
\end{split}
\end{equation}
\end{definition}
We remark that it is trivial to check that any local product satisfies strong Chen's relation on $\Ii(\tau) \in \Ii(\Pp)$.
The following lemma is half of our inductive step for proving Chen's relation.
\begin{lemma}\label{lem: strong chen induction}
Suppose a local product $\locprod$ satisfies Chen's relation on $\tau \in \mNn$, then $\locprod$ satisfies strong Chen's relation on $\Ii(\tau)$. 
\end{lemma}
\begin{proof} 
Expanding both sides of \eqref{eq:strong_chen} gives 
\begin{equation}\label{eq:strong_chen_work_1}
\begin{split}
{}
&
\mathcal{L}^{-1}
\big(
\path_{\bullet,y}
\tau
\big)(y)
+
y_{j} \locprod^{\rec}_{y} \Ii_{j}^{+}(\tau)\\
=&\mathcal{L}^{-1}(\path_{\bullet,y} \tau)|^{y}_{x}
+
(y_{i} - x_{i}) \locprod_{y}^{\rec} \Ii^{+}_{i}(\tau)\\  
{}
&
+
x_{k}
\big(
\locprod_{y}^{\rec} \Ii^{+}_{k}(\tau)
-
\sum_{\tilde{\tau} \in \mNn}
(\locprod_{x}^{\rec} \Ii^{+}_{k}(\tilde{\tau}))
\path_{x,y}C_{+}(\tilde{\tau},\tau)
\big)\\
{}&
-
\sum_{\tilde{\tau} \in \mNn}
\big(
\locprod_{x}^{\rec}\Ii(\tilde{\tau})
\big)
\big(
\path_{x,y}
C_{+}(\tilde{\tau},\tau)
\big)
\end{split}
\end{equation}
Each of the three lines on the right hand side above come from one of the three terms on the right hand side of the first line of \eqref{eqs: recursive deltaplus}.  

Doing the explicit cancellations lets us simplify \eqref{eq:strong_chen_work_1} to 
\begin{equation}\label{eq:strong_chen_work_2}
\begin{split} 
0=&\mathcal{L}^{-1}(\path_{\bullet,y}\tau)(x)
+
x_{k}
\sum_{\tilde{\tau} \in \mNn}
(\locprod_{x}^{\rec} \Ii^{+}_{k}(\tilde{\tau}))
\path_{x,y}C_{+}(\tilde{\tau},\tau)\\
{}&
+
\sum_{\tilde{\tau} \in \mNn}
\big(
\locprod_{x}^{\rec}\Ii(\tilde{\tau})
\big)
\big(
\path_{x,y}
C_{+}(\tilde{\tau},\tau)
\big)
\end{split}
\end{equation}
We then obtain \eqref{eq:strong_chen_work_2} by using our assumption on Chen's relation for $\tau$ to write
\begin{equation*}
\begin{split}
{}&\mathcal{L}^{-1}(\path_{\bullet,y}\tau)(x)
=
\mathcal{L}^{-1}\Big(
(\path_{\bullet,x} \otimes \path_{x,y})
\Delta 
\tau\Big)(x)\\
=&
\sum_{\tilde{\tau} \in \mNn}
\big(
\mathcal{L}^{-1}\path_{\bullet,x}(\tilde{\tau})\big)(x)
\big(
\path_{x,y}C_{+}(\tilde{\tau},\tau)
\big)\;,
\end{split}
\end{equation*}
and then recalling that for any $\tilde{\tau}$ in the above sum one has
\[
\mathcal{L}^{-1}(\path_{\bullet,x}\tilde{\tau}\big)(x)
=
-
\locprod_{x}^{\rec}\Ii(\tilde{\tau})
-
x_{k}
\locprod_{x}^{\rec}\Ii_{k}^{+}(\tilde{\tau})\;.
\]
\end{proof}
The following lemma is the second half of our inductive step.
\begin{lemma}\label{lem: alt proof of chen} 
Fix $\tau \in \mNn$ and suppose $\locprod$ is a local product that satisfies the strong Chen property on $\Ii(\bar{\tau})$ for every $\bar{\tau} \subsetneq \tau$. 
Then $\locprod$ satisfies Chen's relation on $\tau$. 
\end{lemma}
\begin{proof}
We have
\begin{equation*}
\begin{split}
(
\path_{x,y} \otimes \path_{y,z})
\Delta \tau
=& 
(
\locprod_{x} \otimes \locprod_{y}^{\rec} \otimes \path_{y,z})
(\Delta \otimes \id) \Delta \tau\\
=& 
(
\locprod_{x} \otimes \locprod_{y}^{\rec} \otimes \path_{y,z})
(\id \otimes \Delta) \Delta \tau\\
=& 
(
\locprod_{x} \otimes \locprod_{z}^{\rec})
 \Delta \tau\\
=& 
\path_{x,z} \tau\;.
\end{split}
\end{equation*}
In the first equality we used our identity \eqref{eq: path -  coproduct formula for product trees} for $\path_{x,y}$ and in the second we used the co-associativity property of Lemma~\ref{lem:coassoc}. 

For the third equality we used the fact that $\Delta$ is multiplicative over forests of planted trees so we can use either Lemma~\ref{lem: strong chen induction} or \eqref{eq: alt-alt-eq for gamma} for the planted trees that appear in the forests that appear on the right factor of $\Delta \tau$. 
Fix  $\tilde{\tau} \not \in \{\one, \X_{1},\dots,\X_{d}\}$. 
Then for any $\Ii(\bar{\tau}) \in C_{+}(\tilde{\tau},\tau)$ one has $\bar{\tau} \subsetneq \tau$ so one can use Lemma~\ref{lem: strong chen induction} for these factors. 
For the factors $\Ii^{+}_{i}(\bar{\tau}) \in C_{+}(\tilde{\tau},\tau)$ one can just use \eqref{eq: alt-alt-eq for gamma}.    
\end{proof}
Putting together these two lemmas for our inductive step, combined with Lemma~\ref{lem: Chen on wild is trivial} which gives us the bases cases for our induction, we arrive at the following proposition.
\begin{proposition}\label{prop:Chen}
Any local product $\locprod$ satisfies Chen's relation on $\Tt$. 	
\end{proposition}
\subsubsection{Order bound}
Below, for any local product $\locprod$ and forest of planted trees $\sigma_{1} \cdots \sigma_{n}$, $n \in \Z_{\ge 0}$, we write 
\[
[\locprod ; \sigma_{1} \cdots \sigma_{n}]
:= 
\prod_{j=1}^{n}
[\locprod ; \sigma_{j}]\;.
\]
We also write $[\locprod ; 0] := 0$.
With this notation we can state the following lemma. 
\begin{lemma}\label{lem:OB}
For any $\tau \in \Tr$ and uniform over local products $\locprod$ one has the estimate
\begin{equation}\label{eq: integration order bound}
[\locprod ; \Ii(\tau)]
\lesssim
\begin{dcases}
[\locprod ; \tau]& 
\textnormal{ for } \tau \in \Ww\;,\\
[\locprod ; \tau]
+
\max_{\bar{\tau} < \tau}
[\locprod ; \bar{\tau}]
[\locprod ; C_{+}(\bar{\tau},\tau)]& 
\textnormal{ for } \tau \in \mNn\;.
\end{dcases}
\end{equation}
Suppose $\tau \in \tNn$, then, for any $1 \le i \le d$, and uniform over local products, one has
\begin{equation}\label{eq: iplus order bound}
[ \locprod ; \Ii_{i}^{+}(\tau)]
\lesssim
[\locprod ; \tau]
+
\max_{\bar{\tau} < \tau}
[\locprod ; \bar{\tau}]
[\locprod ; C_{+}(\bar{\tau},\tau)]
\end{equation}
It follows that any $\locprod$ satisfies an order bound for any $\sigma \in \Tp$. 
\end{lemma}
\begin{proof} 
We start with proving \eqref{eq: integration order bound} for any $\tau \in \Tr$. 
Clearly the bound is trivial when the corresponding right hand side of \eqref{eq: integration order bound} is infinite so we assume that they are finite. 

When $\tau \in \Ww$ the desired estimate follows from Lemma~\ref{lemschauder-Neg} where we set $U(\bullet,x) = \locprod_{\bullet,x}\Ii(\tau) = \locprod_{\bullet}\Ii(\tau)$ and we can take $M \lesssim [\locprod ; \tau]$ by Lemma~\ref{lemma: multiplication by smooth function} (we need this lemma because of the presence of the cut-off function $\rho$ in \eqref{eq: cut-off heat equation} ) in \eqref{lemschauder-Neg}. 

For $\tau \in \mNn$ we will instead appeal to Lemma~\ref{lemschauder-Pos}.  
We set
\begin{equation}\label{eq: def of U for order bound}
U(y,x)
=
\mathcal{L}^{-1}(\path_{\bullet, x}\tau)(y)
-
\mathcal{L}^{-1}(\path_{\bullet, x}\tau)
(x)\;.
\end{equation}
It is clear that we can take $M^{(1)} \lesssim [\locprod ; \tau]$ for the assumption \eqref{lemschauder2-Pos}.

We now verify the three point continuity condition \eqref{lemschauder2}.
The role of $\lambda(\bullet,\bullet)$ in \eqref{lemschauder2-Pos} will be played by the quantity $\lambda_{\tau}(y,x) = (\lambda^{(i)}_{\tau}(y,x))_{i=1}^{d}  \in \R^{d}$ where 
\begin{equation}\label{eq: def of gamma}
\lambda^{(i)}_{\tau}(y,x)
=
-
\sum_{
\bar{\tau} \in \mNn
\atop
\bar{\tau} \not = \tau}
(\locprod^\rec_y \Ii^+_i(\bar{\tau}))\path_{y,x}C_+(\bar{\tau},\tau)\;.
\end{equation}
We then write
\begin{equation}\label{eq: three point continuity verification}
\begin{split}	
{}&
U(z,x)-U(y,x)-U(z,y)
-
(z_{i}-y_{i})
\lambda^{(i)}_{\tau}(y,x)\\
=&(\mathcal{L}^{-1}\path_{\bullet, x}\tau)|^z_y-(\mathcal{L}^{-1}\path_{\bullet, y}\tau)|^z_y
-
(z_{i}-y_{i})
\lambda^{(i)}_{\tau}(y,x)\\
=&(\mathcal{L}^{-1}(\path_{\bullet, y}\otimes \path_{y,x})\Delta\tau)|^z_y-(\mathcal{L}^{-1}\path_{\bullet, y}\tau)|^z_y
-
(z_{i}-y_{i})
\lambda^{(i)}_{\tau}(y,x)\\
=&
\sum_{
\bar{\tau} \in \mNn
\atop
\bar{\tau} \not = \tau}
(\mathcal{L}^{-1}(\path_{\bullet, y}\bar{\tau}))|^z_y\path_{y,x}C_+(\bar{\tau},\tau)
-
(z_{i}-y_{i})
\lambda^{(i)}_{\tau}(y,x)\\
=&
\sum_{
\bar{\tau} \in \mNn
\atop
\bar{\tau} \not = \tau}
\path_{z,y}
\Ii(\bar{\tau})\path_{y,x}C_+(\bar{\tau},\tau).
\end{split}
\end{equation}
For the second equality above we appealed to  Proposition~\ref{prop:Chen} to use  Chen's relation for $\tau$. 
Then by inserting the order bound for every term in the last line of \eqref{eq: three point continuity verification} we see that in  \eqref{lemschauder2-Pos} we can take $M^{(2)} \lesssim \max_{\bar{\tau} < \tau}
[\locprod ; \bar{\tau}]
[\locprod ; C_{+}(\bar{\tau},\tau)]$.

We turn to proving \eqref{eq: iplus order bound} and so we fix $\tau \in \tNn$. 
We obtain the desired estimate by applying Lemma~\ref{corschauderFS}.  
Here we again define $U(\bullet,\bullet)$ as in \eqref{eq: def of U for order bound} and $\lambda_{\tau}(\bullet,\bullet)$ by \eqref{eq: def of gamma}. 
Thanks to the computation \eqref{eq: three point continuity verification} we see we can take $M = \max_{\bar{\tau} < \tau}
[\locprod ; \bar{\tau}]
[\locprod ; C_{+}(\bar{\tau},\tau)]$ in \eqref{corschauder 3ptFS}. 
We also note that $\nu_{\tau}$ is then the optimal $\nu$ referenced in  Lemma~\ref{corschauderFS} and that we have
\[
\path_{y,x}
\Ii_{i}^{+}(\tau)
=
\nu_{\tau}^{(i)}(y) - \nu_{\tau}^{(i)}(x) + \lambda_{\tau}^{(i)}(y,x)
\]
and so the desired estimate is given by \eqref{corschauder2FS}. 
\end{proof}
%
%
\section{Modelled distribution}\label{sec: Modelled distribution}
%
\label{s:CP}
%
%
With the definition of local products and their associated paths in place, we now show how  to use them to give 
 a good local approximation to the solution $v$ of the remainder equation.
 As explained above in \eqref{outl-12}, we seek a local approximation to $v$ of the form
 \begin{align}\label{e:MD1}
v(y) \approx &    \;
\path_{y,x} \Theta(x)
=    
 \sum_{\tau \in \mathcal{N}  }     \Theta_{x}(\tau) \path_{y,x} \mathcal{I}(\tau), 
 \end{align}
for suitable coefficients $\Theta \in C^{\infty}(\R \times \R^{d}; \Vec(\Nn) )$  
(which we interchangeably view as a map $\Theta: \Nn \rightarrow C^{\infty}$ taking $\tau \mapsto \Theta_{\bullet}(\tau)$) and 
with an error  of order $\lesssim d(x,y)^\gamma$. 
In this section we first introduce a family of seminorms that measure the regularity of the coefficient map $\Theta$ and that ultimately permit to bound renormalized products. 
Subsequently, we turn to a specific choice of coefficients $\Theta$ (denoted by $\Upsilon$, see  Definition~\ref{def-Upsilon-rec}) which arise in ``freezing of coefficient procedure" described in Section~\ref{ss:2-4}.
The main result of this section, Theorem~\ref{theorem one zero}, shows a close connection between the various seminorms for this specific choice of coefficient.

 In order to  motivate the regularity condition  we rewrite equation~\ref{e:MD1} for another base-point $\bar{x}$ (but for the same argument $y$)
 \begin{align}\label{e:MD2}
v(y) \approx &   
 \sum_{\bar{\tau} \in \Nn  }        \Theta_{\bar{x}}(\bar{\tau}) \path_{y,\bar{x}} \mathcal{I}(\bar{\tau}) , 
 \end{align}
then use Chen's relation \eqref{equ: def of Chen relation}  and Lemma~\ref{lem:coproduct} in the form 
\[
\path_{y,\bar{x}} \mathcal{I}(\bar{\tau}) =  \sum_{\tau \in \Nn }
 \path_{y,x} \mathcal{I}(\tau)  
\path_{x,\bar{x}} C_+(\tau, \bar{\tau})
\]
to rewrite the right hand side of \eqref{e:MD2} and compare the resulting expression to \eqref{e:MD1},  arriving at
 \begin{align}\label{e:MD3}
\Big|  \sum_{\tau \in \Nn }  \Big(       \Theta_{x}(\tau)  -\sum_{\bar{\tau} \in \Nn }  \Theta_{\bar{x}}(\bar{\tau}) \path_{x,\bar{x}} C_+(\tau, \bar{\tau})  \Big)
  \path_{y,x} \mathcal{I}(\tau)  \Big|
 \lesssim d(x,y)^\gamma + d(\bar{x},y)^\gamma .
 \end{align}
Specialising this inequality to those $y$ for which $d(\bar{x},y) \approx d (x,y) \approx d(x,\bar{x}) \approx d $ yields the estimate
 \begin{align}\label{e:MD4}
\Big|  \sum_{\tau \in \Nn}  \Big(       \Theta_{x}(\tau)  -\sum_{\bar{\tau} \in \Nn  }  \Theta_{\bar{x}}(\bar{\tau}) \path_{x,\bar{x}} C_+(\tau, \bar{\tau})  \Big)
  \path_{y,x} \mathcal{I}(\tau)   \Big| \lesssim d^\gamma .
 \end{align}
In view of the the order bound \eqref{equ: order of order seminorm} 
\[
|  \path_{y,x} \mathcal{I}(\tau) | \lesssim d^{|\tau|+2},
\]
the following definition is natural.
\begin{definition}\label{def:Uandseminorm}
Let $\locprod_{\bullet}$ be a local product and $\path_{\bullet, \bullet}$ be the path constructed from $\locprod_{\bullet}$.
Then for  $\Theta \in C^{\infty}(\R \times \R^{d}; \Vec(\Nn) )$  for $\tau \in \Nn $ and $0<\gamma<2$ we define 
\begin{equation}\label{def:U}
U_{\gamma-2}^\tau(y,x) :=  \Theta_{y}(\tau)  -\sum_{\underset{|\bar{\tau}|<\gamma-2}{\bar{\tau} \in \Nn} }  \Theta_{x}(\bar{\tau}) \path_{y,x} C_+(\tau, \bar{\tau}),
\end{equation} 
and the seminorm
\begin{equation}\label{def:Mod-dist}
[U^\tau]_{ \gamma-|\tau|-2 } := \sup_{x,y}  \frac{1}{d(x, y)^{\gamma-|\tau|-2} }  \big| U_{\gamma-2}^\tau(y,x)  \big|.
\end{equation}
\end{definition}
 It is important to observe that the semi-norm $[U^\tau]_{ \gamma-|\tau|-2 } $ involves the coefficients, $\Theta_{\bar{\tau}}$ as well as the paths $\path_{\bullet, \bullet}$ on all symbol $\bar{\tau}$ for which $C_+(\tau, \bar{\tau})$
 does not vanish, and that all of these trees $\bar{\tau}$ satisfy $\tau\leq \bar{\tau}$ .  Also, for $\tau = \one$, in view of the identity $C_+(\one, \bar{\tau}) = \bar{\tau}$ and $| \one|=-2$ the quantity 
$[U^\one]_{ \gamma } $ 
 measures exactly the size of the error in the expression \eqref{e:MD1} at the beginning of this discussion.
 
 \begin{remark}
 The definition of the semi-norm corresponds exactly to Hairer's definition of a \emph{modelled distribution}, \cite[Definition 3.1]{hairer2014theory}.
 In Hairer's notation the expression
$| U_{\gamma-2}^\tau(y,x) |$ becomes 
\[
\| \Theta(x) - \Gamma_{xy}\Theta(y) \|_{|\Ii(\tau)|}\;.
\]
\end{remark}
The following lemma relates the notion of classical derivative with the generalised derivatives that appear in the modeled distribution. 
\begin{lemma}\label{lemma: derivative of modeled distribution}
Let $1 < \gamma < 2$. 
Fix a local product $\locprod$ and $\Theta \in C^{\infty}(\R \times \R^{d}; \Vec(\Nn) )$ with the property that, with $U^{\one}_{\gamma-2}(y,x)$ defined as in \eqref{def:U}, we have $[U^\one]_{\gamma } < \infty$.
Then, for $1 \le i \le d$, 
\begin{equation}\label{eq: formula for derivative}
\Theta_{x}(\X_{i})
=
\partial_{i}   \Big( \Theta_{y}(\one)
-
\sum_{\bar{\tau} \in \mNn
\atop
|\bar{\tau}| < -1}
\Theta_{x}(\bar{\tau})
\path_{y,x}C_{+}(\one,\bar{\tau})
\Big)
\Big|_{y=x}\;,
\end{equation}
where the partial derivative $\partial_i $ acts in the  variable $y$. 
 \end{lemma}
 \begin{proof}
 Note that by assumption we have that that $|U^{\one}_{\gamma - 2}(y,x)|\lesssim d(y,x)^{\gamma}$ and since $\gamma > 1$ it follows that 
 \[
 \big(
 \partial_i
 U^{\one}_{\gamma - 2}(y,x) 
 \big)\big|_{y=x} = 0 \;.
 \]
We obtain the desired result by plugging in the definition of $U^{\one}_{\gamma-2}(y,x)$ and recalling that 
\begin{equation*}
\begin{split}
\big(
\partial_{i} \path_{y,x}C_{+}(\one,\one)
\big)
\big|_{y=x}
=&
\big(
\partial_{i} \path_{y,x}\Ii(\one)
\big)
\big|_{y=x} = 0\;,\\ 	
\big(
\partial_{i} \path_{y,x}C_{+}(\one,\X_{j})
\big)
\big|_{y=x}
=&
\big(
\partial_{i} \path_{y,x}\Ii(\X_{j})
\big)
\big|_{y=x}
= \delta_{\{j = i\}}\;,\\
\textnormal{and }
\big(
\partial_{i} \path_{y,x}C_{+}(\one,\bar{\tau})
\big)
\big|_{y=x} =& 0\textnormal{ for }\bar{\tau} \in \mNn \textnormal{ with }|\bar{\tau}| > -1\;.
\end{split}
\end{equation*}
In the last statement we are using that $|\path_{y,x}C_{+}(\one,\bar{\tau})| \lesssim d(y,x)^{|\bar{\tau}| + 2}$. 
 \end{proof}

 We now introduce some short-hand notation that will be very useful in the following calculations. First, for a given local product $\path$, for  $\Theta \in C^{\infty}(\R \times \R^{d}; \Vec(\Nn) )$
  and $\gamma \in (0,2)$ we denote by

 \begin{equation}
\label{def:V}
V_\gamma(y,x):=\sum_{\underset{|\tau|<\gamma-2}{\tau\in \Nn,}}\Theta_x(\tau)\Y_{y,x}\Ii(\tau) \;.
\end{equation}

We also introduce a truncated ``square'' and a ``spatial derivative'':
\begin{align}
\label{def:V2}
V_\gamma^2(y,x)&:=\sum_{\underset{|\tau_1|+|\tau_2|<\gamma-4}{\tau_1,\tau_2\in\Nn,}}\Theta_x(\tau_1)\Theta_x(\tau_2)\Y_{y,x}\Ii(\tau_1)\Y_{y,x}\Ii(\tau_2)\;,\\
\label{def:V'}
V^{(i)}_\gamma(y,x)&:=\sum_{\underset{|\tau|<\gamma-1}{\tau\in\tNn,}}\Theta_x(\tau)\Y_{y,x}\Ii^+_i(\tau)\;.
\end{align}
Note that due to the choice of index set $V_\gamma^2(y,x)$ does not coincide with the point-wise square $(V_\gamma(y,x))^2$. Note furthermore, 
that recalling the definition \eqref{def:U} specialised to $\tau = \X_i$ as well as the identity $C_+(\X_i, \tau) = \Ii^+_i(\tau)$ for $\tau \in \tNn$ and $=0$ otherwise 
(see Table~\ref{table:co-prod1})  we have the identity
\begin{equation}\label{def:V'bis}
U_\gamma^{\X_i}(y,x) = \Theta_y(\X_i) - V_{\gamma-1}^{(i)}(y,x).
\end{equation}

 A first nice observation is a control for the ``three point continuity operator'' for $V_\gamma$ (the left-hand side of \eqref{e:lem-3-point} below) in terms of the $U_{\gamma-2}^\tau$ and  $\path_{\bullet, \bullet}$.
This ``three point continuity operator'' 
 corresponds exactly to Gubinelli's $\delta$ operator \cite{gubinelli2004controlling,gubinelli2010ramification}.
 In our calculations this quantity is needed to bound derivatives (see \eqref{eq:relV'}  below) and as input to the  Schauder lemmas presented in Section~\ref{ss:LSL}.

\begin{lemma}\label{lemma 3pt}
Let $\locprod_{\bullet}$ be a local product.
Let   $\Theta \in C^{\infty}(\R \times \R^{d}; \Vec(\Nn) )$ and let $V$ be defined as in \eqref{def:V}. Then for any space-time points $x,y,z \in \R \times \R^{d}$ we have
\begin{align}\notag
&V_\gamma(z,x)   - V_\gamma(z,y)+ V_\gamma(y,y)-V_\gamma(y,x) \\
\label{e:lem-3-point}
& \qquad= - \sum_{\underset{|\tau|<\gamma-2}{\tau\in\Nn\setminus\{ \one\},}} 
 U^{\tau}_{\gamma-2} (y,x) \,    \path_{z,y} \Ii(\tau) \;.
\end{align}
\end{lemma}

\begin{proof}
We re-organise the terms on the left hand side of \eqref{e:lem-3-point} to write
\begin{align}
\notag
&V_\gamma(z,x) -V_\gamma(z,y) + V_\gamma(y,y) - V_\gamma(y,x)\\
\notag
	& \qquad = \sum_{\underset{|\bar{\tau}|<\gamma-2}{\bar{\tau}\in\Nn  ,}}
		\Theta_x(\bar{\tau})
			\big(\path_{z,x}\Ii(\bar{\tau})- \path_{y,x}\Ii( \bar{\tau} ) \big) \\
\label{e:MD5}
	& \qquad\qquad  -\sum_{\underset{|\tau|<\gamma-2}{\tau \in\Nn ,}}
		 \Theta_y(\tau)
		 	\big( \path_{z,y}\Ii(\tau) - \path_{y,y}\Ii(\tau) \big).
\end{align}
We use Chen's relation \eqref{equ: def of Chen relation}
for the terms in the first sum on the right hand 
\begin{align*} 	
\path_{z,x}\Ii(\bar{\tau})- \path_{y,x}\Ii( \bar{\tau}) 
&=\sum_{\underset{|\tau|<\gamma-2}{\tau \in\Nn }}
	\big(\path_{z,y}\Ii(\tau) -\path_{y,y}\Ii(\tau) \big)
		\path_{y,x}C_+(\tau,\bar{\tau}) \;.
\end{align*}
Plugging this into the first term on the right hand side of \eqref{e:MD5}, exchanging the summation in $\tau$ and $\bar{\tau}$ gives
\begin{align*}
V_\gamma(z,x) -&V_\gamma(z,y) + V_\gamma(y,y) - V_\gamma(y,x)=\\
\sum_{\underset{|\tau|<\gamma-2}{\tau \in\Nn ,}}&\Big(\big( \path_{z,y}\Ii(\tau) - \path_{y,y}\Ii(\tau) \big)
	\big(\sum_{\underset{|\bar{\tau}|<\gamma-2}{\bar{\tau}\in\Nn  ,}}
		\Theta_x(\bar{\tau})
		\path_{y,x}C_+(\tau,\bar{\tau} ) -
		 \Theta_y(\tau)\big)\Big).
\end{align*}
Finally,
noting that $\path_{z,y}\Ii(\one) -\path_{y,y}\Ii(\one)= 0$  and $\path_{y,y}\Ii(\tau) = 0$ for $\tau \in \Nn \setminus\{ \one\}$  leads to the desired 
expression \eqref{e:lem-3-point}.
\end{proof}

The following lemma gives relations between $V$, $V^2$ and $V^{(i)}$. They will be used heavily in Section~\ref{s:PMT}. 
\begin{lemma}
Truncation: for $0<\beta<\gamma<2$, $V_{\gamma-\beta}$ is a truncation of $V_\gamma$
\label{lem:relations V V2}
 \begin{align}
\notag
\Theta_y(\one)-V_\gamma(y,x)=&\Theta_y(\one)-V_{\gamma-\beta}(y,x) \\
\label{eq:relV}
&-\sum_{\gamma-\beta-2\leqslant|\tau|<\gamma-2}\Theta_x(\tau)\Y_{y,x}\Ii(\tau)\;.
\end{align}
Multiplication: 
for $0<\gamma<1$, $V^2_\gamma$ is a truncation of $(V_\gamma)^2$:
 \begin{align}
\label{eq:relV2}
\Theta_y(\one)^2&-V^2_{\gamma}(y,x)
=\Theta_y(\one)(\Theta_y(\one)-V_{\gamma}(y,x))\\
&+\sum_{-2\leqslant|\tau|<\gamma-2}\Theta_x(\tau)\Y_{y,x}\Ii(\tau)(\Theta_y(\one)-V_{\gamma-|\tau|-2}(y,x)).\notag
\end{align}
Derivative: To control derivatives we use the following reorganisation of Lemma~\ref{lemma 3pt}: for $1<\gamma<2$  we have
\begin{align}\label{eq:relV'}
&V_\gamma(z,x) - V_\gamma(z,y) + V_\gamma(y,y)  -V_\gamma(y,x)  \\&+\sum_{i=1}^d \big(\Theta_y(\X_i) - V_{\gamma-1}^{(i)}(y,x)\big)(y-z)_i\notag	   
 = -\sum_{\underset{|\tau|<\gamma-2}{\tau\in\mNn ,}} 
	 U^{\tau}_{\gamma-2} (y,x)   \,    \path_{z,y} \Ii(\tau).
\end{align}
\end{lemma}
This last identity \eqref{eq:relV'} will be combined with Lemma~\ref{corschauder} to give a bound on $V^{(i)}$ below.
\begin{proof}
The first two identities are immediate. In the third one, we  use the identity \eqref{def:V'bis} and rewrite the term corresponding to $\tau=\x_i$, for which $\path_{z,y}\x_i=(y-z)_i$. 
\end{proof}
We introduce in the following definition a coefficient map $\Upsilon$ depending on some real valued functions $v_\one$ and $ v_{\X_i}$, $i=1...d$, on $\R \times \R^{d}$.
\begin{definition}\label{def-Upsilon-rec} Given real parameters $v_\one,\ v_{\X_i},\ i=1...d$ and $\tau \in \Nn \sqcup \Ww$ we set
\begin{equation}\label{eq:recursive_ups}
\Upsilon(\tau)[v_\one,v_\x]
:=\begin{cases}
v_\one, & \tau=\one, \\
v_{\X_{i}}, & \tau=\X_{i},\\
1, & \tau=\Xi,\\
-\prod_{i=1}^3\Upsilon(\tau_i)[v_\one,v_\x], & \tau=\Ii(\tau_1)\Ii(\tau_2)\Ii(\tau_3),
\end{cases}
\end{equation}
where we adopt, above and in what follows, the notation convention $v_{\x} = (v_{\x_{i}})_{i=1}^{d}$. 
We may omit the parameters $[v_\one,v_\x]$ from the notation when there is no possible confusion,
We usually work in the case where $v_{\one}$ and $v_{\X}$ are functions of spacetime $\R \times \R^{d}$ in which case we use the shorthand: $\Upsilon(\tau)[v_\one(z),v_\x(z)]=:\Upsilon_z(\tau)$.

We extend $\Upsilon$ to planted trees, in particular we set  $\Upsilon(\Ii(\tau)) := \Upsilon(\tau)$ 
 for $\tau \in \Ww\cup\Nn$ and $\Upsilon(\Ii^+_i(\tau))= \Upsilon(\tau)$ for $\tau \in \tNn \cup \{\X_{i}\}_{i=1}^{d}$. 
We also extend $\Upsilon$ to forests of planted trees by multiplicativity.  
\end{definition}
It is straightforward to derive an explicit formula from the earlier recursive formula $\Upsilon$, we state this as a lemma. 
\begin{lemma}\label{Upsilon-Lemma}
\[
\Upsilon(\tau)[v_\one,v_\x]=(-1)^{\frac{\num(\tau)-1}{2}} v_\one^{\numone(\tau)} \prod_{i=1}^{d}v_{\x_{i}}^{\numpolyi(\tau)}.
\]
\end{lemma}
Note that since we only consider trees of negative order, we always have $\numone(\tau)+2\numpoly(\tau)<3$. 
In particular only $\pm v_{\one}$, $\pm v_\one^2$, $\pm v_{\x_i}$ or $\pm 1$ can appear and these possibilities correspond, respectively, to $\numone(\tau) = 1$, $\numone(\tau) = 2$, $\numpolyi(\tau) = 1$, and $\numone(\tau) = \numpoly(\tau) = 0$.
\begin{assumption}\label{rem:assumecoherence} 
For the remainder of the paper, we will always assume that the  coefficient map  $\Theta \in C^{\infty}(\R \times \R^{d}; \Vec(\Nn) )$  is of the form
\begin{equation}\label{eq:coeff_given_by_ups}
\Theta_{z}(\bullet) = \Upsilon(\bullet)[v_{\one}(z),v_{\X}(z)] = \Upsilon_{z}(\bullet)
\textnormal{ for some }v_\one,\ v_{\X}\;.
\end{equation}
We enforce the relation \eqref{eq:coeff_given_by_ups} for the rest of the article, in particular this is implicit in any  use of the notation $U^\tau$.
\end{assumption}
The following identities are the main motivation behind our definition of $\Upsilon$.
\begin{lemma}\label{lemma: self-similarity of coherence}
One has
\begin{equation}
\label{eq: mul-coherence}
-
\Big(
\sum_{\tau \in \Nn \cup \Ww}
\Upsilon(\tau) \Ii(\tau) \Big)^{3}
=
\sum_{\tau \in \mNn \cup \mWw}
\Upsilon(\tau)
\tau\;.
\end{equation}
We also have for any $\tau\in\Nn, \ \bar{\tau}\in\mNn$ such that $C_+(\tau,\bar{\tau})\neq 0$, 
\begin{equation}\label{eq: coherence relation for upsilon +}
\Upsilon(\bar{\tau})=(-1)^\frac{\num(\tau)-1}2\Upsilon(C_+(\tau,\bar{\tau})),
\end{equation}
where $\Upsilon$ acts on forests multiplicatively.
\end{lemma}
\begin{proof}
For the equality \eqref{eq: mul-coherence} we first note that we have this equality if we dropped all the $\Upsilon$'s and dropped the minus sign on the left hand side, this is just \eqref{eq:cube}. 
What is left is to make sure that by inserting the minus sign and $\Upsilon$'s in \eqref{eq: mul-coherence}, the unplanted trees on either side of the equation have the same coefficient, but this is an immediate consequence of the last line of \eqref{eq:recursive_ups}.

The identity \eqref{eq: coherence relation for upsilon +} follows from Lemma~\ref{Upsilon-Lemma} and Lemma~\ref{lem:lemC}.
\end{proof}
The key result of this section is the observation that under the structure assumption described in Remark~\ref{rem:assumecoherence}, all continuity conditions are controlled by the condition on $\one$ and $\x_i$.  This follows from the bounds established in Lemma~\ref{lem:relations V V2} and the following theorem which uses the structure of $\Upsilon$.
\begin{theorem}\label{theorem one zero} 
For $\tau\in\Tt$, the quantity $U_{\gamma}^\tau(y,x)$ takes the following form
\begin{equation}\label{eq:expansion U tau}
U^\tau_\gamma(y,x)=\begin{cases}
(-1)^\frac{\num(\tau)-1}2(v_\one(y)-V_{\gamma-|\tau|}(y,x)) & \textnormal{if }\numone(\tau) = 1,\\
(-1)^\frac{\num(\tau)-1}2(v_\one(y)^2-V^2_{\gamma-|\tau|}(y,x)) & \textnormal{if }\numone(\tau) = 2,\\
(-1)^\frac{\num(\tau)-1}2(v_{\X_i}(y)-V^{(i)}_{\gamma-|\tau|}(y,x)) & \textnormal{if }\numpolyi(\tau)=1,\\
0 & \textnormal{if }\numone(\tau), \numpoly(\tau) = 0.
\end{cases}
\end{equation}
\end{theorem} 
\begin{proof}
From Lemma~\ref{lemma: self-similarity of coherence}, we have
\[\Upsilon(\bar{\tau}) C_+(\tau, \bar{\tau})=(-1)^\frac{\num(\tau)-1}2\Upsilon(C_+(\tau,\bar{\tau}))C_+(\tau, \bar{\tau}).\]
This allows to write:
\[
U^\tau_\gamma(y,x)=\Upsilon_y(\tau)-(-1)^{\frac{m(\tau)-1}{2}}\sum_{|\bar{\tau}|<\gamma}\Upsilon_x(C_+(\tau,\bar{\tau}))\path_{y,x}C_+(\tau,\bar{\tau}).
\]
We can use Lemma~\ref{lem:lemC} to study the different cases:
\begin{enumerate}
\item[$\bullet$]If $\numone(\tau) = 1$ and $C_+(\tau,\bar{\tau})\neq 0$ then there exists a unique $\tilde{\tau}\in\Nn$ such that $C_+(\tau,\bar{\tau})=\Ii(\tilde{\tau})$. Conversely, for each $\tilde{\tau}\in \Nn$ with $|\tilde{\tau}|<\gamma-|\tau|-2$, there exists a unique $\bar{\tau}\in \Nn$ with $|\bar{\tau}|<\gamma-2$ such that $C_+(\tau,\bar{\tau})=\Ii(\tilde{\tau})$.   Indexing the sum over this $\tilde{\tau}$ gives the expression of $V_{\gamma-|\tau|}$.
\item[$\bullet$]If $\numpolyi(\tau) = 1$ and $C_+(\tau,\bar{\tau})\neq 0$ then there exists a unique $\tilde{\tau}\in\Nn$ such that $C_+(\tau,\bar{\tau})=\Ii^+_i(\tilde{\tau})$. Indexing the sum over this $\tilde{\tau}$ gives the expression of $V^{(i)}_{\gamma-|\tau|}$.
\item[$\bullet$]If $\numone(\tau) = 2$ and $C_+(\tau,\bar{\tau})\neq 0$ then there exists a unique non-commutative couple $(\tilde{\tau}_1,\tilde{\tau}_2)\in\Nn^2$ such that $C_+(\tau,\bar{\tau})=\Ii(\tilde{\tau}_1)\cdot\Ii(\tilde{\tau}_2)$. Indexing the sum over these $\tilde{\tau}_1,\tilde{\tau}_2$ gives the expression of $V^{2}_{\gamma-|\tau|}$, using also the multiplicative action of $\path$ on forests of planted trees.
\end{enumerate}
We finally see that we get the correct order using the fact that $|\tau|+|C(\tau,\bar{\tau})| = |\bar{\tau}|$.
\end{proof}
Lemma~\ref{lemma: derivative of modeled distribution} above showed that the continuity condition on a modelled distribution enforces the relation \eqref{eq: formula for derivative} between the coefficients $\Theta_{z}(\X_{i})$ and the other coefficients. 
Since we are now imposing the structural condition \eqref{eq:coeff_given_by_ups}, we see that all the left hand side of \eqref{eq: formula for derivative} is given by $v_{\X_i}(x)$ and the right hand side of \eqref{eq: formula for derivative} has no dependence on $v_{\X}(x)$ - therefore the continuity condition combined with \eqref{eq:coeff_given_by_ups} determines $v_{\X}(x) = (v_{\X_i}(x))_{i=1}^{d}$ as a function of $v_{\one}$ and the local product $\locprod$ along with associated derivatives.  
For future use we encode this as a map  $(v_{\one},\locprod) \mapsto \mathcal{D}^{\locprod}_iv_{\one} = v_{\X_{i}}$.  
\begin{definition}\label{def: vX map}
Given a local product $\locprod$ and a smooth function $v:\R \times \R_x^{d} \rightarrow \R$ we define $\mathcal{D}^{\locprod}v = (\mathcal{D}^{\locprod}_{i}v)_{i=1}^{d}$, $\mathcal{D}^{\locprod}_{i}v:\R \times \R^{d} \rightarrow \R$ by setting, for $1 \le i \le d$, 
\[
(\mathcal{D}^{\locprod}_{i}v)(x)
=
\sum_{\bar{\tau} \in \mNn
\atop
|\bar{\tau}| < -1}
\Upsilon_{x}(\bar{\tau})
\big(
\partial_{i} \path_{y,x}C_{+}(\one,\bar{\tau})
\big)
\big|_{y=x}
-
\partial_{i}v(y)\big|_{y=x}
\;,
\]
where the partial derivatives above acts in the dummy variable $y$ and the $\Upsilon_{x}(\cdot)$ coefficients above are defined using the parameter $v(x)=v_{\one}(x)$. 
Note that we do not need to specify a parameter $v_{\X}(x)$ for the $\Upsilon_{x}$ map above since every $\bar{\tau}$ appearing in this sum satisfies $\numpoly(\bar{\tau}) = 0$. 
\end{definition}

\section{Renormalised products of tree expansions}\label{sec: actual products}
In this section we  define renormalised ``point-wise'' products, taking as input a local product and tree expansions.
Fix some smooth noise $\xi$ and let $\locprod$ be a lift of $\xi$. 
Upon fixing the choice of $\locprod$ we will arrive at analog of the \eqref{phi4} equation which we now try to identify. 
The solution to this yet to be identified equation will be written in the form 
\begin{align}
\notag
\phi(z) &= v(z) + \sum_{\tau \in \Ww} (-1)^\frac{\num(\tau)-1}{2} \locprod_{z}\Ii(\tau) \\
\notag
&= \sum_{\tau \in \Ww \cup \{ \one\} } \Upsilon_z(\tau)  \path_{z,z}\Ii(\tau) \\
&= \sum_{\tau \in \Ww \cup \Nn } \Upsilon_z(\tau)  \path_{z,z}\Ii(\tau)\;.
\label{sec9.4-1}
\end{align}
Above, in passing from the first to second line, we have used the definition of $\Upsilon$ on $\Ww$ in Lemma~\ref{Upsilon-Lemma}  as well as the fact that 
$\path_{z,z} \Ii(\tau) = \locprod_z \Ii(\tau)$ for $\tau \in \Ww$.
In passing to the last line we used the simple observation that for $\tau \in \Nn \setminus \{ \one \}$, 
$\path_{z,z} \Ii(\tau) = 0$.
This trivially gives the identity
\begin{align}
\phi^3(z) 
\notag
&= \sum_{\tau_1, \tau_2, \tau_3  \in \Ww \cup \Nn } \Upsilon_z(\tau_1) \Upsilon_z(\tau_2)\Upsilon_z(\tau_3)  \path_{z,z}\Ii(\tau_1)  \path_{z,z}\Ii(\tau_2) \path_{z,z}\Ii(\tau_3) \\
\label{sec9.4-2}
&=    - \sum_{\tau_1, \tau_2, \tau_3  \in \Ww \cup \Nn } \Upsilon_z(\Ii(\tau_1) \Ii (\tau_2)\Ii (\tau_3) )  \path_{z,z}\Ii(\tau_1)  \path_{z,z}\Ii(\tau_2) \path_{z,z}\Ii(\tau_3) \;.
\end{align}
The renormalization now consists of replacing each of the point-wise products $ \path_{z,z}\Ii(\tau_1)  \path_{z,z}\Ii(\tau_2) \path_{z,z}\Ii(\tau_3) $ which in general we do not control 
by the terms   $ \path_{z,z}\Ii(\tau_1)  \Ii(\tau_2)\Ii(\tau_3) $ which we control by assumption.
The following definition extends this idea to more general expansions.
\begin{definition}\label{def: renormalised product}
Fix a local product $\locprod$. 
Suppose we are given, for $1 \le i \le 3$, smooth functions $\theta^{(i)}(z) :\R \times \R^{d} \rightarrow \R$ and  
$\Theta^{(i)}:\R \times \R^{d} \rightarrow \Vec(\Tl)$ with 
\begin{equation}\label{eq: tree expansion for renormalised product}
\theta^{(i)}(z)
=
\path_{z,z}\Theta^{(i)}(z)
=
\sum_{\tau \in \Nn \cup \Ww}
\Theta^{(i)}_{z}(\tau)
\path_{z,z}\Ii(\tau)\;.
\end{equation}
Then, we define
\begin{equation}	\label{equ: def of renormalised product}
\begin{split}
(\theta^{(1)} \circ_{\locprod} \theta^{(2)} 
\circ_{\locprod} \theta^{(3)})(z)
=&\\
\sum_{
\tau = \Ii(\tau_{1})\Ii(\tau_{2})\Ii(\tau_{3}) \in \mNn \cup \mWw}
\Theta^{(1)}_{z}(\tau_{1})
\Theta^{(2)}_{z}(\tau_{2})
\Theta^{(3)}_{z}(\tau_{3})
\path_{z,z} \tau \;.
\end{split}
\end{equation}
\end{definition}
In the case where there is a single function $\theta^{(i)} = \theta$ and a single corresponding tree expansion $\Theta^{(i)} = \Theta$ then we just write $\theta^{\circ_{\locprod}3}$ for the left hand side of \eqref{equ: def of renormalised product}.

We adopt the convention that when $\theta^{(i)}(z) = \locprod_{z,z} \Ii(\tau) = \locprod_{z} \Ii(\tau)$, where $\tau \in \Ww$, appears as a factor in a local product, we will implicitly take $\Theta^{(i)}(z) = \Ii(\tau)$. 
We remark that, with the generality we allow in the definition of local products, there is no reason to expect that  $(\theta^{(1)} \circ_{\locprod} \theta^{(2)} 
\circ_{\locprod} \theta^{(3)})(z)$ is given by some polynomial in the functions $\theta^{(i)}(z)$ and their spatial derivatives. 
However, a trivial case where there is such a correspondence is in the case of a multiplicative local product in which case one clearly has $(\theta^{(1)} \circ_{\locprod} \theta^{(2)} 
\circ_{\locprod} \theta^{(3)})(z) = \theta^{(1)}(z)\theta^{(2)}(z)\theta^{(3)}(z)$. %

\begin{remark}
An important observation about the importance of these tree expansions is the following. 
In \eqref{eq: tree expansion for renormalised product}, the contribution on the two right hand sides from $\tau \in \Nn \setminus \{\one\}$ vanishes due to the order bound for any local product. 

Similarly, if the local product $\locprod$ is multiplicative then none of the terms involving either $\tau_{1}$ or $\tau_{2}$ or $\tau_{3} \in \Nn$ contribute to the value of the renormalised product. 

However, if the $\locprod$ is not multiplicative then it can certainly be the case that these terms from $\tau \in \Nn \setminus \{\one\}$ contribute to the value of the renormalised product even though they do not contribute to the value of the $\theta^{(i)}(z)$. 
At the same time, in the end we only need to keep products of trees $\Ii(\tau_{1})\Ii(\tau_2)\Ii(\tau_{3})$ with $|\Ii(\tau_{1})\Ii(\tau_2)\Ii(\tau_{3})| < 0$ in our analysis of renormalised products.  
This motivates our truncation convention for tree products described at the end of Section~\ref{ss:order}. 
\end{remark}
We can now specify the equations we obtain a priori bounds for.
\begin{definition}\label{def: solution}
Fix a local product $\locprod$. 
Then we say the solution of the $\Phi^4$ equation driven by $\locprod$ is a smooth function $\phi:\R \times \R^{d} \rightarrow \R$ solving
\begin{equation}\label{eq: renormalised product phi equation}
\heat \phi = 
\phi^{\circ_{\locprod} 3} + \xi\;.
\end{equation} 
where we write $\xi = \locprod_{z}\Xi$ and the tree expansion $\Phi$ for $\phi$ used to define $\phi^{\circ_{\locprod} 3}$ is defined by
\[
\Phi(z)
=
\sum_{\tau \in \Nn \cup \Ww} \Upsilon_{z}(\tau) \Ii(\tau)\;,
\]
where $\Upsilon$ is defined as in Definition~\ref{def-Upsilon-rec} and the parameter $v_{\one}(z)$ is given by
\begin{equation}\label{eq: v definition}
v_{\one}(z)=
\phi(z) -
\sum_{\tau \in \Ww} \Upsilon_{z}(\tau) \locprod_{z}\Ii(\tau)\;,
\end{equation}
while we set parameter $v_{\X} = \mathcal{D}^{\locprod}v_{\one}$ as given in Definition~\ref{def: vX map}.  
\end{definition}
\begin{definition}\label{def: remainder solution}
Fix a local product $\locprod$. 
Then we say the solution of the $\Phi^4$-remainder equation driven by $\locprod$ is a smooth function $v:\R \times \R^{d} \rightarrow \R$ satisfies the equation 
\begin{equation}\label{eq: remainder equation}
\begin{split} 
\heat v
=& 
-
\Big( 
v^{3}
+
3
\sum_{\tau \in \Ww} (-1)^{\frac{\num(\tau)-1}{2}}
v \circ_{\locprod} v \circ_{\locprod} \locprod_{\bullet}\Ii(\tau)\\
{}&
\quad
+3
\sum_{\tau_{1},\tau_{2} \in \Ww}
 (-1)^{\frac{\num(\tau_1) + \num(\tau_2)-2}{2}}
v \circ_{\locprod}  \locprod_{\bullet}\Ii(\tau_1)
\circ_{\locprod} \locprod_{\bullet}\Ii(\tau_2)\\
{}&
\quad
-
\sum_{
\tau_{1},\tau_{2},\tau_{3} \in \Ww
\atop
|\tau_{1}| + |\tau_{2}| + |\tau_{3}| > -8}
(-1)^{\frac{\num(\tau_1) + \num(\tau_2) + \num(\tau_3)-3}{2}}
\locprod_{\bullet}\Ii(\tau_{1})\Ii(\tau_{2})\Ii(\tau_{3})
\Big).
\end{split}
\end{equation} 
Here, the tree expansion that governs $v$ is given by
\[
V(z)
=
\sum_{\tau \in \Nn } \Upsilon_{z}(\tau) \Ii(\tau)\;
\]
where $\Upsilon_{z}$ is defined as in Definition~\ref{def-Upsilon-rec} using the function $v(z)$ and the parameters $v_{\one} = v$ and $v_{\X} = \mathcal{D}^{\locprod}v$.  	
\end{definition}
The following statement is then straightforward.
\begin{lemma}
There is a one to one correspondence between solutions in the sense of Definition~\ref{def: solution} to those of Definition~\ref{def: remainder solution}, the correspondence is given by taking $\phi$ which is a solution in the sense  Definition~\ref{def: solution} and mapping it to $v = v_{\one}$ which will be a solution in the sense of \eqref{eq: v definition}. 
\end{lemma}
%
\section{A useful class of local products}
\label{sec:useful class of local products}
This section is somewhat orthogonal to the main result of this paper. 
Here we present a particular subset of local products which are defined in terms of recursive procedure that guarantees that the renormalised product appearing in \eqref{eq: renormalised product phi equation} is a local polynomial in $\phi$ and its derivatives. 
This class of local products also includes those that satisfy the necessary uniform stochastic estimates in order to go to the rough setting, namely the BPHZ renormalisation of \cite{MR3935036,2016arXiv161208138C}.

\subsection{Another derivative edge}\label{ss:ADE}
Our class of local products will, for $\delta$ sufficiently small, allow the renormalised product $\phi^{\circ_{\locprod} 3}$ to involve spatial derivatives $\partial_{i}\phi$ for $1 \le i \le d$. 

To describe the generation of these derivatives in terms of operations on trees we will introduce yet another set of edges $\{\Ii^{-}_{i}\}_{i=1}^{d}$ and another set of planted trees 
\[
\Tlm 
=
\Tl \cup 
\big\{ \Ii^{-}_{i}(\tau): \tau \in \Tr, 1 \le i \le d\} \; \cup \{
\Ii^{+}_{i}(\X_{i}): 1 \le i \le d\}\;.
\]
We also adopt the notational convention that 
\[
\Ii^{-}_{i}(\X_{i}) = \Ii^{+}_{i}(\X_{i}),\;\Ii^{-}_{i}(\one) = 0, 
\textnormal{ and }\Ii^{-}_{i}(\X_{j}) = 0
\textnormal{ if }
i \not = j\;.
\] 
Both the new set of edges $\{\Ii^{-}_{i}\}_{i=1}^{d}$ and the set of edges $\{\Ii^{+}_{i}\}_{i=1}^{d}$ introduced in Section~\ref{sec:coproduct} should be thought of as representing a spatial derivative of a solution to a heat equation. 
However these two sets of edges play different roles in our argument:
a symbol $\Ii^{+}_{i}(\tau)$ for $\tau \in \tNn$ is used to describe centering terms while the terms $\Ii^{-}_{i}(\tilde{\tau})$ for $\tilde{\tau} \in \Tr$ are only used so we can write tree expansions for derivatives generated by our renormalised products.

In particular, since the renormalised product associated to a local product $\locprod$ is defined in terms of the path built from it, it will be useful to extend this path to act on such $\Ii_{i}^{-}(\tilde{\tau})$ trees and the natural action to choose here will be different than the action of the path on $\Ii_{i}^{+}(\tau)$ trees.  
 
Another difference between these two sets of derivative edges is that while we adopted the convention that, for any $\tau \in \Tr \setminus \tNn$, one has $\Ii^{+}_{i}(\tau) = 0$. 
We do not adopt the same convention for $\Ii^{-}_{i}(\tau)$.

For convenience we will treat the symbols $\Ii_{i}^{-}(\X_{i})$ and $\Ii^{+}_{i}(\X_{i})$ as the same and also adopt the convention that $\Ii_{i}^{-}(\one) = 0$.
We also extend our notion of order to $\Tlm$ by setting, for $\tau \in \Tr$, $|\Ii_{i}^{-}(\tau)|=|\tau|+1$. 
 
We extend any local product $\locprod$ to the new trees we have added in $\Tlm$ by setting, for $1 \le i \le d$, 
\begin{equation}\label{eq: extension of local product 2}
\locprod_{z} \Ii_{i}^{-}(\tau)=
\begin{cases}	
\partial_{i}
(\mathcal{L}^{-1} \locprod_{\bullet}\tau)(z)
\; & \textnormal{ if }\tau \in \Tr\;,\\
1 & \textnormal{ if }\tau = \X_{i}\;.
\end{cases}
\end{equation}
\subsection{Operations on $\Ii_{i}^{-}$ trees}
Given a local product $\locprod$, we extend the corresponding path to $\Ii_{i}^{-}$ trees by setting, for any $\tau \in \Tr$ and $1 \le i \le d$, 
\begin{equation}\label{eqs: path definition addendum}
\path_{z,w} \Ii^{-}_{i}(\tau)
=
\big( \locprod_{z} \otimes \locprod^{\rec}_{w}
\big)
\Delta \Ii^{-}_{i}(\tau)\;. 
\end{equation}
where we extend the formulae of \eqref{eqs: recursive deltaplus} by setting, for $1 \le i \le d$, 
\begin{equation*}\label{eqs: recursive deltaplus-addendum}
\begin{split}	
\Delta 
\Ii^{-}_{i}(\tau)
&=
(\Ii^{-}_{i} \otimes \id)
\Delta \tau
+
\Ii^{+}_{i}(\X_{i}) \otimes \Ii^{+}_{i}(\tau),\ 
\tau \in \Tr\;.
\end{split}
\end{equation*}
\begin{remark}
We remark that our convention that $\Ii_{i}^{-}(\X_{i}) = \Ii^{+}_{i}(\X_{i})$ also seems natural since this guarantees \eqref{eq:coassoc} holds for $\sigma$ of the form $\Ii_{i}^{-}(\tau)$, but this observation will not play any role in our argument. 
\end{remark}
We then have the following easy lemma.
\begin{lemma}\label{lemma: derivatives do right thing}
Let $\locprod$ be a local product, then for any $1 \le i \le d$ and $\tau \in Ww \sqcup \Nn$, one has
\begin{equation}\label{eq: derivative of path}
\partial_{i}
\path_{z,w}\Ii(\tau)
=
\path_{z,w}\Ii^{-}_{i}(\tau)
\textnormal{ for any } \tau \in \Ww \cup \Nn\;,
\end{equation}
where the derivative $\partial_{i}$ above acts in the variable $z$.
In particular, one has 
\begin{equation}\label{eq:diagonal_deriv_vanish}
\path_{z,z}\Ii^{-}_{i}(\tau) = 0 \textnormal{ if }|\Ii^{-}_{i}(\tau)| > 0
\end{equation}
. 
 
Moreover, if $\Theta: \Nn \rightarrow C^{\infty}$ has the property that, for some $1 < \gamma <2$, we have $[U^{\one}]_{\gamma} < \infty$ (where $U^{\one}_{\gamma-2}$ is defined as in \eqref{def:U}) then we have that
\begin{equation}\label{eq: derivative of reconstruction}
\partial_{i} \Theta_{z}(\one)
=
\sum_{\tau \in \Nn
\atop
|\tau| \le -1}
\Theta_{z}(\tau)
\locprod_{z,z}\Ii^{-}_{i}(\tau)\;.
\end{equation}
\end{lemma}
\begin{proof}
The first statement \eqref{eq: derivative of path} is a straightforward computation using \eqref{eq: derivative of path} and \eqref{eq: extension of local product 2} and \eqref{eqs: path definition addendum}.
The second statement \eqref{eq:diagonal_deriv_vanish} then follows from the first one and the order bound $[\locprod: \Ii(\tau)]$ for such $\tau$. 
 
Finally, the third statement \eqref{eq: derivative of reconstruction} follows immediately from combining \eqref{eq: derivative of path} and Lemma~\ref{lemma: derivative of modeled distribution}. 
\end{proof}
\subsection{A recipe for local products}
With this notation in hand, our recipe for building a local product $\locprod$ will be to first specify the smooth function $\locprod_{z}\Xi$ and then inductively define, for $\tau \in \prodtree$ (recall that $\prodtree$ was defined in Definition~\ref{def: small class of trees}), 
\begin{equation}\label{eq: inductive def of renormalised model}
\locprod_{z} \tau  = \locprod_{z} R \tau 
\end{equation}
where $R: \prodtree \rightarrow \Alg(\Tlm)$, and on the right hand side, we apply $\locprod_{z}$ multiplicatively over the planted trees appearing in the forests of $\Alg(\Tlm)$ and use the conventions of Section~\ref{subsec: ext of loc prods} and \eqref{eqs: recursive deltaplus-addendum} to reduce the right-hand side to evaluating $\locprod_{z}$ on $\Nn \cup \Ww$.
For this to be a well-defined way to construct local products the map $R$ must satisfy the following two criteria
\begin{itemize}

\item For the induction \eqref{eq: inductive def of renormalised model} to be closed, it is natural to enforce that $R$ should have a triangular structure in that, for any $\tau \in \prodtree$, any planted tree appearing in a forest appearing in $R\tau$ should be of the form $\Ii(\tilde{\tau})$ or $\Ii^{-}_{i}(\tilde{\tau})$ with $\tilde{\tau}$ strictly fewer edges than  $\tau$. 

\item In order for $\locprod_{z}$ to be invariant under permutations of non-commutative tree products it also natural to enforce that $R$ act covariantly with respect to such permutations.
 
\end{itemize}
If we have a map $R$ as above and use it to build a local product $\locprod$ using \eqref{eq: inductive def of renormalised model} then we say $\locprod$ is built from $R$. 

For what follows it is useful to define the map $q_{F}$ which takes tree products to forest products, namely $q_{F}$ maps $\mNn \cup \mWw \ni \Ii(\tau_{1}) \Ii(\tau_{2})  \Ii(\tau_{3}) \mapsto \Ii(\tau_{1}) \cdot \Ii(\tau_{2}) \cdot \Ii(\tau_{3}) \in \Alg(\Tl)$.
\begin{remark}\label{rem: multiplicative local product}
One possible choice for a renormalisation operator $R_{\tiny \mathrm{mult}}$ is setting, for each $\tau \in \prodtree$, $R_{\tiny \mathrm{mult}}\tau = q_{F} \tau$.
If one uses $R_{\tiny \mathrm{mult}}$ to build a local product $\locprod$ then it follows that $\locprod$ is a multiplicative local product. 
\end{remark}
However, in order to allow more flexibility than a multiplicative local product but still make it easy to show that that the product $\phi^{\circ_{\locprod} 3}$ in \eqref{eq: renormalised product phi equation} admits a nice formula, 
we impose a structural assumption
on the operator $R$. 

This assumption can be expressed in terms of a slightly modified version of our earlier defined coproduct.

\subsection{A modified coproduct and local renormalisation operators}
The modified coproduct is defined with a map $C_-$, modification of $C_+$.
$C_-:(\bar{\tau},\tau)\in \Tt\times \Tt\rightarrow \Ff$ is given by the table below.
\begin{table}[h!]
\centering
\begin{tabular}{|c|c c c c|}
\hline 
$\tau\setminus\bar{\tau}$ & $\one$ & $\x_i$ & $\Xi$ & $\Ii(\bar{\tau}_1)\Ii(\bar{\tau}_2)\Ii(\bar{\tau}_3)$ \\ 
\hline 
$\one$ & $\Ii(\one)$ & $0$ & $0$ & $0$ \\ 
$\x_j$ & $\Ii(\x_j)$ & $\Ii^-_i(\x_j)$ & $0$ & $0$ \\ 
$\Xi$ & $\Ii(\Xi)$ & $\Ii^-_i(\Xi)$ & $1$ & $0$ \\ 
$\Ii(\tau_1)\Ii(\tau_2)\Ii(\tau_3)$ & $\Ii(\tau)$ & $\Ii^-_i(\tau)$ & 0 & $C_-(\bar{\tau}_1,\tau_1)C_-(\bar{\tau}_2,\tau_2)C_-(\bar{\tau}_3,\tau_3)$ \\ 
\hline 
\end{tabular} 
\vskip1ex
\caption{
This table gives a recursive definition of $C_-(\bar\tau, \tau)$. Possible values of $\tau$ are displayed in the first column, while possible values of $\bar{\tau}$ are shown in 
the first row. The corresponding values of $C_-(\bar\tau, \tau)$ are shown in the remaining fields. 
 }
\label{table:co-prod2} 
\end{table}

The difference between $C_+$ and $C_-$ is the removal of the projection on positive planted trees. Similarly, we never assume that $\Ii^-_i(\tau)=0$ if $\tau\not\in \tilde{\Nn}$.	$C_-$ also satisfies Lemma \ref{lem:lemC}, with the first implication in \eqref{prop_C} being an equivalence in this case.

The following lemma, which follows in an immediate way from the  definition of our sets of trees $\Nn$ and $\Ww$, 
will be useful when we try to drive an explicit formula for $ \phi^{\circ_{\locprod} 3}$.  
\begin{lemma}\label{lemma: tree self-sim 1}
For any fixed $\bar{\tau} \in \mNn \cup \mWw $,
\begin{align*}
&
p_{\leqslant 0}\sum_{\tau \in \mNn \cup \mWw} 
C_{-}(\bar{\tau},\tau)
\\
&=\begin{cases} 
\big(
\sum_{\tau \in  \Ww}
\Ii(\tau) 
+
\Ii(\one)
\big)&  
\textnormal{ if }\numone(\bar{\tau}) = 1\;,\\
\big(
\sum_{\tau \in  \Ww}
\Ii(\tau) 
+
\Ii(\one)
\big) \cdot 
\big(
\sum_{\tau \in  \Ww}
\Ii(\tau) 
+
\Ii(\one)
\big)
&
\textnormal{ if }\numone(\bar{\tau}) = 2\;,\\
\sum_{\tau \in \Nn \cup \Ww 
\atop
|\tau| \le -1}
\Ii_{i}^{-}(\tau)
& 
\textnormal{if }\numpolyi(\bar{\tau}) = 1\;.
\end{cases}
\end{align*}
Above, $p_{\leqslant}: \Alg(\Tlm) \rightarrow \Alg(\Tlm)$ is the projection that annihilates any forest of planted trees that contains a planted tree of strictly positive degree.
\end{lemma}
\begin{definition}\label{def: local renormalisation operator}
Given a map $r: \prodtree \rightarrow \R$ such that $r$ is invariant under permutations we define a corresponding map $R:\prodtree \rightarrow \Alg(\Tlm)$ by defining 
\begin{equation}\label{eq: acceptable renorm}
R(\tau)= q_{F} \tau +\sum_{\tau'\in\ \prodtree }r(\tau')C_-(\tau',\tau)\;.
\end{equation}

Note that counterterm maps $r$ and local renormalisation operators $R$ determine each other uniquely. 
\end{definition}
We then immediately have the following lemma.

\begin{lemma}\label{lemma: slight extension of inductive def of renormalised model}
Let $R$ be a local renormalisation operator. . 

Given a local product $\locprod$ built from $R$, the formula \eqref{eq: inductive def of renormalised model} defining $\locprod$ for $\tau \in \prodtree$ actually also holds as an identity for $\tau \in \mNn \setminus \prodtree$ where $R$ is itself is extended to $\mNn \cup \mWw$ by applying the formula \eqref{eq: acceptable renorm}. 
\end{lemma}
Suppose we are given a $\locprod$ built from $R$. 
Then the formula \eqref{eq: inductive def of renormalised model} allows us to compute the action of $\locprod_{\bullet}$ on any element $\tau \in \Nn \cup \Ww$ in terms of its actions on simpler trees. 
However, the starting formula for $\phi^{\circ_{\locprod} 3}$ involves the action of the corresponding path $\locprod_{z,z}$. 
Therefore in order to work out an explicit formula for $\phi^{\circ_{\locprod} 3}$ it would be good to have an analog of \eqref{eq: inductive def of renormalised model} for paths $\locprod_{\bullet,\bullet}$ instead of just the underlying local product $\locprod_{\bullet}$. 
Heuristically the idea for getting such a formula is showing that the action of a local renormalisation operator will ``commute'' with our centering operations. 
To this end we have the following lemma.

\begin{lemma}\label{lemma: renorm commute with structure group}
Let $R$ be a local renormalisation operator. 
Then one has, for any $\tau \in \mNn \cup \mWw$, the identity
\begin{equation}\label{eq: renorm commute with structure group}
\Delta R \tau = (R \otimes \id) \Delta \tau\;.
\end{equation}
\end{lemma}
\begin{proof} 
\[
\Delta R(\tau)=
\Delta q_{F}\tau + \sum_{\bar{\bar{\tau}}\in\Tt}r(\bar{\bar{\tau}})\Delta C_-(\bar{\bar{\tau}},\tau).
\]
We split the sum above depending on $\numone(\bar{\bar{\tau}})$ and $\numpoly(\bar{\bar{\tau}})$, for $\bar{\bar{\tau}}\leqslant\tau$. If $\numone(\bar{\bar{\tau}})=\numpoly(\bar{\bar{\tau}})=0$, then $C_-(\bar{\bar{\tau}},\tau)=\delta_{\{\tau=\bar{\bar{\tau}}\}}.$ 

If $\numone(\bar{\bar{\tau}})=1$, then we have $C_-(\bar{\bar{\tau}},\tau)=\Ii(\sigma)$ where $\sigma\subset\tau$ and 
\[
\Delta \Ii(\sigma)=\sum_{\tilde{\tau}\in\mNn\cup\Ww}\Ii(\tilde{\tau})\otimes C_+(\tilde{\tau},\tilde{\sigma}).
\]
For each element $\tilde{\tau}$ in this sum, we define a corresponding $\bar{\tau}$ by replacing the occurrence of $\sigma$ in $\tau$ identified above by $\tilde{\tau}$. We have $\bar{\bar{\tau}}\leqslant\bar{\tau}\leqslant\tau$ and by the inductive formulas,
\[
\Ii(\tilde{\tau})=C_-(\bar{\bar{\tau}},\bar{\tau})
\]
and
\[
C_+(\tilde{\tau},\tilde{\sigma})=C_+(\bar{\tau},\tau).
\]
The following picture is a representation of $\tau$ and the relation between its different subtrees, for one choice of $\tilde{\tau}$, to give an intuition of the proof. In blue is $C_+(\tilde{\tau},\tilde{\sigma})=C_+(\bar{\tau},\tau)$, which in this example is a product of two planted trees.
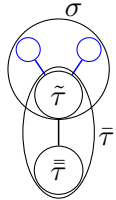
\begin{center}
\begin{tikzpicture}
\path (0,0) node[circle, inner sep=3pt,draw](x) {$\bar{\bar{\tau}}$};
\path (0,1) node[circle, inner sep=3pt,draw](y) {$\tilde{\tau}$};
\path (-.4,1.6) node[circle, inner sep=3pt,draw=blue, minimum size=9pt](z1) {};
\path (.4,1.6) node[circle, inner sep=3pt,draw=blue, minimum size=9pt,](z2) {};
\draw[line width=.2mm,blue] (y) -- (z1);
\draw[line width=.2mm,blue] (y) -- (z2);
\draw[line width=.2mm] (y) -- (x);
\draw (0,0.5) circle [x radius=0.47, y radius=0.87];
\draw (0,1.34) circle (0.666);
\path (0.63,0.45) node[circle, inner sep=3pt](h) {$\bar{\tau}$};
\path (0.2,2.13) node[circle, inner sep=3pt](k) {$\sigma$};
\end{tikzpicture}
\end{center}

If $\numpoly(\bar{\bar{\tau}})=1$, the same holds by replacing $\Ii$ by $\Ii^-_i$ in the argument.

If $\numone(\bar{\bar{\tau}})=2$, then we write $C_-(\bar{\bar{\tau}},\tau)=\Ii(\sigma_1)\cdot\Ii(\sigma_2)$ where $\sigma_i\subset\tau$ and for $i=1,2$,
\[
\Delta \Ii(\sigma_i)=\sum_{\tilde{\tau}_i\in\mNn\cup\Ww}\Ii(\tilde{\tau}_i)\otimes C_+(\tilde{\tau}_i,\tilde{\sigma}_i).
\]
For each element $\tilde{\tau}_1$ and $\tilde{\tau}_2$, we define $\bar{\tau}$ by replacing $\sigma_i$ by $\tilde{\tau}_i$ in $\tau$. We have $\bar{\bar{\tau}}\leqslant\bar{\tau}\leqslant\tau$ and by the inductive formulas,
\[
\Ii(\tilde{\tau}_1)\cdot\Ii(\tilde{\tau}_2)=C_-(\bar{\bar{\tau}},\bar{\tau})
\]
and
\[
C_+(\tilde{\tau}_1,\tilde{\sigma}_1)\cdot C_+(\tilde{\tau}_2,\tilde{\sigma}_2)=C_+(\bar{\tau},\tau).
\]
In all the cases discussed, we can index the sum induced by the coproduct in terms of $\bar{\tau}$ instead of $\tilde{\tau}$ or $\tilde{\tau}_1,\tilde{\tau}_2$. Permutation of that sum with the sum over $\bar{\bar{\tau}}$ then gives:
\begin{align*}
\Delta R(\tau)
=&\Delta q_{F} \tau +\sum_{\bar{\tau}\in\mNn\cup\Ww }(\sum_{\bar{\bar{\tau}}\in\Tt }r(\bar{\bar{\tau}}) C_-(\bar{\bar{\tau}},\bar{\tau}))\otimes C_+(\bar{\tau},\tau)\\
=&\sum_{\bar{\tau}\in\mNn\cup\mWw }( q_{F} \bar{\tau}+\sum_{\bar{\bar{\tau}}\in\Tt}r(\bar{\bar{\tau}}) C_-(\bar{\bar{\tau}},\bar{\tau}))\otimes C_+(\bar{\tau},\tau)\\
=& \sum_{\bar{\tau}\in\mNn\cup\mWw }R(\bar{\tau})\otimes C_+(\bar{\tau},\tau)\\
=&(R\otimes \id)\Delta \tau\;,
\end{align*}	 
where above we adopt the convention that $r(\tau') = 0$ if $\tau' \not \in \prodtree$.
\end{proof}
With this identity we can now give an analog of \eqref{eq: inductive def of renormalised model} for our paths. 
\begin{lemma}\label{lemma: renormalised formula for local path}
Suppose that the local product $\locprod$ was built from an local renormalisation map $R$. 
Then, for any $x,y \in \R^{d}$, and tree $\tau \in \mNn \cup \mWw$ one has
\begin{equation}\label{eq: recursive formula for renormalised rough path}
\path_{x,y} \tau
=
\path_{x,y} R \tau
\end{equation} 
where on the right hand side we extend $\path_{x,y}$ to forests of planted trees multiplicatively. 
\end{lemma}
\begin{proof}
Our proof is by induction in the size of $\tau$. 
The bases cases where $\num_{\tau} = 3$ are straightforward to check by hand. 
For the inductive step, we note that one has
\begin{equation*}
\path_{x,y} \tau =
(\locprod_{x} 
\otimes 
\locprod^{\rec}_y)
\Delta \tau
=
(\locprod_{x}R 
\otimes 
\locprod^{\rec}_y)
\Delta \tau\\
=
(\locprod_{x} 
\otimes 
\locprod^{\rec}_y)
\Delta R\tau
=
\path_{x,y} R\tau\;,
\end{equation*}
where in the second equality we used  Lemma~\ref{lemma: slight extension of inductive def of renormalised model} and in the third equality we used Lemma~\ref{lemma: renorm commute with structure group}. 
\end{proof}
\begin{remark}\label{rem: phi43 renorm}
We describe how the renormalisation of $\Phi^4_3$ (which in our setting corresponds to fixing $\delta = 1/2-$ with Gaussian noise) used in previous works such as \cite{Pedestrians} corresponds to a choice of a local renormalisation operator $R$.  

We define $\prodtree_{\tiny{\textnormal{wick}}}$ to be the three different elements of $\mNn$ obtained by permuting the tree product in $\Ii(\one)\Ii(\Xi)^{2}$, that is 
\[
\prodtree_{\tiny{\textnormal{wick}}} = \{\Ii(\one)\Ii(\Xi)^{2},\Ii(\Xi)\Ii(\one)\Ii(\Xi),\Ii(\Xi)^{2}\Ii(\one)\}
\]
 and similarly define $\prodtree_{\tiny{\textnormal{sunset}}}$ to be the collection of the $9$ different elements of $\mNn$ which are obtained by permutations of the tree product in 
 \[
 \Ii(\Xi)\Ii\big(\Ii(\Xi)\Ii(\one)\Ii(\Xi)\big)\Ii(\Xi) =
 \<2K*2OAS_black>\;.
 \] 
There are nine elements because there are three different orders for each of the two tree products appearing in this tree, for instance one also has 
\[
\Ii\big(\Ii(\one)\Ii(\Xi)^{2}\big)\Ii(\Xi)^{2} \in \prodtree_{\tiny{\textnormal{sunset}}}\;.
\]
The corresponding counterterm map $r$ is given by 
\[
r(\tau)
=
\begin{cases}
- C_{\tiny{\textnormal{wick}}} 
&
\textnormal{ if }\tau \in \prodtree_{\tiny{\textnormal{wick}}},\\
- C_{\tiny{\textnormal{sunset}}}
&
\textnormal{ if } \tau \in 
\prodtree_{\tiny{\textnormal{sunset}}},\\
0 & \textnormal{ otherwise.}
\end{cases}
\]
where one has
\begin{equation*}
\begin{split}
C_{\tiny{\textnormal{wick}}}
=& \mathbb{E}[(\mathcal{L}^{-1}\xi)(0)^2].\\
C_{\tiny{\textnormal{sunset}}}
=&\mathbb{E}[\theta(0)\mathcal{L}^{-1}\theta(0)],
\end{split}
\end{equation*}
where $\xi$ is our (regularised) noise,   $P$ is the space-time Green's function for the heat kernel and $\theta$ is defined by
\[
\theta(z)=(\mathcal{L}^{-1}\xi)(z)^2-\mathbb{E}[(\mathcal{L}^{-1}\xi)(0)^2]=(\mathcal{L}^{-1}\xi)(z)^2-C_{\tiny{\textnormal{wick}}}.
\]

The promised local renormalisation operator is then given by building $R$ from $r$ as in \eqref{eq: acceptable renorm}.

As an example, we compute
\begin{equation*}
\begin{split}	
R\ \<3_black>
=&
\<1_black> \<1_black> \<1_black>
- 
3
C_{\tiny{\textnormal{wick}}}
\<1_black> 
\quad
\textnormal{and}
\quad
R\ \<2K*2OAS_black>
=
\<1_black> 
\<K*3one_black>
\<1_black>
-
C_{\tiny{\textnormal{wick}}}
\<K*3one_black>
-
C_{\tiny{\textnormal{sunset}}}
\<K*one_black>\;.
\end{split}
\end{equation*} 
\end{remark}
%
\subsection{Formula for the renormalised cube}
\label{ss:RE}
%
The next proposition gives the explicit formulae for our renormalised product that promised at the beginning of this section. 
\begin{proposition}\label{prop: local renorm eq}
Let $\locprod$ be built from a local renormalisation operator $R$.

Fix smooth functions $v_{\one}, v_{\X_{1}},\dots,v_{\X_{d}}: \R \times \R^{d} \rightarrow \R$ and let
\[
\Phi(z)
=
\sum_{\tau \in \Nn \cup \Ww}
\Upsilon_{z}(\tau) \Ii(\tau)
\] 
where $\Upsilon$ is defined in terms of the parameters $v_{\one}$ and $v_{\X}$, and 
\begin{equation*}
\phi(z)
= \locprod_{z,z}\Phi(z)
=
\Upsilon_{z}(\one) + 
\sum_{\tau \in \Ww}
\Upsilon(\tau) \locprod_{z}\Ii(\tau)\;.
\end{equation*}
Moreover, suppose that, for some $1 < \gamma < 2$, if $U^{\one}_{\gamma-2}$ as in \eqref{def:U} with $\Theta_{\bullet}(\cdot) = \Upsilon_{\bullet}(\cdot)|_{\Nn}$, we have that $[U^{\one}]_{\gamma} < \infty$. 

Then, if we define $\phi^{\circ_{\locprod} 3}$ as in Definition~\ref{def: renormalised product} using $\Theta$ as our tree expansion for $\phi$, we then have
\begin{equation}\label{Mphi4_abstract}
\phi^{ \circ_{\locprod} 3}(z)  = \phi^3(z) - r_1
- r_\Phi \phi (z) - r_{\Phi^2}\phi^2(z)
- \sum_{i=1}^{d} r_{\partial_{i}\Phi} \partial_{i}\phi(z),
\end{equation}
where
\begin{equation}\label{eq: renorm constants} 
\begin{split}
r_1=
\sum_{\bar{\tau}\in\mNn\cup\mWw,
\atop
\numone(\bar{\tau}) + \numpoly(\bar{\tau})=0}(-1)^\frac{\num(\bar{\tau})-1}2r(\bar{\tau}),\quad
r_\Phi=\sum_{\underset{\numone(\bar{\tau})=1}{\bar{\tau}\in\mNn}}(-1)^\frac{\num(\bar{\tau})-1}2r(\bar{\tau}), \\
r_{\Phi^2}=\sum_{\underset{\numone(\bar{\tau})=2}{\bar{\tau}\in\mNn}}(-1)^\frac{\num(\bar{\tau})-1}2r(\bar{\tau}) ,\quad
r_{\partial_{i}\Phi}=\sum_{\numpolyi(\bar{\tau})=1 
\atop
\bar{\tau}\in\mNn}(-1)^\frac{\num(\bar{\tau})-1}2r(\bar{\tau}).
\end{split}
\end{equation}
Here $r: \prodtree \rightarrow \R$ is the map from which $R$ is built. 
\end{proposition}
\begin{proof}
We have
\begin{equation*}
\begin{split}	
\phi^{\circ_{\locprod}3}(z)
=&
\sum_{\tau_{1}, \tau_{2},\tau_{3} \in \Nn \cup \mWw}
\Upsilon_{z}(\tau_{1})
\Upsilon_{z}(\tau_{2})
\Upsilon_{z}(\tau_{3})
\path_{z,z}
\Ii(\tau_{1})
\Ii(\tau_{2})
\Ii(\tau_{3})\\
&
=
\sum_{\tau \in \mNn \cup \mWw}
\Upsilon_{z}(\tau) \path_{z,z}\tau\\
=& 
\sum_{\tau \in \mNn \cup \mWw}
\Upsilon_{z}(\tau) \path_{z,z}R\tau\\
=&
\sum_{\tau\in \mNn\cup \mWw}\Upsilon_{z}(\tau)( \path_{z,z}q_{F}\tau+\sum_{\bar{\tau}\in\Tt}r(\bar{\tau}) \path_{z,z}C_-(\bar{\tau},\tau))\;,
\end{split}
\end{equation*}
where the first equality follows from the definition of renormalised local products, the second equality comes from Lemma~\ref{lemma: self-similarity of coherence}, and the third equality comes from Lemma~\ref{lemma: renormalised formula for local path}.
 
By appealing to Lemma~\ref{lemma: self-similarity of coherence} once more, we can rewrite the first term of the last line above as
\begin{equation*}
\begin{split}	
-
\path_{z,z}q_{F}
\sum_{\tau\in \mNn \cup \mWw}\Upsilon_{z}(\tau)\tau
=&
\path_{z,z}q_{F}
\big(
\sum_{\tau \in \Nn\cup \Ww}\Upsilon_{z}(\tau)\Ii(\tau)
\big)^{3}\\
=&
\big(
\sum_{\tau \in \Nn\cup \Ww}\Upsilon_{z}(\tau)\path_{z,z}\Ii(\tau)
\big)^{3} = \phi^{3}(z)\;.
\end{split}
\end{equation*}
By using Lemma~\ref{lemma: tree self-sim 1} (note that $\path_{z,z}p_{\le 0} = \path_{z,z}$ on $\Alg(\Tlm)$) followed by equation \eqref{eq: coherence relation for upsilon +}, and using the fact that $\path_{z,z}\sigma=0$ if $|\sigma|>0$,  we have  
\begin{align}
\sum_{\tau\in \mNn \cup \mWw}&\Upsilon_z(\tau)\sum_{\bar{\tau}\in\mNn \cup \mWw}r(\bar{\tau})\path_{z,z}C_-(\bar{\tau},\tau)\\
&=\sum_{\bar{\tau}\in\mNn\cup\mWw,
\atop
\numone(\bar{\tau}) + \numpoly(\bar{\tau}) =0}(-1)^\frac{\num(\bar{\tau})-1}2r(\bar{\tau}) 1 \notag\\
&+\sum_{\underset{\numone(\bar{\tau})=1}{\bar{\tau}\in\mNn}}(-1)^\frac{\num(\bar{\tau})-1}2r(\bar{\tau}) \sum_{\tau \in \{\one\} \cup \Ww} \Upsilon_{z}(\tau)\path_{z,z}\Ii(\tau)\notag\\
&+\sum_{\underset{\numone(\bar{\tau})=2}{\bar{\tau}\in\mNn}}(-1)^\frac{\num(\bar{\tau})-1}2r(\bar{\tau}) \big( \sum_{\tau \in \{\one\} \cup \Ww} \Upsilon_{z}(\tau)\path_{z,z}\Ii(\tau) \big)^{ 2} \notag\\
&+
\sum_{i=1}^{d}
\sum_{\underset{\numpolyi(\bar{\tau})=1}{\bar{\tau}\in\mNn}}(-1)^\frac{\num(\bar{\tau})-1}2r(\bar{\tau}) 
\sum_{\tau \in \Nn \cup \Ww 
\atop
|\tau| \le -1}
\Upsilon_{z}(\tau)\path_{z,z}\Ii^{-}_{i}(\tau) \notag.
\end{align}
We then obtain the desired result by observing that, for the second and third terms on the right hand side above, 
\[
\sum_{\tau \in \{\one\} \cup \Ww} \Upsilon_{z}(\tau)\path_{z,z}\Ii(\tau)
=
\Upsilon_{z}(\one)
+
\sum_{\tau \in \Ww} \Upsilon_{z} \path_{z} \Ii(\tau)
=
\phi(z)
\] 
and, for the third term on the right hand side above, we have, for $1 \le i \le d$, 
\begin{equation*}
\begin{split}
\sum_{\tau \in \Nn \cup \Ww 
\atop
|\tau| \le -1}
\Upsilon_{z}(\tau)\path_{z,z}\Ii^{-}_{i}(\tau)
=&
\sum_{\tau \in \Ww}
\Upsilon(\tau)\locprod_{z}\Ii_{i}^{-}(\tau)
+
\sum_{\tau \in \Nn
\atop
|\tau| \le -1}
\Upsilon_{z}(\tau)\path_{z,z}\Ii^{-}_{i}(\tau)\\
=&
\sum_{\tau \in \Ww}
\Upsilon_{z}(\tau) \partial_{i}\locprod_{z}\Ii(\tau)
+
\partial_{i}\Upsilon_{z}(\one) 
=
\partial_{i}\phi(z)\;.
\end{split}
\end{equation*}
For the first equality of the second line above we used Lemma~\ref{lemma: derivatives do right thing} - in particular, \eqref{eq: derivative of reconstruction} - with $\Theta_{\bullet}(\cdot) = \Upsilon_{\bullet}(\cdot)|_{\Nn}$. 
\end{proof}
\begin{remark}
Under the assumptions of Proposition~\ref{prop: local renorm eq} one can also show that each one of the renormalised products in \eqref{eq: remainder equation} can also be expressed in terms of local polynomials of $v(z)$, $\{\partial_{i}v(z)\}_{i=1}^{d}$, and $\{\locprod_{z}\Ii(\tau): \tau \in \Ww\}$.

However, we refrain from doing this because the index sets for the summations that define the analogs of the constants \eqref{eq: renorm constants} become quite complicated.  
\end{remark}

\begin{remark}\label{rem: phi43 renorm equation}
Returning to the example of $\Phi^4_3$ described in Remark~\ref{rem: phi43 renorm}, one then sees that $r_{\partial_{i}\Phi} = r_{\Phi^{2}} = r_{1} = 0$ for all $1 \le i \le d$ and 
\[
r_{\Phi}
=
\sum_{\tau \in \prodtree_{\tiny{\textnormal{wick}}}}
(-1)^{\frac{3-1}{2}}
(-C_{\tiny{\textnormal{wick}}})
+
\sum_{\tau \in \prodtree_{\tiny{\textnormal{sunset}}}}
(-1)^{\frac{5-1}{2}}
(-C_{\tiny{\textnormal{sunset}}})
=
3C_{\tiny{\textnormal{wick}}}
-
9C_{\tiny{\textnormal{sunset}}}\;.
\]
\end{remark}

\section{Main result}
\label{s:MR}

\subsection{Statement of main theorem}
\label{ss:MT}

We recall the definition of the parabolic cylinders $D$ and $D_R$:
\[
D=(0,1)\times \{|x|<1\},\quad D_R=(R^2,1) \times \{|x|<1-R\}
\]
We introduce the parabolic ball of center $z=(t,x)$ and radius $R$ in this metric $d$, looking only into the past: 
\begin{equation}\label{paraball}
B(z,R)=\{\bar{z}=(\bar{t},\bar{x})\in\R\times\R^d, d(z,\bar{z})<R,\bar{t}<t\}.
\end{equation}

Note that for $R'<R\leq 1$ we have $D_{R'}+B(0,R'-R)\subset D_R$. 
 
The notation $\|\cdot\|$ denotes the $L^\infty$ norm on the whole space $\R \times \R^d$ and for any open set $\|\cdot\|_B$ the norm of the restriction of the function to $B$.

\begin{theorem}\label{th:main theorem}
There exists a constant $C$ such that if $v$ is a pointwise solution on $D$ to the remainder equation driven by a local product $\path$, according to Definition~\ref{def: remainder solution} then 
\begin{equation}\label{eq:mainth}
\|v\|_{D_R}\leq C\max\Big\{\frac1R,[\path;\tau]^\frac1{\delta \numnoise(\tau)}, |\tau|<0, \numnoise(\tau)\neq 0\Big\}.
\end{equation}
\end{theorem}

This theorem generalises to an arbitrary domain $\tilde{D}$ in the following way: the local path is defined in a similar way, only replacing the cut-off function $\rho$ by a cut-off function that has value $1$ on a $1$-enlargement of $\tilde{D}$, and vanishes on a $2$ enlargement of the set. Then for every point in $\tilde{D}$, one can obtain a bound depending only on the path by applying a translated version of the theorem, for $R=1$. 

The following corollary is a reformulation of this theorem following from Definition~\ref{def: solution}.

\begin{corollary}\label{th:main theorem 2}
There exists a constant $C$ such that if $\phi$ is a pointwise solution to equation~\eqref{eq: renormalised product phi equation} driven by a local product $\path$ then for $v=\phi-\sum_{w\in\Ww}\Upsilon(w)\path\Ii(w)$, the bound \eqref{eq:mainth} holds.
\end{corollary}

The following result is a particular case of the local product introduced in Section \ref{sec:useful class of local products}. It follows from Proposition~\ref{prop: local renorm eq}. 

\begin{corollary}\label{th:main theorem 3}
Let $R$ be a local renormalisation operator, and $\path$ be the local product built from $R$. There exists a constant $C$ such that if $\phi$ is a pointwise solution to 
\begin{equation}
\heat\phi=-\phi^3(z) + r_1
+ r_\Phi \phi (z) + r_{\Phi^2}\phi^2(z)
+ \sum_{i=1}^{d} r_{\partial_{i}\Phi} \partial_{i}\phi(z),
\end{equation}
where the coefficients $r_1,\  r_\Phi,\ r_{\Phi^2}$ and $r_{\partial_{i}}$ are given by \eqref{eq: renorm constants}, then for $v=\phi-\sum_{w\in\Ww}\Upsilon(w)\path\Ii(w)$, the bound \eqref{eq:mainth} holds.
\end{corollary}

The typical application of these results concerns paths $\path$ which are constructed from a Gaussian noise $\xi$. In this case, one typically has  that for a given tree $\tau$ the quantity $[\locprod;\tau]$ is in the (inhomogeneous)
 Wiener chaos of order $\numnoise(\tau)$. In particular, in this case one gets for some $\lambda >0$,
\[
\mathbb{E}[\exp(\lambda[\locprod;\tau]^\frac2{\numnoise(\tau)})]<\infty.
\]
This implies the following corollary:
\begin{corollary}
\label{cor:9-6}
If $\xi$ is a Gaussian noise and the path $\path$ is built from the local renormalisation operator, 
if $v$ is a pointwise solution to the remainder equation driven by a local product $\path$ on $D$, according to Definition~\ref{def: remainder solution} then there exists a constant $\bar{\lambda}$ such that
 \[
\mathbb{E}[\exp(\bar{\lambda}\|v\|_{D_\frac12}^{2\delta})]<\infty.
 \]
\end{corollary}

\begin{remark}
The results presented here also imply a bound for the corresponding elliptic equation in dimension $6-$, i.e.
\[
\Delta \phi=\phi^3-\xi   \qquad\qquad   x \in \R^6,
\]
where $\xi$ is a $6$-dimensional white noise which is slightly regularized (e.g. by applying $(1-\Delta)^{-\bar{\delta}}$ for an arbitrary $\bar{\delta}>0$).
The four and five dimensional versions of this equation where recently studied in \cite{GH,albeverio2019elliptic}.
Our Corollary~\ref{cor:9-6} can be applied directly, if $\phi$ is viewed as a stationary solution of the parabolic equation (i.e. with $\partial_t \phi=0$). However, 
to treat the elliptic case it would be more natural to define the action of $\locprod$ on $\Ii$ slightly differently in terms of the inverse Laplace, rather than 
the inverse heat operator. Such a change could  be implemented easily.

\end{remark}

\subsection{Proof of Theorem \ref{th:main theorem} }
\label{ss:MP}
We start the proof by specifying further the norms that we are using.
We often work with norms which only depend on the behaviour of functions / distributions on a fixed subset of time-space: if $B\subset\mathbb{R}\times\mathbb{R}^d$ is a bounded set, then the addition of a subscript $B$ such as $[U^\tau]_{\alpha,B}$ means that the correspnding supremum is restricted to variables in $ B$. The use of a third index $r$ as in $[U^\tau]_{\alpha,B,r}$ indicates that the supremum is restricted to $z$ and $\bar{z}$ at distance at most $r$ from each other.
Similarly for a function of two variable, $\|\cdot\|_{B,r}$ is the norm restricted to $z,\bar{z}\in B$ with $d(z,\bar{z})\leq r$.

We remind the reader that the notation $\tau$  always refers to an unplanted tree.
In particular, sums indexed by $|\tau|\in J$ for $J$ an interval only refer to unplanted trees of that order. Planted trees will be explicitly denoted $\Ii(\tau)$.

 We recall the remainder equation:
\begin{align*}
\heat v=-v^3-3\sum_{w\in\Ww}\Upsilon(w)v\circ_\path v\circ_\path \path_\bullet\Ii(w)\\
-3\sum_{w_1,w_2\in\Ww}\Upsilon(w_1)\Upsilon(w_2)v\circ_\path\path_\bullet\Ii(w_1)\circ_\path\path_\bullet\Ii(w_2)+\sum_{\tau\in \partial\Ww}\Upsilon(\tau)\path_\bullet\tau.
\end{align*}
Here we have introduced $\partial\Ww=\{\tau\in\Nn,\tau=\Ii(w_1)\Ii(w_2)\Ii(w_3),w_i\in\Ww\}$. Note that the product $v^3$ does not need to be expressed using the renormalised product $\circ_\path$ since $v$ is of positive regularity. All the factors $\Upsilon$ in there are just combinatorial factors $\pm1$ which is why we omitted the subscript variable $y$.

The first thing we do is to convolve this equation with the kernel $\Psi_L$ introduced in Section~\ref{ss:conventions} (see also Appendix~\ref{ss:RL}), 
and we obtain:

\begin{align}\label{eq:convolved}
&\heat v_L=-v_L^3+(v_L^3-(v^3)_L)-3\sum_{w\in\Ww}\Upsilon(w)(v\circ_\path v\circ_\path \path_\bullet\Ii(w))_L\\
&-3\sum_{w_1,w_2\in\Ww}\Upsilon(w_1)\Upsilon(w_2)(v\circ_\path\path_\bullet\Ii(w_1)\circ_\path\path_\bullet\Ii(w_2))_L-\sum_{\tau\in \partial\Ww}\Upsilon(\tau)(\path_\bullet\tau)_L.\notag
\end{align}

We are going to use the following maximum principle on this equation.
\begin{lemma}\label{lem_max}
Let $u$ be a continuous function defined on $\bar{D}$, for which the following holds point-wise in $(0,1]\times (-1,1)^d$:
\begin{equation}\label{eq:max rd 1}
(\partial_t-\Delta)u=-u^3+g,
\end{equation}
where $g$ is a bounded function.
We have the following point-wise bound on $u$, for all $(t,x)\in(0,1]\times (-1,1)^d$:
\begin{equation}\label{eq:max theorem 1}
| u(x,t) | \leqslant C\max\Big\{\frac1{\min\{\sqrt{t},(1-x_i),(1+x_i),i=1,2,3\}},\|g\|^\frac13\Big\},
\end{equation}
for some independent constant $C$.
\end{lemma}
This lemma is taken directly from \cite[Lemma 2.7]{2018arXiv181105764M}, 
and the proof can be found in there.

To apply this lemma, we will therefore need bounds on the commutator, which is easy enough:
\[|(v^3)_L-(v_L)^3|\leqslant \|v\|^2[v]_\alpha L^\alpha,\]
as well as quantities of the type:
\[
(\Y\tau)_L(x)\text{ where }\tau\in\partial\Ww,
\]
which is bounded by $[\tau]_{|\tau|}L^{|\tau|}\leq c\|v\|_{D_d}^{\delta \numnoise(\tau)}L^{-3+\delta \numnoise(\tau)}$, in view of Assumption~\ref{assumption} and \eqref{formula_hom}. The following two products will require more work:
\[
(v\circ_\path v\circ_\path \path_\bullet\Ii(w))_L\text{ with }w\in\Ww\]
and
\[(v\circ_\path\path_\bullet\Ii(w_1)\circ_\path\path_\bullet\Ii(w_2))_L\text{ with }w_1,w_2\in\Ww\]
Since the path $\Y$ is built such that two symmetric trees have the same image, studying those cases will be enough to be exhaustive.

We formulate the following assumption, for some $1>c>0$ to be fixed later in \eqref{schauder final} and \eqref{eq:apply_max2}.
\begin{assumption}\label{assumption}
For all $\tau\in\Tt$ such that $\numnoise(\tau)\neq 0$, 
\begin{equation}
[\path;\tau]\leq c\|v\|_D^{\delta \numnoise(\tau)}.
\end{equation}
\end{assumption}

\begin{remark}
Note that it is not necessary to do the proof of the main theorem by contradiction with this Assumption, but it simplifies greatly the computations and allows to write everything in powers of $\|v\|$. Alternatively a proof can be made by keeping all norms of trees in the computation, but that becomes very messy fast.
\end{remark}
Under Assumption~\ref{assumption}, we have by Lemma~\ref{lem:OB} the following lemma:
\begin{lemma}
For all $\tau\in\Tt$ such that $\numnoise(\tau)\neq 0$,  
\begin{equation}
[\path;\Ii(\tau)]\lesssim c\|v\|_D^{\delta \numnoise(\tau)},
\end{equation}
\end{lemma}
\begin{proof}
For $\tau\in \Ww$, it is immediate. \\
For $\tau\in\mNn$, one simply has to notice that for $\bar{\tau}\leq\tau$, $\numnoise(\tau)=\numnoise(\bar{\tau})+\numnoise(C_+(\bar{\tau},\tau))$ and then use induction on $\numnoise(\tau)$ to bound $[\path;C_+(\bar{\tau},\tau)]$.
\end{proof}

With this set-up, we manage to prove the following lemmas, which hold for any domain $D$, uniformly over $x\in D$.

\begin{lemma}\label{linear bound}
For $w_1,w_2\in\Ww$, there exists an $\epsilon>0$ such that for $J=[-2,-6-|w_1|-|w_2|+\epsilon)$, 
 under the Assumption~\ref{assumption},  
\begin{align}\label{e-linbound1}
\Big|(v\circ_\path&\path_\bullet\Ii(w_1)\circ_\path\path_\bullet\Ii(w_2))_L(x)-
\sum_{|\tau|\in J}\Big(\Upsilon_x(\tau)\path_{\bullet,x}(\Ii(\tau)\Ii(w_1)\Ii(w_2))\Big)_L(x)\Big|\notag\\
&\les c\sum_{|\tau|\in J}\|v\|_{D}^{\delta \numnoise(\tau+w_1+w_2)}[U^\tau]_{-6-|w_1|-|w_2|-|\tau|+\epsilon}L^\epsilon.
\end{align}
We also have, for $|\tau|\in J$,
\begin{align}
\label{e-linbound2}
\Big|\Upsilon_x(\tau)\Big(&\path_{\bullet,x}(\Ii(\tau)\Ii(w_1)\Ii(w_2))\Big)_L(x)\Big|\\
&\leq c\|v\|_{D}^{\numone(\tau)+\delta \numnoise(\tau+w_1+w_2)}\|v_\x\|_{D_d}^{\numpoly(\tau)}L^{6+|\tau|+|w_1|+|w_2|}\notag.
\end{align}
\end{lemma}

In this lemma we extend the notation $\numnoise$ to sums of trees linearly. The functions $\numone, \numpoly$ and $\num$ will also be extended similarly later.

\begin{lemma}\label{quadratic bound}
For $w\in\Ww$, there exists $\epsilon>0$ such that for $\tilde{J}=\{(a,b)\in [-2,-1]^2, a+b<-6-|w|+\epsilon\}$, under the Assumption~\ref{assumption}, 
\begin{align}\label{eq:rev-v2}
\Big|(v&\circ_\path v\circ_\path \path_\bullet\Ii(w))_L(x\notag)\\
&-\sum_{(|\tau_1|,|\tau_2|)\in \tilde{J}}\Upsilon_x(\tau_1)\Upsilon_x(\tau_2)\Big(\path_{\bullet,x}(\Ii(\tau_1)\Ii(\tau_2)\Ii(w))\Big)_L(x)\Big|\\
&\les c\sum_{(|\tau_1|,|\tau_2|)\in \tilde{J}}\|v\|_D^{\delta \numnoise(\tau_1+\tau_2+w)}[U^{\Ii(\tau_1)\Ii(\tau_2)\Ii(\Xi)}]_{-6-|w|-|\tau_1|-|\tau_2|+\epsilon}L^\epsilon.\notag
\end{align}
We also have, for $(|\tau_1|,|\tau_2|)\in \tilde{J}$ 
\begin{align}\label{eq:rev-v2.2}
\Big|\Upsilon_x(\tau_1)\Upsilon_x(\tau_2)&\Big(\path_{\bullet,x}(\Ii(\tau_1)\Ii(\tau_2)\Ii(w))\Big)_L(x)\Big|\\
&\leqslant c\|v\|_{D}^{\numone(\tau_1+\tau_2)+\delta \numnoise(\tau_1+\tau_2+w)}L^{6+|\tau_1|+|\tau_2|+|w|}.\notag
\end{align}
\end{lemma}
In both of these lemmas, the existence of the $\epsilon$ follows from the following remark.
\begin{remark}\label{rem:noninteger}
Our choice of $\delta$ is such that $\Ii(\one)\Ii(\one)\Ii(\one)$ is the only tree of order $0$. Therefore for any non-trivial product, the sum can be indexed over trees $\tau$ of order $|\tau|<\epsilon$, for some $\epsilon>0$. The renormalised product is therefore described up to positive order $\epsilon$.
\end{remark}

Applying the Schauder Lemma yields the following lemma:
\begin{lemma}\label{lemma: apply Schauder}
Under the Assumption~\ref{assumption}, for any $1>d_0>0$ 
we have
\begin{equation}\label{schauder final vx}
\sup_{d\leqslant d_0}d\|v_\x\|_{D_d}\lesssim  \|v\|_D,
\end{equation}
and
\begin{equation}\label{schauder final general}
\sup_{d\leq d_0}d^{\gamma-\beta-|\tau|+\numpoly(\tau)}[U^\tau]_{\gamma-\beta-|\tau|,D_d,d}\lesssim \|v\|_D^{\numone(\tau)+\numpoly(\tau)}.
\end{equation}
\end{lemma}

A few more computations allow to close this argument, with a specification of $d_0$ in \eqref{fixd0}.

\begin{lemma}\label{lemtheorem}
Under the Assumption~\ref{assumption}, there exist $\lambda>0$ such that 
\begin{equation}\label{conclusion}
\|v\|_{D_{\lambda \|v\|_D^{-1}}}\leq \frac{\|v\|_D}2.
\end{equation}

\end{lemma}
The final proof of the main theorem relies on a iteration of this result. We define a finite sequence $0=R_0<...<R_N=1$ by setting 
\[
R_{n+1}-R_n=\lambda\|v\|_{D_{R_n}}^{-1},
\]
as long as the $R_{n+1}$ defined that way stay less than $1$. We terminate this algorithm once it would produce $R_{n+1}\geqslant 1$ in which case we set $R_N=R_{n+1}=1$, or once Assumption~\ref{assumption} does not hold for $D'=D_{R_n}$. Note that $\|v\|_{D_{R_n}}^{-1}$ is strictly increasing so the sequence necessarily terminates after finitely many steps. Rewriting Lemma~\ref{lemtheorem} replacing $D$ by $D_{R_n}$ then gives the bounds for smaller and smaller boxes.
\[
\|v\|_{D_{R_{n+1}}}\leqslant\frac{\|v\|_{D_{R_n}}}2
\]
We now prove that Theorem~\ref{th:main theorem} holds for all $d=R_n, n=0...N$. If Assumption~\ref{assumption} does not hold for $N$ then it is immediate, else for $k<n$, $\|v\|_{D_{R_n}}\leqslant\|v\|_{D_{R_k}}2^{k-n}$ hence
\[
R_n=\sum_{k=0}^{n-1}R_{k+1}-R_k=\sum_{k_0}^{n-1}\lambda\|v\|_{D_{R_k}}^{-1}\leqslant\lambda\|v\|_{D_{R_{n-1}}}^{-1}\sum_{k_0}^{n-1}2^{k-n+1}\lesssim\|v\|_{D_{R_{n-1}}}^{-1}.
\]
This implies that for any $R\in (R_{n-1},R_n)$, $\|v\|_{D_R}\leqslant\|v\|_{D_{R_{n-1}}}\lesssim R_n^{-1}\leqslant R^{-1}$, which proves the theorem in that case.

If the end-point is $R_N=1$, we either have $R_{N-1}>\frac12$ or $R_N-R_{N-1}>\frac12$. In both cases $\|v\|_{D_{R_{N-1}}}\lesssim R_{N-1}^{-1}\lesssim 1$.


\section{Proof of the intermediate results}
\label{s:PMT}

\subsection{A technical lemma}
We first quantify the expansions given in equations \eqref{def:V},\eqref{def:V2} and \eqref{def:V'}, used now with $\Theta=\Upsilon$.

\begin{align*}
[V]_{\alpha}&=\sup_{x,y}\frac{|v_\one(x)-V_\alpha(y,x)|}{d(x,y)^\alpha},\\
[V^2]_{\alpha}&=\sup_{x,y}\frac{|v_\one(x)^2-V^2_\alpha(y,x)|}{d(x,y)^\alpha},\\
[V^{(i)}]_{\alpha}&=\sup_{x,y}\frac{|v_{\x_i}(x)-V^{(i)}_\alpha(y,x)|}{d(x,y)^\alpha}.
\end{align*}
For any domain $D$, we denote the restriction of this norm to $x,y\in D$ by adding the subscript $D$. A second subscript $d$ may be added when we  restrict to $x,y$ satisfying $d(x,y)<d$.

Using Theorem~\ref{theorem one zero} we have the identities.
\begin{equation}\label{eq:cases U}
[U^\tau]_{\gamma-|\tau|}=\begin{cases}
[V]_{\gamma-|\tau|}& \text{if }\numone(\tau)=1,\\
[V^2]_{\gamma-|\tau|} & \text{if }\numone(\tau)=2,\\
[V^{(i)}]_{\gamma-|\tau|} & \text{if }\numpolyi(\tau)=1,\\
0 & \text{if }\numone(\tau),\numpoly(\tau) = 0\;.
\end{cases}
\end{equation}

Using  Lemma~\ref{lem:relations V V2} and the Assumption~\ref{assumption} to replace all order bounds on trees in this lemma by powers of $\|v\|_{D}$, we get the following general bound for the norm of $U$. The bound in the case of $\numpoly(\tau)=1$ is a straightforward application of Lemma~\ref{corschauder}.

\begin{lemma}\label{lem:hol-relation}

Under the Assumption~\ref{assumption}, for any $\tau \in \Nn$, and $0<\beta<\gamma<2$, 
\begin{align}\label{eq:corshauder-applied}
&\sup_{d\leq d_0}d^{\gamma-\beta+2-\delta \numnoise(\tau)}[U^\tau]_{\gamma-\beta-|\tau|,D_d,d}\\
&\qquad \les \sup_{d\leq d_0}\Big[[V]_{\gamma,D_d,d}d^{\gamma}+{\mathbbm{1}}_{\{\gamma-\beta+2-\delta \numnoise(\tau)<1\}}d\|v_\x\|_{D_d}\notag\\
& \qquad +c\sum_{\gamma-\beta-\delta \numnoise(\tau)\leq|\bar{\tau}|<\gamma-2}d^{|\bar{\tau}|+2}\|v\|_{D}^{\numone(\bar{\tau})+\numnoise(\bar{\tau})\delta}\|v_\x\|_{D_d}^{\numpoly(\bar{\tau})}\Big].\notag
\end{align}
\end{lemma}

\subsection{Proof of Lemma \ref{linear bound}}

 Take $w_1,w_2\in\Ww$. From Definition \ref{def: renormalised product}, 
 there exists $\epsilon>0$ such that for $J=[-2,-6-|w_1|-|w_2|+\epsilon)$, 
\begin{align}\label{eq:expansion_v_tau}
&(v\circ_\path\path_\bullet\Ii(w_1)\circ_\path\path_\bullet\Ii(w_2))_L(x)=\\
& \qquad \sum_{|\tau|\in J}\Big(\Upsilon_\bullet(\tau)\path_{\bullet,\bullet}(\Ii(\tau)\Ii(w_1)\Ii(w_2))\Big)_L(x).\notag
\end{align}
We know $J$ is the right interval even though we have a longer expansion of $v$ because the unplanted  trees of positive homogeneity vanish in our formalism. This corresponds to $|\tau|+|w_1|+|w_2|+6<0$, and  Remark \ref{rem:noninteger} tells us that this expansion is the same to positive level $\epsilon$, for $\epsilon>0$ small enough.

We prove estimate ~\eqref{e-linbound1} by using the reconstruction Lemma~\ref{Reconstruction}. Define $F(y,x)=\sum_{|\tau|\in J}\Upsilon_{x}(\tau)\path_{y,x}(\Ii(\tau)\Ii(w_1)\Ii(w_2))$, and we aim to bound a suitable regularisation of $F(y,x) - F(x,x)$. Lemma \ref{Reconstruction} and Assumption~\ref{assumption} imply the desired estimate ~\eqref{e-linbound1} 
as soon as the following identity is established:
\begin{align}\label{eq:proof-lin-rec}
&\Big| \int\Psi_l(x_2-y)(F(y,x_1)-F(y,x_2))dy \Big|\leq\\
&\notag \qquad \sum_{|\tau|\in J}[\path;\Ii(\tau)\Ii(w_1)\Ii(w_2)][U^\tau]_{-6-|w_1|-|w_2|-|\tau|+\epsilon}\\
&\notag  \qquad \times l^{6+|w_1|+|w_2|+|\tau|}d(x_1,x_2)^{-6-|w_1|-|w_2|-|\tau|+\epsilon}\;.
\end{align}
 By multiplicativity of the coproduct, and since $w_1,w_2\in\Ww$, we first note that for $\bar{\tau} \in \Nn$
\begin{align*}
\Delta(\Ii(\bar{\tau})\Ii(w_1)\Ii(w_2))&=\Delta\Ii(\bar{\tau})\Delta\Ii(w_1)\Delta\Ii(w_2)\\
&=\sum_{-2\leqslant|\tau|\leqslant|\bar{\tau}|}\Ii(\tau)\Ii(w_1)\Ii(w_2)\otimes C_+(\tau,\bar{\tau}).
\end{align*}
Using Chen's relation, we have:
\begin{align*}
F(y,x_1) &=\sum_{|\bar{\tau}|\in J}\Upsilon_{x_1}(\bar{\tau})\path_{y,x_1}(\Ii(\bar{\tau})\Ii(w_1)\Ii(w_2))\\
&=\sum_{|\bar{\tau}|\in J}\Upsilon_{x_1}(\bar{\tau})\sum_{|\tau|\in J}\path_{y,x_2}(\Ii(\tau)\Ii(w_1)\Ii(w_2))\path_{x_2,x_1}C_+(\tau,\bar{\tau})\\
&=\sum_{|\tau|\in J}\path_{y,x_2}(\Ii(\tau)\Ii(w_1)\Ii(w_2))\sum_{|\bar{\tau}|\in J}\Upsilon_{x_1}(\bar{\tau}) \path_{x_2,x_1}C_+(\tau,\bar{\tau}).
\end{align*}
Therefore, 
\begin{align*}
&F(y,x_1)-F(y,x_2)\\
&=\sum_{|\tau|\in J}\path_{y,x_2}(\Ii(\tau)\Ii(w_1)\Ii(w_2))\Big(\sum_{|\bar{\tau}|\in J}\Upsilon_{x_1}(\bar{\tau})\path_{x_2,x_1}C_+(\tau,\bar{\tau})-\Upsilon_{x_2}(\tau)\Big)\\
&=-\sum_{|\tau|\in J}\path_{y,x_2}(\Ii(\tau)\Ii(w_1)\Ii(w_2))U_{-6-|w_1|-|w_2|-|\tau|+\epsilon}^\tau(x_2,x_1)\\
\end{align*}
which proves \eqref{eq:proof-lin-rec} and thus \eqref{e-linbound1}.

The bound \eqref{e-linbound2} is simply the order bound on the trees, which for $x\in D_d$ and for $L<d$, can be expressed as:
\begin{align}\label{eq:bound-lin-trees}
|\Upsilon_x(\tau)&\Big(\path_{\bullet,x}(\Ii(\tau)\Ii(w_1)\Ii(w_2))\Big)_L(x)|\\
&\notag \leqslant \|v\|_{D}^{\numone(\tau)}\|v_\x\|_{D_d}^{\numpoly(\tau)}[\path;\Ii(\tau)\Ii(w_1)\Ii(w_2)]L^{6+|\tau|+|w_1|+|w_2|}\\
&\notag \overset{\text{Ass. } \ref{assumption}}{\leqslant}c\|v\|_{D}^{\numone(\tau)+\delta \numnoise(\tau+w_1+w_2)}\|v_\x\|_{D_d}^{\numpoly(\tau)}L^{6+|\tau|+|w_1|+|w_2|}.
\end{align}

\subsection{Proof of Lemma \ref{quadratic bound}}
 Take $w\in \Ww$. From Definition \ref{def: renormalised product} and  Remark \ref{rem:noninteger}, there exists $\epsilon>0$ such that for $\tilde{J}=\{(a,b)\in [-2,-1]^2, a+b<-6-|w|+\epsilon\}$,
\begin{align}
\notag
&(v\circ_\path v\circ_\path \path_\bullet\Ii(w))_L(x) \\
\label{eq:expansion_vv_tau}
&=\sum_{(|\tau_1|,|\tau_2|)\in \tilde{J}}\Big(\Upsilon(\tau_1)\Upsilon(\tau_2)\path_{\bullet,\bullet}(\Ii(\tau_1)\Ii(\tau_2)\Ii(w))\Big)_L(x).
\end{align} 
We know $J$ is the right domain even though we have a longer expansion for $v$ because unplanted trees of positive order vanish in our setting.

We prove the bound \eqref{eq:rev-v2}.
 Define $F(y,x)=\sum_{(|\tau_1|,|\tau_2|)\in\tilde{J}}\Upsilon_{x}(\tau_1)\Upsilon_{x}(\tau_2)\path_{y,x}(\Ii(\tau_1)\Ii(\tau_2)\Ii(w))$, and we aim to bound a suitable regularisation of $F(y,x)-F(x,x)$.
Lemma~\ref{Reconstruction} and Assumption~\ref{assumption} implies the desired bound as soon as the following bound is established:
\begin{align}
\notag
&\Big| \int\Psi_l(x_2-y)(F(y,x_1)-F(y,x_2))dy \Big|\leq\\
&\notag \qquad \sum_{(|\tau_1|,|\tau_2|)\in \tilde{J}}[\path;\Ii(\tau_1)\Ii(\tau_2)\Ii(w)][U^{\Ii(\tau_1)\Ii(\tau_2)\Ii(\Xi)}]_{-6-|w|-|\tau_1|-|\tau_2|+\epsilon}\\
\label{eq:proof-quad-rec}
&\qquad  l^{6+|\tau_1|+|\tau_2|+|w|}d(x_1,x_2)^{-6-|w|-|\tau_1|-|\tau_2|+\epsilon}
\end{align}
 By multiplicativity of the coproduct, and since $w\in\Ww$, we first note that for $(|\bar{\tau}_1|,|\bar{\tau}_2|)\in\tilde{J}$, 
\begin{align*}
\Delta(\Ii(\bar{\tau}_1)\Ii(\bar{\tau}_2)\Ii(w))&=\Delta\Ii(\bar{\tau}_1)\Delta\Ii(\bar{\tau}_2)\Delta\Ii(w)\\
&=\sum_{\underset{-2\leqslant|\tau_1|\leqslant|\bar{\tau}_1|}{-2\leqslant|\tau_2|\leqslant|\bar{\tau}_2|}}\Ii(\tau_1)\Ii(\tau_2)\Ii(w)\otimes C_+(\tau_1,\bar{\tau}_1) C_+(\tau_2,\bar{\tau}_2).
\end{align*}
Using Chen's relation, we have
\begin{align*}
&F(y,x_1)\\
&=\sum_{(|\bar{\tau}_1|,|\bar{\tau}_2|)\in\tilde{J}}
	\Upsilon_{x_1}(\bar{\tau}_1) \Upsilon_{x_1}(\bar{\tau}_2) \path_{y,x_1}(\Ii(\bar{\tau}_1)\Ii(\bar{\tau}_2)\Ii(w))\\
&=\sum_{(|\bar{\tau}_1|,|\bar{\tau}_2|)\in\tilde{J}}\Upsilon_{x_1}(\bar{\tau_1})
	\Upsilon_{x_1}(\bar{\tau_2})\\
& \qquad \times\sum_{\underset{-2\leqslant|\tau_1|\leqslant|\bar{\tau}_1|}{-2\leqslant|\tau_2|\leqslant|\bar{\tau}_2|}}
	\path_{y,x_2}(\Ii(\tau_1)\Ii(\tau_2)\Ii(w))\path_{x_2,x_1}( C_+(\tau_1,\bar{\tau}_1) C_+(\tau_2,\bar{\tau}_2))\\
&=\sum_{(|\tau_1|,|\tau_2|)\in\tilde{J}}
	\path_{y,x_2}(\Ii(\tau_1)\Ii(\tau_2)\Ii(w))\\
& \qquad \times\sum_{(|\bar{\tau}_1|,|\bar{\tau}_2|)\in\tilde{J}}
	\Upsilon_{x_1}(\bar{\tau}_1)\Upsilon_{x_1}(\bar{\tau}_2)\path_{x_2,x_1}( C_+(\tau_1,\bar{\tau}_1) C_+(\tau_2,\bar{\tau}_2)).
\end{align*}
In the following computation, we introduce a mock $\Upsilon(\Xi)$, which is just a factor $-1$, and $C_+(\Xi,\Xi)=1$ to make explicit that the structure of the terms appearing here is that of $U^{\tilde{\tau}}_\beta$ for some $\tilde{\tau}$ and $\beta$.
\begin{align*}
&F(y,x_1)-F(y,x_2)
=\sum_{(|\tau_1|,|\tau_2|)\in\tilde{J}}\path_{y,x_2}(\Ii(\tau_1)\Ii(\tau_2)\Ii(w))\\
&\times\Big(\sum_{(|\bar{\tau}_1|,|\bar{\tau}_2|)\in\tilde{J}}\Upsilon_{x_1}(\bar{\tau}_1)\Upsilon_{x_1}(\bar{\tau}_2)\path_{x_2,x_1}( C_+(\tau_1,\bar{\tau}_1) C_+(\tau_2,\bar{\tau}_2))-\Upsilon_{x_2}(\tau_1)\Upsilon_{x_2}(\tau_2)\Big)\\
&=\sum_{(|\tau_1|,|\tau_2|)\in\tilde{J}}\path_{y,x_2}(\Ii(\tau_1)\Ii(\tau_2)\Ii(w))\times\Big(\Upsilon_{x_2}(\tau_1)\Upsilon_{x_2}(\tau_2)\Upsilon_{x_2}(\Xi)\\
&-\sum_{(|\bar{\tau}_1|,|\bar{\tau}_2|)\in\tilde{J}}\Upsilon_{x_1}(\bar{\tau}_1)\Upsilon_{x_1}(\bar{\tau}_2)\Upsilon_{x_1}(\Xi)\path_{x_2,x_1}( C_+(\tau_1,\bar{\tau}_1) C_+(\tau_2,\bar{\tau}_2)C_+(\Xi,\Xi))\Big)\\
&=\sum_{(|\tau_1|,|\tau_2|)\in\tilde{J}}\path_{y,x_2}(\Ii(\tau_1)\Ii(\tau_2)\Ii(w))\times\Big(\Upsilon_{x_2}(\Ii(\tau_1)\Ii(\tau_2)\Ii(\Xi))\\
&-\sum_{(|\bar{\tau}_1|,|\bar{\tau}_2|)\in\tilde{J}}\Upsilon_{x_1}(\Ii(\tau_1)\Ii(\tau_2)\Ii(\Xi))\path_{x_2,x_1}( C_+(\Ii(\tau_1)\Ii(\tau_2)\Ii(\Xi),\Ii(\bar{\tau}_1)\Ii(\bar{\tau}_2)\Ii(\Xi))\Big)\\
&=\sum_{(|\tau_1|,|\tau_2|)\in\tilde{J}}\path_{y,x_2}(\Ii(\tau_1)\Ii(\tau_2)\Ii(w))U^{\Ii(\tau_1)\Ii(\tau_2)\Ii(\Xi)}_{-3+\delta-|w|+\epsilon},
\end{align*}
which proves \eqref{eq:proof-quad-rec} and thus \eqref{eq:rev-v2}.

The bound \eqref{eq:rev-v2.2} is directly the order bound on the trees, which for $x\in D_d$ and for $L<d$, can be expressed as:
\begin{align}\label{eq:bound-quad-trees}
|\Upsilon_x(\tau_1)\Upsilon_x(\tau_2)&\Big(\path_{\bullet,x}(\Ii(\tau_1)\Ii(\tau_2)\Ii(w))\Big)_L(x)|\\
&\notag \leqslant \|v\|_{D}^{\numone(\tau_1+\tau_2)}[\path;\Ii(\tau_1)\Ii(\tau_2)\Ii(w)]L^{6+|\tau_1|+|\tau_2|+|w|}\\
&\notag \overset{\text{Ass. }\ref{assumption}}{\leqslant}c\|v\|_{D}^{\numone(\tau_1+\tau_2)+\delta \numnoise(\tau_1+\tau_2+w)}L^{6+|\tau_1|+|\tau_2|+|w|}.
\end{align}
Note that here the term $\|v_\x\| $ does not appear since $\tilde{J}$ does not contain any homogeneities higher than $1$.

\subsection{Proof of Lemma \ref{lemma: apply Schauder}}
For $\gamma\in (2-2\delta,2)$ we have
\begin{align*}
\heat_y V_\gamma(y,x)=\sum_{-2<|\tau|<\gamma-2}\Upsilon_{x}(\tau)\path_{yx}(\tau).
\end{align*}
We write trees in this sum as $\tau=\Ii(\tau_1)\Ii(\tau_2)\Ii(\tau_3)$. The first remark we make is that if $|\tau_i|\geqslant -2$ for $i=1,2,3$, then $|\tau|\geqslant 0>\gamma-2$. We also remark that for $w\in\Ww$, $\Upsilon_x(w)$ is independent of $x$ and for $\tau\in\partial\Ww$, $\path_{yx}\tau$ is also independent of $x$. Therefore, accounting for symmetries with the factor $3$, we get:
\begin{align*}
(\partial_t-&\Delta)_y V_\gamma(y,x)=\\
-3&\sum_{w\in\Ww}\Upsilon_{y}(w)\sum_{|\tau_1|+|\tau_2|<\gamma-8-|w|}\Upsilon_x(\tau_1)\Upsilon_x(\tau_2)\path_{y,x}(\Ii(\tau_1)\Ii(\tau_2)\Ii(w))\\
-3&\sum_{w_1,w_2\in\Ww}\Upsilon_{y}(w_1)\Upsilon_{y}(w_2)\sum_{|\tau|<\gamma-8-|w_1|-|w_2|}\Upsilon_x(\tau)\path_{y,x}(\Ii(\tau)\Ii(w_1)\Ii(w_2))\\
+&\sum_{\tau\in\partial\Ww}\Upsilon_y(\tau)\path_{y,y}\tau.
\end{align*}
Using the remainder equation, we have:
\begin{align*}
\heat (v-&V_\gamma(\cdot,x))(y)=-v^3(y)\\
-3\sum_{w\in\Ww}&\Upsilon_y(w)\Big((v\circ_\path v\circ_\path \path_y\Ii(w))(y)\\
&-\sum_{|\tau_1|+|\tau_2|<\gamma-8-|w|}\Upsilon_x(\tau_1)\Upsilon_x(\tau_2)\path_{y,x}(\Ii(\tau_1)\Ii(\tau_2)\Ii(w))\Big)\\
-3\sum_{w_1,w_2\in\Ww}&\Upsilon_y(w_1)\Upsilon_y(w_2)\Big((v\circ_\path\path_y\Ii(w_1)\circ_\path\path_y\Ii(w_2))(y)\\
&-\sum_{|\tau|<\gamma-8-|w_1|-|w_2|}\Upsilon_x(\tau)\path_{y,x}(\Ii(\tau)\Ii(w_1)\Ii(w_2))\Big).
\end{align*}
We need to bound this after integration against $\Psi_{L_1}(z-y)dy$ for $z\in B(x,L_2)$, for $x\in D_{2d}$, for $L_1<\frac{d}2$ and $L_2<\frac{d}4$ to apply the Schauder Lemma~\ref{lemschauder}. We first have:
\[
|(v^3)_{L_1}(z)|\leqslant\|v\|^3_{D}.
\]
If $F(y,x)= \sum_{|\tau_1|+|\tau_2|<\gamma-8-|w|}\Upsilon_x(\tau_1)\Upsilon_x(\tau_2)\path_{y,x}(\Ii(\tau_1)\Ii(\tau_2)\Ii(w))$, then the Lemma~\ref{quadratic bound} gives a bound on $(v\circ_\path v\circ_\path \path_\bullet\Ii(w)-F(\bullet,z))_{L_1}(z)$, and from equation~\eqref{eq:proof-quad-rec} we have a bound on $(F(\bullet,x)-F(\bullet,z))_{L_1}(z)$. 
Together with Assumption~\ref{assumption}, they give
\begin{align}\label{prepshquad}
\Big|\Big((v\circ_\path v\circ_\path \path_\bullet\Ii(w))&\\
-\sum_{|\tau_1|+|\tau_2|<\gamma-8-|w|}\Upsilon_x(\tau_1)&\Upsilon_x(\tau_2)\path_{\cdot,x}(\Ii(\tau_1)\Ii(\tau_2)\Ii(w))\Big)_{L_1}(z)\Big|\notag\\
\les
c\sum_{|\tau_1|+|\tau_2|<\gamma-8-|w|}&[U^{\Ii(\tau_1)\Ii(\tau_2)\Ii(\Xi)}]_{-6-|w|-|\tau_1|-|\tau_2|+\epsilon,D_d,d}\notag\\
\times\|v\|_D^{\delta \numnoise(\tau_1+\tau_2+w)}&(L_1^\epsilon+ L_1^{6+|\tau_1|+|\tau_2|+|w|}d(x,z)^{-6-|w|-|\tau_1|-|\tau_2|+\epsilon})\notag.
\end{align}
Similarly with $F(y,x)=\sum_{|\tau|<\gamma-8-|w_1|-|w_2|}\Upsilon_x(\tau)\path_{y,x}(\Ii(\tau)\Ii(w_1)\Ii(w_2))$,  Lemma~\ref{linear bound} gives a bound on $(v\circ_\path\path_\bullet\Ii(w_1)\circ_\path\path_\bullet\Ii(w_2))-F(\bullet,z))_{L_1}(z))$ and from Equation~\eqref{eq:proof-lin-rec}, we have a bound on $(F(\bullet,x)-F(\bullet,z))_{L_1}(z)$.
Together with Assumption~\ref{assumption}, they give
\begin{align}\label{prepshlin}
\Big|\Big(v\circ_\path\path_\bullet\Ii(w_1)\circ_\path\path_\bullet\Ii(w_2)&\\
-\sum_{|\tau|<\gamma-8-|w_1|-|w_2|}&\Upsilon_x(\tau)\path_{\bullet,x}(\Ii(\tau)\Ii(w_1)\Ii(w_2))\Big)_L(z)\Big|\notag\\
\les c\sum_{|\tau|<\gamma-8-|w_1|-|w_2|}&[U^\tau]_{-6-|w_1|-|w_2|-|\tau|+\epsilon,D_d,d}\notag\\
\times\|v\|_{D}^{\delta \numnoise(\tau+w_1+w_2)}&(L^\epsilon+L^{6+|w_1|+|w_2|+|\tau|}d(x,z)^{-6-|w_1|-|w_2|-|\tau|+\epsilon}).\notag
\end{align}
We also need the three-point continuity. It is a consequence of Lemma~\ref{lemma 3pt}, and can be quantified, for $x\in D_d$, for $y\in B(x,\frac{d}4)$, for $z\in B(y,\frac{d}4)$, as:
 \begin{align}\label{prepsh3}
|V_\gamma(z,x) &  - V_\gamma(z,y)+ V_\gamma(y,y)-V_\gamma(y,x)+V^i_\gamma(y,x)(z_i-y_i)|\\
\leq& \sum_{\underset{\tau\neq\x}{-2<|\tau|<\gamma-2}}
[ U^{\tau}]_{\gamma-2-|\tau|,D_\frac{d}2,\frac{d}2} d(y,x)^{\gamma-2-|\tau|} \,    [\path; \Ii(\tau)]d(z,y)^{|\tau|+2}\notag\\
\overset{\text{Ass. }\ref{assumption}}{\leq}& c\sum_{\underset{\tau\neq\x}{-2<|\tau|<\gamma-2}}
\|v\|_{D}^{\delta \numnoise(\tau)}[ U^{\tau}]_{\gamma-2-|\tau|,D_\frac{d}2,\frac{d}2} d(y,x)^{\gamma-2-|\tau|} \,  d(z,y)^{|\tau|+2}.\notag
\end{align}
We now notice that all the term appearing in \eqref{prepshquad},\eqref{prepshlin} and \eqref{prepsh3} have a common structure. By replacing the homogeneities by their expressions in terms of $\delta \numnoise$ for the trees in $\Ww$, and relabelling $\tau=\Ii(\tau_1)\Ii(\tau_2)\Ii(\Xi)$ in \eqref{prepshquad}, we get after application of Lemma~\ref{lemschauder}:
\begin{align}\label{eq:shauderapplied}
\sup_{d\leqslant d_0}d^\gamma [V]_{\gamma,D_d}\les&\sup_{d\leq d_0}d^2\|v\|_{D}^3\\
&+c\sup_{a,b,\tau}\sup_{d\leq d_0}\Big(d^{b}\|v\|_{D_d}^{a+\delta \numnoise(\tau)}[U^{\tau}]_{b-2-a-|\tau|,D_d,d}\Big)\notag,
\end{align}
where $ [V]_{\gamma,D_d}$ is the $\gamma$-H\"older norm of $V_\gamma$, restricted to the domain $D_d$, and 
where the supremum is taken over a finite subset of
\[
\{(a,b,\tau)\in \R_+^2\times \Tr,a\geq 0\; , b\geq \gamma\; , b-2-a-|\tau|< \gamma\}.
\]

We apply the Lemma \ref{lem:hol-relation} to the second part.

\begin{align*}
\sup_{d\leqslant d_0}d^\gamma [V]_{\gamma,D_d}&\les\sup_{d\leq d_0}d^2\|v\|_{D}^3\\
+c\sup_{a,b,\tau}\sup_{d\leq d_0}&d^{a+\delta \numnoise(\tau)}\|v\|_{D_d}^{a+\delta \numnoise(\tau)}\Big([V]_{\gamma,D_d}d^{\gamma}+{\mathbbm{1}}_{\{b-a-\delta \numnoise(\tau)<1\}}d\|v_\X\|_{D_d}\\
&+c\sum_{b-a-2-\delta \numnoise(\tau)\leqslant |\bar{\tau}|<\gamma-2}d^{|\bar{\tau}|+2}\|v\|_D^{\numone(\bar{\tau})+\delta \numnoise(\bar{\tau})}\|v_\X\|_{D_d}^{\numpoly(\bar{\tau})}\Big).
\end{align*}
We see now that if we take 
\begin{equation}\label{fixd0}
d_0=\|v\|_D^{-1}
\end{equation}
, then there exist a value of  $c_0<1$ such that for any $0<c<c_0$, the occurrences of $[V]_{\gamma,D_d}$ can be absorbed into the left-hand side, and the other terms also simplify: in the last sum, if $\numpoly(\bar{\tau})=1$, we bound 
\[d^{|\bar{\tau}|+2}\|v\|_D^{\numone(\bar{\tau})+\delta \numnoise(\bar{\tau})}\|v_\x\|_{D_d}^{\numpoly(\bar{\tau})}\leqslant d^{-1+2\numpoly(\bar{\tau})}\|v_\x\|_{D_d}^{\numpoly(\bar{\tau})}=d\|v_\X\|_{D_d},\]
 and if $\numpoly(\bar{\tau})=0$, 
 \[d^{|\bar{\bar{\tau}}|+2}\|v\|_D^{\numone(\bar{\tau})+\delta \numnoise(\bar{\tau})}=d^{-1+\numone(\bar{\tau})+\delta \numnoise(\bar{\tau})}\|v\|_D^{\numone(\bar{\tau})+\delta \numnoise(\bar{\tau})}\leq\|v\|_D.\]
  In conclusion,
\begin{align}\label{schauderapplied3}
\sup_{d\leqslant d_0}d^\gamma [V]_{\gamma,D_d}\les\|v\|_{D}
+c\sup_{d\leq d_0}d\|v_\X\|_{D_d}.
\end{align}

We now prove the bound on $\|v_\x\|$. For that we take $\epsilon$ 
small enough such that there is no tree of regularity between $1$ and $1+\epsilon$. Then we can apply Corollary~\ref{corschauder} with $\kappa=1+\epsilon$ but with $U(x,y)=\sum_{-2\leqslant|\tau|<-1}\Upsilon_x(\tau)\Y_{yx}\Ii(\tau)$.
We get
\[
\|v_\x\|_{D_d}\les [V]_{1+\epsilon,D_d,d}d^\epsilon+\|U\|_{D_d,d}d^{-1}.
\] 
We have by Assumption~\ref{assumption}
\[
\|U\|_{D_d,d}d^{-1}\les\sum_{n}d^{n-2}\|v\|_{D}^{n},
\]
and from \eqref{eq:corshauder-applied}
\[
[V]_{1+\epsilon,D_d,d}d^\epsilon\les [V]_{\gamma,D_d,d} d^{\gamma-1}+c\sum_{n,m}d^{n+2m-2}\|v\|_{D_d}^{n}\|v_\x\|_{D_d}^m,
\]
where the sum ranges over a finite set of indices $n\geq 0$ and $m\in\{0,1\}$. We have, assuming $d_0\leqslant \|v\|_D^{-1}$,

\[
\sup_{d\leq d_0}d\|v_\x\|_{D_d}\les \sup_{d\leq d_0}\Big( [V]_{\gamma,D_d,d} d^{\gamma}+c(\|v\|_{D}+d\|v_\x\|_{D_d})\Big).
\]
 If  we take 
  $c$ small enough, depending on the constant implicit in $\les$,  for some constant $C>0$ we have,
\begin{equation}\label{eq:bound_vX}
\sup_{d\leq d_0}d\|v_\x\|_{D_d}\leqslant C\sup_{d\leq d_0}\Big( [V]_{\gamma,D_d,d} d^{\gamma}+\|v\|_{D}\Big).
\end{equation}
Together with \eqref{schauderapplied3}, this gives, for a constant $c$ small enough,
\begin{equation}\label{schauder final}
\sup_{d\leqslant d_0}d^\gamma [V]_{\gamma,D_d,d}\lesssim  \|v\|_D.
\end{equation}
\subsection{Proof of Lemma~\ref{lemtheorem}}
We apply the Lemma~\ref{lem_max} to the convolved equation \eqref{eq:convolved}, on the domain $D$.
\begin{align}\label{eq:apply_max1}
\|v\|_{D_{d+d'}}\lesssim \max \Big\{&{d'}^{-1},\quad \|(v^3)_L-v_L^3\|_{D_{d}}^{\frac13},\quad \|(\path\tau)_L\|_{D_d}^{\frac13} ,\quad \tau\in\partial\Ww,\notag\\
&\|(v\circ_\path\path_\cdot\Ii(w_1)\circ_\path\path_\cdot\Ii(w_2))_L\|_{D_d}^{\frac13},\; w_i\in\Ww,\\
&\|(v\circ_\path v\circ_\path \path_\bullet\Ii(w))_L\|_{D_d}^{\frac13}, \; w\in\Ww,\quad\|v-v_L\|_{D_{d+d'}}\Big\}\notag
\end{align} 
We have for $d>L$
\begin{align}
\notag
\|(v^3)_L-v_L^3\|_{D_{d}}&\lesssim \|v\|^2_{D_{d-L}}[v]_{\alpha,{D_{d-L}},L}L^\alpha  \qquad \text{ and } \\
\label{commu}
\|v-v_L\|_{D_{d+d'}} &\leqslant [v]_{\alpha,{D_{d-L}},L}L^\alpha
\end{align}
and for $\alpha$ small enough, we have by Lemma~\ref{lemma: apply Schauder}, $[v]_{\alpha,{D_{d-L}},L}\lesssim {(d-L)}^{-\alpha}\|v\|_D$.

From Lemma~\ref{linear bound} we get for $w_1,w_2\in\Ww$,
\begin{align*}
& \|(v\circ_\path\path_\bullet\Ii(w_1)\circ_\path\path_\bullet\Ii(w_2))_L\|_{D_d}\\
&\qquad \lesssim  c\sum_{|\tau|\in J}\|v\|_{D}^{\delta \numnoise(\tau+w_1+w_2)}\Big([U^\tau]_{-6-|w_1|-|w_2|-|\tau|+\epsilon,D_{d-L},d}L^\epsilon\\
&\qquad+\|v\|_{D}^{\numone(\tau)}\|v_\x\|_{D_{d-L}}^{\numpoly(\tau)}L^{6+|\tau|+|w_1|+|w_2|}\Big)\notag.
\end{align*}
From Lemma~\ref{quadratic bound}, we get for $w\in\Ww$,
\begin{align*}
&\|(v\circ_\path v\circ_\path \path_\bullet\Ii(w))_L\|_{D_d}\\
& \les c\sum_{(|\tau_1|,|\tau_2|)\in \tilde{J}}\|v\|^{\delta \numnoise(\tau_1+\tau_2+w)}\Big([U^{\Ii(\tau_1)\Ii(\tau_2)\Ii(\Xi)}]_{-6-|w|-|\tau_1|-|\tau_2|+\epsilon,D_{d-L},d}L^\epsilon\\
 & \qquad  +\|v\|_{D}^{\numone(\tau_1+\tau_2)}L^{6+|\tau_1|+|\tau_2|+|w|}\Big).
\end{align*}
Using Lemma~\ref{lemma: apply Schauder} and setting $d=\|v\|_D^{-1}$ and $L=\frac{d}{k}$ for some $k\geqslant2$ gives  for $w_1,w_2\in \partial \Ww$,
\begin{align}\label{eq:final bound linear}
\|(v\circ_\path\path_\bullet\Ii(w_1)\circ_\path\path_\bullet\Ii(w_2))_L\|_{D_d}\lesssim c\|v\|_D^3K(w_1,w_2,k) \;,
\end{align}
where  
\[
K(w_1,w_2,k)=\sum_{|\tau|\in J}\Big(k^{-\epsilon}+k^{3-\delta \numnoise(\tau+w_1+w_2)-\numone(\tau)-2\numpoly(\tau)}\Big),
\]
and for $w\in\Ww$,
\begin{align}\label{eq:final bound quad}
\|(v\circ_\path v\circ_\path \path_\bullet\Ii(w))_L\|_{D_d}\lesssim c\|v\|_D^3K'(w,k)
\end{align}
where
\[
K'(w,k)=\sum_{(|\tau_1|,|\tau_2|)\in\tilde{J}}\Big(k^{-\epsilon}+k^{3-\delta \numnoise(\tau_1+\tau_2+w)-\numone(\tau_1+\tau_2)}
\Big).
\]
Finally, for $\tau\in\partial\Ww$, we get from  Assumption~\ref{assumption}
\begin{equation}\label{eq:final bound constant}
\|(\path\tau)_L\|_{D_d}\lesssim c L^{-3+\delta \numnoise(\tau)}\|v\|_D^{\delta \numnoise(\tau)}=c\|v\|_D^3k^{3-\delta \numnoise(\tau)}.
\end{equation}

With \eqref{commu}, \eqref{eq:final bound linear}, \eqref{eq:final bound quad} and \eqref{eq:final bound constant}, the bound \eqref{eq:apply_max1} becomes
\begin{align}\label{eq:apply_max2}
\|v\|_{D_{\|v\|_D^{-1}+R}}\lesssim \max \Big\{&R^{-1},\quad k^{-\frac\alpha3}\|v\|_D,\quad c^\frac13\|v\|_Dk^{1-\frac{\delta \numnoise(\tau)}3} ,\; \tau\in\partial\Ww,\notag\\
&c^\frac13\|v\|_DK(w_1,w_2,k)^\frac13,\; w_1,w_2\in\Ww, \\
& c^\frac13\|v\|_DK'(w,k)^\frac13,\; w\in\Ww,\quad k^{-\alpha}\|v\|_D,\Big\}.\notag
\end{align} 
We see that we can now choose $k>2$ large enough and then $c<c_0$, 
as well as $R=(\lambda-1) \|v\|_D^{-1}$ for $\lambda$ large enough such that \eqref{eq:apply_max2} becomes:
\begin{equation}\label{eq:apply_max3}
\|v\|_{D_{\lambda \|v\|_D^{-1}}}\leq \frac{\|v\|_D	}2.
\end{equation}
\appendix
%
%
\section{Reconstruction lemma}
\label{ss:RL}
In this section, we present the result that is essential for the proof of Lemmas~\ref{linear bound} and \ref{quadratic bound}, allowing to define a function given its local description. It is inspired from \cite[Proposition 3.28]{hairer2014theory}. We show here a localised result that was introduced in \cite{2018arXiv181105764M}. We also reproduce the short proof here. It depends strongly on a specific choice of kernel to measure the regularity, which we construct hereafter.

 We fix a non-negative smooth function $\Phi$ with support in $B(0,1)$, symmetric in space, with $\Phi(x)\in [0,1]$ for all $x\in \R\times\R^d$ and with integral $1$.
Setting $\Phi_L(t,x)=L^{-5}\Phi(\frac{t}{L^2},\frac{x}L)$,  
we now define $\Psi_{L,n}= \Phi_{L2^{-1}}\ast\Phi_{L2^{-2}}\ast...\ast\Phi_{L2^{-n}}$ and $\Psi_L=\lim_{n\rightarrow\infty}\Psi_{L,n}$ so that $\Psi_L=\Phi_\frac{L}2\ast\Psi_\frac{L}2$. $\Psi_L$ and $\Psi_{L,n}$ are non-negative and smooth, symmetric in space and with support $B(0,1)$ and $B(0,1-2^n)$. We define the operator $(\cdot)_L$ by convolution with $\Psi_L$, and $(\cdot)_{L,n}$ by convolution with $\Psi_{L,n}$ for $n\geqslant 1$. $(\cdot)_{L,0}$ is the identity. Since $\Psi_{L,n+m}=\Psi_{L,n}\ast\Psi_{L2^{-n},m}$, we have
\begin{equation}\label{semigroup 2}
(\cdot)_{L,n+m}=((\cdot)_{L2^{-n},m})_{L,n}.
\end{equation}
Taking $m$ to infinity in this, or equivalently noticing that $\Psi_L=\Psi_{L,n}\ast\Psi_{L2^{-n}}$, we have the desired relation between dyadic scales
\begin{equation}\label{semigroup 1}
(\cdot)_L=((\cdot)_{L2^{-n}})_{L,n}.
\end{equation}
\begin{lemma}[Reconstruction]\label{Reconstruction}
Let $\gamma>0$ and $A$ be a finite subset of $(-\infty,\gamma]$. Let $L\in (0,1)$ and $x\in \R\times \R^d$. 
For a function $F \colon B(x,L)^2 \to \R$ assume that for all $\beta\in A$ there exist constants $C_\beta>0$ and $\gamma_\beta\geqslant\gamma$ such that for all $l\in(0,L)$, for all $x_1,x_2\in B(x,L-l)$
\begin{equation}\label{eq:reconstruction_hypothesis}
\Big| \int\Psi_l(x_2-y)(F(y,x_1)-F(y,x_2))dy \Big|\leq \sum_{\beta\in A}C_\beta d(x_1,x_2)^{\gamma_\beta-\beta}l^\beta.
\end{equation}
Then $f:y\mapsto F(y,y)$ satisfies
\begin{equation}\label{eq:reconstruction_conclusion}
\Big| \int\Psi_L(x-y)(F(y,x)-f(y))dy \Big|\les\sum_{\beta\in A}C_\beta L^{\gamma_\beta},
\end{equation}
where ``$\les$'' represents a bound up to a multiplicative constant depending only on $\gamma$ and $A$.
\end{lemma}
In the proof we will use the following notations for $f$ a function of one variable and  $F$ a function of two variables:
\begin{align}\label{notations}
[F,(\cdot)_L](x)&=\int\Psi_L(x-y)F(x,y)dy
\end{align}
\begin{proof}
This is the only place where our particular choice of convolution kernel is crucial. It allows to use the following factorisation:
\begin{align*}
\Big|[F,(\cdot&)_{L2^{-n}}](x_1)-\Big([F,(\cdot)_{L2^{-n-1}}]\Big)_{L2^{-n},1}(x_1)\Big|\\
=&\Big|\int\int\Psi_{L2^{-n-1}}(x_2-y)\Phi_{L2^{-n-1}}(x_1-x_2)(F(x_1,y)-F(x_2,y))dydx_2\Big|\\
\leq&\sum_{\beta\in A}C_\beta\int\Phi_{L2^{-n-1}}(x_1-x_2)d(x_1,x_2)^{\gamma_\beta-\beta}(L2^{-n-1})^\beta dx_2\\
\leq&\sum_{\beta\in A}C_\beta(L2^{-n-1})^{\gamma_\beta}.
\end{align*}
This proves the convergence of $[F,(\cdot)_{L2^{-n}}]$ to $f:y\mapsto F(y,y)$ and justifies the bound following telescopic sum, obtained once more thanks to the semi-group property of our kernel:
\begin{align*}
\Big|[F,(\cdot)_L]-\big([F,(\cdot)_{L2^{-N}}]\big)_{L,N-1}\Big|
=&\Big|\sum_{n=0}^N\Big([F,(\cdot)_{L2^{-n}}]-\big([F,(\cdot)_{L2^{-n-1}}]\big)_{L2^{-n},1}\Big)_{L,n}\Big|\\
\leq&\sum_{n=0}^N\sum_{\beta\in A}C_\beta(L2^{-n-1})^{\gamma_\beta}\les \sum_{\beta\in A}C_\beta L^{\gamma_\beta},
\end{align*}
where the constant in ''$\les$'' depends only on $\gamma$ (in particular not on $N$), thus proving the lemma.
\end{proof}
In the definition of local product for a planted tree, we solve the heat equation with a cut-off function. The following lemma justifies that for a smooth cut-off, this does not change the order bound.  
\begin{lemma}\label{lemma: multiplication by smooth function}
Let $g$ be a $C^2$ function and let $h$ be a distribution of regularity $\gamma>-2$. Then $gh\in C^\gamma$ with $[gh]_\gamma\les [h]_\gamma(1+\|g''\|)$
\end{lemma}
\begin{proof}
We do the proof in one dimension for simplicity. It is identical in higher dimension. Once again, we do not concern ourselves with existence so we may assume that $g$ and $h$ are actually smooth functions, but we only allow a control with their given regularity. 

Take the Taylor expansion of $g$: $g(y)=g(x)+g'(y)(y-x)+\mathrm{Err}(x,y)$ where 
 $|\mathrm{Err}(x,y)|\leqslant \|g''\||x-y|^2$. Then define $F(y,x)=(g(x)+g'(y)(y-x))h(y)$. We have by triangle inequalities
\begin{align*}
 |F(y,x_1)-F(y,x_2)| &= |h(y)(g(x_1)-g(x_2)+g'(y)(x_2-x_1))|\\
& \leqslant |h(y)|\|g''\|(|x_1-x_2|+|x_2-y|)^2
\end{align*}
Therefore
\[\Big| \int\Psi_l(x_2-y)(F(y,x_1)-F(y,x_2))dy \Big|\leq 4\|g''\|[h]_\gamma (|x_1-x_2|^2l^\gamma+l^{\gamma+2},\]
and the reconstruction lemma applies, giving
\[
\Big| \int\Psi_L(x-y)(g(x)-g(y)+g'(y)(y-x))h(y)dy \Big|\les\|g''\|[h]_\gamma L^{\gamma+2},
\]
which in particular guarantees that $gh\in C^\gamma$ with $[gh]_\gamma\les [h]_\gamma(1+\|g''\|)$.
\end{proof}

\section{Schauder lemmas}
\label{ss:LSL}
We start by introducing a few norms.
For $\alpha \in (0,1)$, we define the \hol semi-norm $[.]_\alpha$ 
\begin{equation}\label{e:def-hol}
[f]_\alpha:=\sup_ {z\neq \bar{z}\in \R\times\R^d}\frac{|f(z)-f(\bar{z})|}{d(z,\bar{z})^\alpha}.
\end{equation}
For $\alpha \in (1,2)$, we define the \hol semi-norm $[.]_\alpha$  
\begin{equation}\label{e:def-hol2}
[f]_\alpha:=\sup_{\substack{z \neq \bar{z}  \in\R\times\R^d\\  z= (t,x); \; \bar{z} = (\bar{t}, \bar{x})}}\frac{|f(z)-f(\bar{z})-\nabla f(z).(x-\bar{x})|}{d(z,\bar{z})^\alpha},
\end{equation}
where $\nabla$ refers to the spatial gradient.

We will often deal with functions $F(z,\bar{z})$ of two variables 
generalising the increments of $f(z) - f(\bar{z})$ in \eqref{e:def-hol2} above. In this case we define for $\alpha \in (1,2)$

\begin{equation}\label{e:def-hol2var}
[F]_\alpha:=\sup_{\substack{z\in \R\times\R^d\\ z =(t,x)}}\inf_{\nu(z)\in \R^d}\sup_{\substack{\bar{z}\in \R\times\R^d\setminus\{z\} \\ \bar{z} = (\bar{t}, \bar{x}) }}\frac{|F(\bar{z},z)-\nu(z).(x-\bar{x})|}{d(z,\bar{z})^\alpha}.
\end{equation}
The infimum over functions $\nu$ is attained when  $\nu(z)$ is the spatial gradient in the second coordinate of $F$ at point $(z,z)$.

We use the same conventions as in Section~\ref{ss:MP} for localised versions of such norms.

The first lemmas we introduce are used to get the order bounds on planted trees.
We have two lemmas, for regularity negative and for regularity between $0$ and $2$. In all this section, "$\les$" denotes a bound that holds up to a multiplicative constant that only depends on $\kappa$ and $A$ when relevant. 
\begin{lemma}\label{lemschauder-Neg}
Let $\kappa<0$ and 
$U$ be a distribution such that $\heat U$ is compactly supported and 
\begin{equation}\label{lemschauder1-Neg}
\|\heat U_{L}\|\leqslant ML^{\kappa-2} .
\end{equation}

Then
\begin{equation}\label{lemschauderC-Neg}
 [U]_{\kappa}\les M.
\end{equation}

\end{lemma}
The second lemma is slightly non-standard due to the presence of a second argument

\begin{lemma}\label{lemschauder-Pos}
Let $0<\kappa<2$ and $A\subset(-\infty,\kappa]$ be finite. 
Let $U$ be a function of two variables such that $U(x,x)=0$ for all $x$ and $\heat U(\cdot,x)$ is compactly supported for all $x$. 
Assume that there exists a constant $M^{(1)}$ such that for all base-points $x$ and length scales $L_2\leqslant L_1\leq 1$, it holds that
\begin{equation}\label{lemschauder1-Pos}
L_2^2\|\heat U_{L_2}(\cdot,x)\|_{B(x,L_1)}\leqslant M^{(1)}\sum_{\beta\in A}L_2^\beta {L_1}^{\kappa-\beta}.
\end{equation}
Assume furthermore that there exists a constant $M^{(2)}$ such that,
for any $x,y \in \R \times \R^{d}$, there exists $\lambda(y,x) = (\lambda^{(i)}(y,x))_{i=1}^{d} \in \R^d$ such that, for any $z \in \R \times \R^{d}$, the following "three-point continuity" holds:
\begin{equation}\label{lemschauder2-Pos}
|U(z,x)-U(y,x)-U(z,y)-(z_{i}-y_{i})\lambda^{(i)}(y,x)|\leqslant M^{(2)}\sum_{\beta\in A}d(y,x)^\beta d(z, y)^{\kappa-\beta}.
\end{equation}
Note that if $\kappa\leq 1$ then $\lambda$ doesn't matter so we can take $\lambda=0$. 
Then
\begin{equation}\label{lemschauderC-Pos}
 [U]_{\kappa}\les M^{(1)}+M^{(2)}.
\end{equation}
\end{lemma}
We now introduce the localised lemma that we use to bound the solutions. The difference between this lemma and the previous one is that instead of a compact support of the right-hand side, we introduce a blow-up at the boundary in the way we measure objects.  
\begin{lemma}\label{lemschauder}
Let $1<\kappa<2$ and $A\subset(-\infty,\kappa]$ be finite. 
Let $U$ be a bounded function of two variables defined on a domain $D\times D$ such that $U(x,x)=0$ for all $x$. 
Let $d_0>0$ and assume that for any $0<d\leq d_0$ and $L_1\leq \frac{d}4$ there exists a constant $M^{(1)}_{D_d,L_1}$ such that for all base-points $x\in D_d$ and length scales $L_2\leqslant L_1$, it holds that
\begin{equation}\label{lemschauder1}
L_2^2\|\heat U_{L_2}(\cdot,x)\|_{B(x,L_1)}\leqslant M^{(1)}_{D_d,{L_1}}\sum_{\beta\in A}L_2^\beta {L_1}^{\kappa-\beta}.
\end{equation}
Assume furthermore, that for $L_1,L_2\leq \frac{d}4$ there exists a constant $M^{(2)}_{D_d,L_1,L_2}$ such that, for any $x\in D_d$ and $y\in B(x,L_1)$, there exists $\lambda(y,x) = (\lambda^{(i)}(y,x))_{i=1}^{d} \in \R^d$ such that, for any $z\in B(y,L_2)$, the following "three-point continuity" holds:
\begin{equation}\label{lemschauder2}
\begin{split}
|U(z,x)-U(y,x)-U(z,y)&
-(z_{i} - y_{i}) \lambda^{(i)}(y,x)|\\
&\leqslant 
M^{(2)}_{D_d,L_1,L_2}\sum_{\beta\in A}d(y,x)^\beta d(z, y)^{\kappa-\beta}.
\end{split}
\end{equation}
Additionally define
\[
M^{(1)}:=\sup_{d\leq d_0}d^\kappa M^{(1)}_{D_d,\frac{d}2},\quad\text{and}\quad M^{(2)}:=\sup_{d\leq d_0}d^\kappa M^{(2)}_{D_d,\frac{d}2,\frac{d}4}.
\]
Then
\begin{equation}\label{lemschauderC}
\sup_{d\leq d_0}d^\kappa [U]_{\kappa,D_d}\les M^{(1)}+M^{(2)} +\sup_{d\leq d_0}\|U\|_{D_d,d}.
\end{equation}
\end{lemma}
The following lemma gives bounds on the derivative. It can be used both for the derivatives of the trees in Section~\ref{sec:coproduct} and for the derivative of the solution in Section~\ref{s:MR}.
\begin{lemma}\label{corschauderFS}
Let $\kappa > 1$ and $U\in C^{\kappa}(\R \times \R^{d})$ then, for the optimal function $\nu$ in \eqref{e:def-hol2var}, for any $r\in (0,\infty)$, 
\begin{equation}\label{corschauder1FS}
\|\nu\|
\les
r^{\kappa-1}[U]_{\kappa}
+
r^{-1}
\|U\|\;.
\end{equation}
Suppose furthermore that there exists a constant $M$ and, for all $x,y \in \R \times \R^{d}$, a vector $\lambda(y,x) = (\lambda^{(i)}(y,x))_{i=1}^{d} \in \R^d$ such that
for any $z \in \R \times \R^{d}$ one has the three-point continuity bound
\begin{equation}\label{corschauder 3ptFS}
|U(z,x)
-U(y,x)
-U(z,y)
-(z_{i} - y_{i})\lambda^{(i)}(y,x)|
\leqslant 
M
\sum_{\beta\in A}
d(y,x)^\beta d(z, y)^{\kappa-\beta}\;.
\end{equation} 
Then, if we write $f(z,w) = (f^{(i)}(z,w))_{i=1}^{d}$ where $f^{(i)}(z,w) = \nu^{(i)}(z) - \nu^{(i)}(w) + \lambda^{(i)}(z,w)$, one has 
\begin{equation}\label{corschauder2FS}
[f]_{\kappa-1}\lesssim  [U]_{\kappa}+M\;.
\end{equation}
\end{lemma}
A localised version of this lemma is as follows:
\begin{lemma}\label{corschauder}
 Assume that $D$ satisfies a spatial interior cone condition with parameters $r_0>0$ and $\beta\in(0,1)$, i.e. for all $r\in[0,r_0]$, for all $x\in D$, for any vector $\theta = (\theta^{(i)})_{i=1}^{d}\in \R^{d}$, there exists $y\in D$ such that $d(x,y)=r$ and 
\[
|\theta^{(i)}y_{i}|\geq \beta d(x,y) |\theta| .
\]
Let $\kappa > 1$ and $U\in C^{\kappa}$ then, for the optimal function $\nu$ in \eqref{e:def-hol2var} and for all $r\in[0,r_0]$, we have the bound
\begin{equation}\label{corschauder1}
\beta\|\nu\|_{D}
\leq 
r^{\kappa-1}[U]_{\kappa,D}
+
r^{-1}
\|U\|_{D,r}\;.
\end{equation}
Suppose furthermore that there exists a constant $M$ and, for all $x,y \in D$, a vector $\lambda(y,x) = (\lambda^{(i)}(y,x))_{i=1}^{d} \in \R^d$ such that
for any $z \in D$ one has the the three-point continuity bound
\begin{equation}\label{corschauder 3pt}
|U(z,x)
-U(y,x)
-U(z,y)
-(z_{i} - y_{i})\lambda^{(i)}(y,x)|
\leqslant 
M
\sum_{\beta\in A}
d(y,x)^\beta d(z, y)^{\kappa-\beta}\;.
\end{equation} 
Then, if we write $f(z,w) = (f^{(i)}(z,w))_{i=1}^{d}$ where $f^{(i)}(z,w) = \nu^{(i)}(z) - \nu^{(i)}(w) + \lambda^{(i)}(z,w)$, one has, for every $r \in [0,r_{0}]$, 
\begin{equation}\label{corschauder2}
[f]_{\kappa-1,D}\lesssim  [U]_{\kappa,D}+M+r^{-\kappa}\|U\|_{D,r}\;.
\end{equation}

\end{lemma}
\section{Symbolic index}
In this appendix, we collect the most used symbols of the article, together
with their meaning and the page where they were first introduced.

\begin{center}
\renewcommand{\arraystretch}{1.1}
\begin{longtable}{lll}
\toprule
Symbol & Meaning & Page\\
\midrule
\endfirsthead
\toprule
Symbol & Meaning & Page\\
\midrule
\endhead
\bottomrule
\endfoot
\bottomrule
\endlastfoot
 $d(\cdot,\cdot)$ & Parabolic distance between space-time points $z,\bar{z} \in \R \times \R^{d}$ & \pageref{eq:parab_metric}\\
 $\Xi$ & The abstract noise & \pageref{s:PE}\\
 $\Pp$ & The set $\{\one,\X_{1},\dots,\X_{d}\}$ & \pageref{sets_of_trees_introduced}\\
 $\Ww$ & Unplanted trees with $|\tau| < -2$ & \pageref{sets_of_trees_introduced}\\
 $\mWw$ & Set of product trees in $\Ww$, namely $\Ww \setminus \{\Xi\}$ & \pageref{sets_of_trees_introduced}\\
 $\Nn$ & Unplanted trees with $|\tau| \in [-2,0]$, includes $\Pp$ & \pageref{sets_of_trees_introduced}\\
 $\mNn$ & Set of product trees in $\Nn$, namely $\Nn \setminus \Pp$ & \pageref{sets_of_trees_introduced}\\
 $\Tr$ & Unplanted trees on right hand side of $\phi$ equation, $\Tr = \Ww \cup \mNn$& \pageref{other_sets_of_trees} \\
 $\Tl$ & Planted trees in expansion of $\phi$, $\Tl = \Ii(\Tr) \cup \Ii(\Pp)$ & \pageref{other_sets_of_trees} \\
 $\Tt$ & Trees on left or right hand side of $\phi$ equation, $\Tt = \Tl \cup \Tr$ & \pageref{other_sets_of_trees} \\
 $\tNn$ & $\{\tau \in \mNn: |\tau| > 1\}$ & \pageref{more_trees_defined} \\
 $\Tp$ & $\Tt \cup \Trec$ & \pageref{more_trees_defined} \\
 $\Trec$ & Planted trees that only appear for centering & \pageref{eq:trec}\\ 
 $\Ii$ & Edge of a tree corresponding to heat kernel & \pageref{s:PE}\\
 $\Ii_+^{(i)}$ & Edge for derivative of heat kernel for positive renormalisation & \pageref{ss:DE}\\
 $\Ii_-^{(i)}$ & Edge for derivative of heat kernel  for negative renormalisation& \pageref{ss:ADE}\\
 $[\bullet]_{\alpha}$ & \hol seminorm of index $\alpha$ & \pageref{e:def-hol}\\ 
 $[\path;\bullet]$ & Seminorm for the local product $\locprod$ applied to a tree & \pageref{def:seminorm}\\
 $\|\bullet\|$ & $L^\infty$ norm &\pageref{paraball}\\
 $\locprod_\bullet$ & Local product &\pageref{eq: admissibility extension}\\
 $\locprod^{\rec}_\bullet$ & Centering map& \pageref{s:positive_renorm}\\
 $\path_{\bullet,\bullet}$ & Local path &\pageref{s:positive_renorm}\\
 $\rho$ & Cut-off function used to define local product &\pageref{eq: cut-off heat equation}\\
 $\leqslant,\subset$ & Relations on trees &\pageref{ss:relations}\\
 $\num(\tau)$ & Number of leaves in a tree $\tau$&\pageref{formula_hom}\\
 $\numnoise$ & Number of noise leaves in a tree $\tau$ &\pageref{formula_hom}\\
 $\numpoly$ & Number of $\{\X_{i}\}_{i=1}^{d}$ leaves in a tree $\tau$ &\pageref{formula_hom}\\
  $\numone$ & Number of $\one$ leaves in a tree $\tau$ &\pageref{formula_hom}\\
  $|\bullet|$ & Order of a tree &\pageref{formula_hom}\\
  $\delta$ & Noise is regularity $C^{-3-\delta}$, $\delta > 0$ &\pageref{phi4}\\
 $D$ & $(0,1)\times \{|x|<1\}$ &\pageref{parabolic_cylinder}\\
 $D_R$ & $(R^2,1) \times \{|x|<1-R\}$ &\pageref{parabolic_cylinder}\\
$\Delta$ & Coproduct &\pageref{sec:coproduct}\\
$C_+$ & Cut map for coproduct &\pageref{table:co-prod1}\\
$C_-$ & Cut map for modified coproduct &\pageref{table:co-prod2}\\
$\mathcal{R}$ & Renormalisation operator &\pageref{eq: inductive def of renormalised model}\\
$\circ_\path$ & Renormalised product of tree expansion &\pageref{def: renormalised product}\\
$(\bullet)_L$ & Convolution with the kernel $\Psi_L$ & \pageref{shauder ou1}\\
$\Psi_L$ & Smooth compactly supported kernel, rescaled at length $L$ & \pageref{shauder ou1}\\
$\Upsilon$ & Coefficient map for solutions to equation & \pageref{def-Upsilon-rec}\\
$V_\gamma$ & Expansion of the remainder solution to level $\gamma$ & \pageref{def:V}\\
$V^2_\gamma$ & Expansion of square of remainder solution to level $\gamma$ & \pageref{def:V2}\\
$V^{((i)}_\gamma$ & Expansion of derivative of the remainder solution to level $\gamma$ & \pageref{def:V'}\\
$U^{\tau}_\gamma$ & Expansion of the local approximation on level $\tau$ &\pageref{def:U}
\end{longtable}
\end{center}
\bibliographystyle{abbrv}
\bibliography{phip}
\end{document}